\documentclass[11pt, preprint]{imsart}
\RequirePackage[OT1]{fontenc}
\RequirePackage{amsthm,amsmath}
\RequirePackage[colorlinks,citecolor=blue,urlcolor=blue]{hyperref}
\usepackage{graphicx}
\usepackage{enumerate}
\usepackage{latexsym}
\usepackage{amssymb}
\usepackage{multirow}
\usepackage{float}
\usepackage{bm}
\usepackage{soul}
\usepackage{paralist}
\usepackage{mathabx}
\usepackage{natbib}
\usepackage{commath}
\usepackage{subcaption}
\usepackage{soul,xcolor}
\usepackage{algorithm,algorithmic}
\usepackage[inline]{enumitem}
\usepackage[figuresright]{rotating}
\usepackage{pifont}
\usepackage{threeparttable}
{\end{quotation}}
\usepackage[normalem]{ulem}

\DeclareMathOperator*{\argmin}{argmin}

\numberwithin{equation}{section}

\def \mcF{\mathcal{F}}

\def \mcH{\mathcal{H}}

\def \mcCS{\mathcal{C}\mathcal{S}}
\def \mcG{\mathcal{G}}
\def \mcD{\mathcal{D}}
\def \mcRn{\mathcal{R}_n}

\def \mcN{\mathcal{N}}
\def \mcP{\mathcal{P}}

\def \bfgamma{\bm{\gamma}}

\setstcolor{red}

\def \bfx{\mathbf{x}}
\def \bfD{\mathbf{D}}
\def \bfg{\mathbf{g}}
\def \bfX{\mathbf{X}}
\def \bfx{\mathbf{x}}
\def \bfZ{\mathbf{Z}}
\def \bfz{\mathbf{z}}
\def \bfd{\mathbf{d}}
\def \bft{\mathbf{t}}
\def \bfh{\mathbf{h}}
\def \bfp{\mathbf{p}}
\def \bfV{\mathbf{V}}
\def \bfa{\mathbf{a}}
\def \bftheta{\bm{\theta}}
\def \bfb{\mathbf{b}}
\def \bfy{\mathbf{y}}
\def \bfD{\mathbf{D}}
\def \bfB{\mathbf{B}}

\def \bff{\mathbf{f}}

\def \bfa{\mathbf{a}}
\def \bfsig{\bm{\sigma}}
\def \bfb{\mathbf{b}}

\def \ev{\mathbb{E}}
\def \pr{\mathbb{P}}
\def \bfW{\mathbf{W}}
\def \bff{\mathbf{f}}

\def \bfA{\mathbf{A}}

\def \bft{\mathbf{t}}
\def \bfv{\mathbf{v}}

\def \bfR{\mathbf{R}}
\def \cid{\xrightarrow[\text{}]{\text{$\mathbb{L}$}}}
\def \cip{\xrightarrow[\text{}]{\text{$\mathbb{P}$}}}
\newcommand{\vo}{\vec{o}\@ifnextchar{^}{\,}{}}

\newcommand{\floor}[1]{\left\lfloor #1 \right\rfloor}
\startlocaldefs

\theoremstyle{plain}

\newtheorem{theorem}{\indent Theorem}
\newtheorem*{theorem*}{\indent Theorem}

\newtheorem{Assumption}{Assumption}

\newtheorem{Definition}{Definition}

\newtheorem{lemma}{\indent Lemma}

\theoremstyle{definition}
\newtheorem{Example}{Example}
\DeclareSymbolFont{largesymbolsA}{U}{txexa}{m}{n}
\DeclareMathSymbol{\varprod}{\mathop}{largesymbolsA}{16}
\usepackage{xparse}
\NewDocumentCommand{\ceil}{s O{} m}{%
    {#2\lceil#3#2\rceil} 
}
\DeclareFontFamily{U}{mathx}{\hyphenchar\font45}
\DeclareFontShape{U}{mathx}{m}{n}{
      <5> <6> <7> <8> <9> <10>
      <10.95> <12> <14.4> <17.28> <20.74> <24.88>
      mathx10
      }{}
\DeclareSymbolFont{mathx}{U}{mathx}{m}{n}
\DeclareMathSymbol{\bigtimes}{1}{mathx}{"91}

\makeatletter\def\theequation{\arabic{section}.\arabic{equation}}
\@addtoreset{equation}{section}\makeatother

\endlocaldefs
\begin{document}
\setstcolor{red}
\title{On Deep Instrumental Variables Estimate}

\runtitle{Deep IV}
	\begin{aug}
		\author{\fnms{Ruiqi} \snm{Liu}\thanksref{m1}\ead[label=e4]{liuruiq@iu.edu}},
		\author{\fnms{Zuofeng} \snm{Shang}\thanksref{k1}\ead[label=e1]{zuofeng.shang@njit.edu}},
		\and
		\author{\fnms{Guang} \snm{Cheng}\thanksref{k2}
			\ead[label=e3]{chengg@purdue.edu}}
        \runauthor{Liu et al.}
         \thankstext{m1}{Department of Mathematical Sciences, Indiana University - Purdue University Indianapolis, IN 46202, USA.}
          \thankstext{k1}{Department of Mathematical Sciences, New Jersey Institute of Technology, NJ 07102, USA.}
          \thankstext{k2}{Department of Statistics, Purdue University, IN 47907, USA.}
	\end{aug}
%
%
%
%
%
%

\begin{center}
\textit{Working Paper Version}  \today
\end{center}
\maketitle

 \begin{center}
\textbf{Abstract}
\end{center}
The endogeneity issue is fundamentally important as many empirical applications may suffer from the omission of explanatory variables, measurement error, or simultaneous causality. Recently, \cite{hllt17} propose a ``Deep Instrumental Variable (IV)" framework based on deep neural networks to address endogeneity, demonstrating superior performances than existing approaches. The aim of this paper is to theoretically understand the empirical success of the Deep IV. Specifically, we consider a two-stage estimator using deep neural networks in the linear instrumental variables model. By imposing a latent structural assumption on the reduced form equation between endogenous variables and instrumental variables, the first-stage estimator can automatically capture this latent structure and converge to the optimal instruments at the minimax optimal rate, which is free of the dimension of instrumental variables and thus mitigates the curse of dimensionality. Additionally, in comparison with classical methods, due to the faster convergence rate of the first-stage estimator, the second-stage estimator has {a smaller (second order) estimation error} and requires a weaker condition on the smoothness of the optimal instruments. Given that the depth and width of the employed deep neural network are well chosen, we further show that the second-stage estimator achieves the semiparametric efficiency bound. Simulation studies on synthetic data and application to automobile market data confirm our theory.

\begin{center}
\textbf{\newpage}
\end{center} 
\section{Introduction}
Endogeneity is a common issue in empirical studies and naturally arises from simultaneous causality, omitted variables, or measurement errors (\citealp{tbr08,a85,y96,m03}).
In the presence of endogeneity, the ordinary least squares (OLS) estimator is known to be inconsistent. One signature tool in addressing the endogeneity issue is to use the so-called two-stage least squares (2SLS) procedure by introducing instrumental variables (IV), as widely used in the literature (\citealp{ak91,ss97,pc06,ll12}). However, the 2SLS estimator is generally inefficient if the reduced form equation between instrumental variables and endogenous variables is not linear. Therefore, as pointed out by \cite{a74} and \cite{n90}, to obtain an efficient estimator, one needs to find the optimal IVs, which involves estimating a set of unknown functions. 

Various non/semiparametric approaches have been proposed to estimate the optimal IVs with guaranteed efficiency (\citealp{a74,n90,npv99,np03}); nonetheless, they could suffer from the curse of dimensionality in the presence of {\em many} IVs. To overcome this difficulty, \cite{bcch12} assume that the optimal IVs can be approximated by a series of basis and propose a lasso-like algorithm to estimate the optimal instruments. Following \cite{bcch12}, \cite{fz18} impose an additive structure assumption on the optimal IVs, which yields dimension-free results, whereas their estimator could be inefficient when the additive structure fails to hold. This motivates \cite{hllt17} to propose a very flexible deep learning framework, named as Deep IV, under which impressive empirical performances are demonstrated even when IVs are high-dimensional and the optimal IVs are of complex structures. Based on the same framework, \cite{bennett2019deep} further propose a Deep Generalized Method of Moments (Deep GMM) for the IV analysis, {while \cite{flm18} use a two-stage estimator based on neural networks to conduct inferences on the treatment effect.} As discussed in \cite{hllt17}, in comparison with Deep IV, the classical series or kernel estimation (e.g., \citealp{np03, bck07, cp12, bcch12}) is computationally intractable in high-dimensional feature spaces. However, theoretical understandings on the benefits of the use of deep neural networks in the IV analysis are still missing. 

{The present work aims to theoretically explain the empirical success of  the Deep IV framework.} For simplicity of presentation, we mainly consider the linear regression model with endogenous predictors and observable IVs. In the first stage, using the IVs as the regressors and endogenous variables as the responses, the optimal IVs are estimated by a fully connected \textit{rectifier linear unit} (ReLU) deep neural network (DNN). By imposing a general compositional structure assumption on the optimal IVs, we derive the convergence rate for the proposed neural network estimator, which is free of the dimension of IVs, as either depth, width, or both diverge. In particular, the derived rate is minimax optimal as long as the product of depth and width for the neural network is greater than the number of IVs and is a polynomial order of the sample size. In practice, the implementation of DNN does not explicitly rely on the imposed structure assumption, i.e., latent compositional structure, unlike additive {or linear} regression. As a side remark, the choices of depth and width have different impacts on the numerical optimization in learning the neural network. Specifically, to obtain the optimal convergence rate, a very deep network is more ``economical" in terms of the number of parameters to be learned; on the other hand it faces the challenge of the vanishing gradient issue in comparison with shallow neural networks.

In the second stage, we perform the least squares estimation for the linear coefficients based on the estimated optimal IVs in the first stage. If the product of depth, width, and IV dimension is of a polynomial order of sample size, the estimator is proven to be asymptotically normal and achieve the semiparametric efficiency bound (\citealp{bickel1993efficient}) as either depth, width, or both diverge. Moreover, by taking advantage of the faster convergent neural network estimate, the second-stage estimate not only has {a smaller (second order) estimation error (in the sense of \cite{cheng2008general} to be explained later)} but also requires a weaker condition on the smoothness of the optimal IVs when compared with classical methods. Specifically, using polynomial spline basis, the series approach in \cite{bcch12} requires that the smoothness degree of the optimal IVs should be greater than half of the IV dimension. In contrast, our results hold as long as the {\em intrinsic} smoothness degree of the optimal IVs is greater than half of their {\em intrinsic} dimension, which will be proven to be weaker. {To be more concrete, we present a scenario where our estimator outperforms some widely used estimators in the literature; see Example \ref{example:5}.} 

Several extensions can be made based on the above theoretical results. To be more specific, we propose a sample-split estimator to remove the requirement on the intrinsic smoothness and intrinsic dimension of the optimal IVs. A specification testing procedure is also proposed to test the validity of the instrumental variables based on deep neural network. Finally, we consider an extended model containing both endogenous and exogenous variables.

{Recently, a number of researchers study deep neural networks from nonparametric perspective. To name a few, \cite{bk19} and \cite{s19} use sparse neural network in the nonparametric regression setting, while \cite{kim2018fast} study the problem of classification based on sparse network. Recently,  \cite{mk19} extend the result in \cite{bk19} to very deep (diverging depth) fully connected neural network with fixed width. Our IV estimator is a fully connected neural network with either depth, width, or both being diverging. This more flexible network structure, which covers the network architecture in \cite{mk19}, avoids the selection of sparsity parameter in practice and does not need to impose a truncation parameter to bound the neural network estimation, in contrast with \cite{bk19, s19, mk19}. Our theoretical results are derived by extending the recent neural network approximation theory in \cite{lsyz20} from Sobolev space to H\" older space.}


This paper is organized as follows. Section \ref{sec:model}
reviews mathematical formulation of linear IV model. Section \ref{sec:DIVE} describes the Deep IV estimation procedure. Section \ref{sec:asymptotics} provides the asymptotic results of the proposed estimator and its advantages over the existing approaches. {Section \ref{sec:extra:results} provides some additional inferential procedures based on deep neural network.} Section \ref{sec:simulation} compares the finite-sample performance of our estimator and some completing approaches
through a simulation study. Section \ref{sec:empirical} applies the proposed procedure to a real-world data set to study the connection between automobile sales and price.  All the mathematical proofs are deferred to the Appendix. 

\section{Linear Instrumental Variable Model}\label{sec:model}
Consider i.i.d. observations $\{Y_i, \bfX_i\}_{i=1}^n$ generated from the following linear regression model:
\begin{equation}\label{eq:model:1}
	Y=\beta_0^\top\bfX+\epsilon, 
\end{equation}
where $Y\in \mathbb{R}$ is the response variable, $\bfX=(X_{1},\ldots, X_{q})^\top\in \mathbb{R}^q$ is the vector of explanatory variables, $\epsilon$ is the random
noise, and $\beta_0\in \mathbb{R}^q$ is the vector of unknown coefficients. The explanatory variables $X_{1}, \ldots, X_{q}$ are assumed to be endogenous in the sense that 
\begin{eqnarray*}
 \textrm{Cov}(\epsilon, X_{s})\neq 0,\quad \textrm{ for all } s=1,\ldots, q.
\end{eqnarray*}
{In order to consistently estimate $\beta_0$, one set  of instrumental variables $\bfZ=(Z_{1},\ldots, Z_{d})^\top \in\mathbb{R}^d$ with i.i.d. observations $\{\bfZ_i\}_{i=1}^n$ is introduced as follows (e.g., see \citealp{wooldridge2008introductory}):}
\begin{equation}\label{eq:IV:Z}
	\ev(\epsilon| Z_{j})=0,\quad \textrm{ for all } j=1,\ldots, d.
\end{equation}

In view of (\ref{eq:model:1}) and (\ref{eq:IV:Z}), one set of unconditional moment equations to identify $\beta_0$ is
\begin{equation}\label{eq:unconditional:moment:condition}
	\ev[(Y-\beta_0^\top\bfX)\bfh(\bfZ)]=0,\quad \textrm{ for all $q$-dimensional vector-valued measurable function } \bfh.
\end{equation}
Define a collection of optimal IVs: $$\bfh_{\textrm{opt}}(\bfz):=(f_{0,1}(\bfz),\ldots, f_{0,q}(\bfz))^\top\;\;\mbox{with}\;\;f_{0,s}(\bfz):=\ev(X_{s} | \bfZ=\bfz).$$ By setting $\bfh(\bfz)=\bfh_{\textrm{opt}}(\bfz)$ in (\ref{eq:unconditional:moment:condition}), the corresponding method of moments estimator solved through (\ref{eq:empirical:moment:condition}) can be proven to be semiparametric efficient (see \citealp{a74} and \citealp{n90}):
\begin{equation}\label{eq:empirical:moment:condition}
	\frac{1}{n}\sum_{i=1}^n (Y_i-\beta^\top\bfX_i)\widehat{\bfh}_{\textrm{opt}}(\bfZ_i)=0,
\end{equation}
where $\widehat{\bfh}_{\textrm{opt}}$ is an estimate for ${\bfh}_{\textrm{opt}}$. 

This above idea has been exhaustively studied based on various forms of $\widehat{\bfh}_{\textrm{opt}}$. For example, \cite{n90}  uses both k-nearest-neighbors and series approximation methods to estimate ${\bfh}_{\textrm{opt}}$. {As pointed out by \cite{cheng2008general} in a general semiparametric setting, the convergence rate of  $\widehat{\bfh}_{\textrm{opt}}$ to ${\bfh}_{\textrm{opt}}$ plays an important role in estimating $\beta_0$: the second order error in estimating $\beta_0$ based on (\ref{eq:empirical:moment:condition}) is smaller when $\widehat{\bfh}_{\textrm{opt}}$ converges at a faster rate.} Moreover, in order to ensure semiparametric efficiency for estimating $\beta$, $\widehat{\bfh}_{\textrm{opt}}$ is required to converge to $\bfh_{\textrm{opt}}$ at a sufficiently fast rate. However, when the dimension of $\bfZ$ is large, the convergence rate is often slow due to the curse of dimensionality. To address this issue, \cite{bcch12} impose a sparsity assumption that ${\bfh}_{\textrm{opt}}$ can be approximated by a functional series and then estimated by a lasso-type series estimator. Following \cite{bcch12}, \cite{fz18} assume that each element of $\bfh_{\textrm{opt}}$, namely $f_{0,s}$, follows an additive model of $Z_1,\ldots, Z_d$. However, if $\bfh(\bfz)$ has interaction terms, this method could lead to an inefficient estimator due to model misspecification. Recently, \cite{hllt17} consider the pure nonparametric regression $Y=g(\bfX)+\epsilon$ with $\bfX$ being endogenous. They propose the Deep IV procedure based on two deep nerual networks to estimate the underlying regression function. A followup work of Deep IV is Deep GMM for instrumental variables analysis by \cite{bennett2019deep}. Both Deep IV and Deep GMM demonstrate very impressive empirical performances but with very limited theoretical investigation. The aim of our work is to understand the theoretical benefits of the use of deep neural networks in instrumental variables analysis.

{\bf Notation:} Let $\|\bfv\|_2^2=\bfv^\top\bfv$ denote the Euclidean norm of the vector $\bfv$. Let $\cid$ and $\cip$ denote the convergence in distribution and convergence in probability, respectively. For two sequence $a_n$ and $b_n$, we say $a_n\asymp b_n$ if $c^{-1} a_n\leq b_n\leq ca_n$ for some constant $c>1$ and all  sufficiently large $n$. For each observation,  let $X_{i1},\ldots, X_{iq}$ and $Z_{i1}, \ldots, Z_{id}$ be the elements of $\bfX_i$ and $\bfZ_i$. For any $f:\mathbb{R}^d \to \mathbb{R}$, define the $L^2$-norm $\|f\|^2=\ev(f^2(\bfZ))$ and its
empirical counter part $\|f\|_n^2=n^{-1}\sum_{i=1}^nf^2(\bfZ_i)$. For $a>0$, let $\floor{a}$ denote the largest integer strictly less than $a$ and $\ceil{a}=\floor{a}+1$.  We say $(L, W)\to \infty$, if ether $L$, $W$ or both diverge. 

\section{Deep Instrumental Variables Estimation}\label{sec:DIVE}

In this section, we first review the setup for fully connected neural networks. 
Let $\sigma$ denote the ReLU activation function,
i.e., $\sigma(x):=(x)_+$ for $x\in\mathbb{R}$. For any $r$-dimensional real vectors $\bfv=(v_1,\ldots,v_r)^T$ and $\bfa=(a_1,\ldots,a_r)^T$, define the shift activation function
$\bfsig_\bfv(\bfa)=(\sigma(a_1-v_1),\ldots,\sigma(a_r-v_r))^T$. Moreover, a vector-valued function $\bff: \mathbb{R}^d \to \mathbb{R}^q$ is a fully connected deep neural network with depth $L$ and width $W$, if it has the following expression:
\begin{eqnarray}
	\bff(\bfz)=\bfv_{L+1}+\bfA_{L+1}\bfsig_{\bfv_{L}}\circ \bfA_{L}\bfsig_{\bfv_{L-1}}\circ\ldots \circ \bfA_2\bfsig_{\bfv_1}\circ \bfA_1 \bfz,\quad \quad \textrm{ for }\bfz \in \mathbb{R}^{d},\label{eq:def:dnn}
\end{eqnarray}
where $\bfv_{L+1}\in \mathbb{R}^q$ and $\bfv_{l}\in \mathbb{R}^{W}$ for $l=1,\ldots, L$ are the shift vectors, $\bfA_{1}\in \mathbb{R}^{W\times d}, \bfA_{L+1}\in \mathbb{R}^{q\times W}$ and $\bfA_{l}\in \mathbb{R}^{W\times W}$ for $l=2,\ldots, L$ are the weight matrices. Finally, we denote $\mcF_{d,q}(L, W)$ as the collection of fully connected deep neural networks with depth $L$, width $W$, $d$-dimensional input and $q$-dimensional output.

We propose a two-stage estimation procedure based on the fully connected neural network. The first stage is to construct a DNN estimate as $\widehat{\bfh}_{\textrm{opt}}$: 
\begin{equation}\label{eq:estimator:step:1}
	\widehat{\bff}:=\argmin_{\substack{\bff\in \mcF_{d,q}(L, W)}}\frac{1}{n}\sum_{i=1}^n\|\bfX_i-\bff(\bfZ_i)\|_2^2,
\end{equation}
where the elements of $\widehat{\bff}$ can be written as $(\widehat{f}_1, \ldots, \widehat{f}_q)^\top$. Correspondingly, we define $\widehat{\bfX}_i:=\widehat{\bff}(\bfZ_i)=(\widehat{f}_1(\bfZ_i),\ldots, \widehat{f}_q(\bfZ_i))^\top\in \mathbb{R}^q$ for $i=1,\ldots, n$. As will be shown in Theorem~\ref{theorem:rate:deep}, the DNN estimation procedure is able to capture the intrinsic structure of ${\bfh}_{\textrm{opt}}$ without explicitly using the prior information of its compositional structure (to be specified later). The second stage is to construct an estimator of $\beta_0$ in an OLS manner:
\begin{equation}\label{eq:estimator:step:2}
	\widehat{\beta}=\bigg(\frac{1}{n}\sum_{i=1}^n\widehat{\bfX}_i\bfX_i^\top\bigg)^{-1}\frac{1}{n}\sum_{i=1}^n\widehat{\bfX}_i Y_i.
\end{equation}
Intuitively, if the DNN estimators $\widehat{f}_1, \ldots, \widehat{f}_q$ are close to the ground truth $f_{0,1},\ldots, f_{0, q}$ enough, the second-stage estimator will be also close to the ``oracle''  estimator obtained by using the ground truth in (\ref{eq:estimator:step:2}), which is known to achieve semiparametric efficiency. 

It is worth mentioning that the optimization problem in (\ref{eq:estimator:step:1}) is unconstrained, and it is usually solved by the Stochastic Gradient Descent (SGD) algorithm or its variants. However, the neural network estimators proposed by \cite{bk19} and \cite{mk19}  are truncated by a threshold parameter, while \cite{s19} needs to solve a constrained optimization problem requiring the estimator is bounded by some predetermined constant. Therefore, our estimator is practically convenient and avoids the issue of choosing all inds of hyper-parameters. Furthermore, the neural networks considered in \cite{bk19} and \cite{s19} are both sparse; namely, some of the weights should be zero. In practice, how to determine the sparsity is difficult. Recently, \cite{mk19} extend their earlier work \cite{bk19} to fully connected networks, but require $W$ to be fixed. Rather, our estimator in (\ref{eq:estimator:step:1}) allows either $L$, $W$, or both to diverge, which is more practically flexible. 

\section{Asymptotic Theory}\label{sec:asymptotics}
In this section, we develop rate of convergence for $\widehat{\bff}$ and asymptotic distribution for $\widehat{\beta}$.
\subsection{Rate of Convergence}
We begin with the definitions of H\" older smooth function and a class of multivariate functions with a compositional structure.
\begin{Definition}\label{Definition:p:M:smooth}
A function $g: \mathbb{R}^d \to \mathbb{R}$ is said to be $(p, C)$-H\" older smooth for some positive constants $p$ and $C$, if for every $\bm{\gamma}=(\gamma_1, \ldots, \gamma_d)\in \mathbb{N}^d$ the following two conditions hold:
\begin{eqnarray}
	\sup_{\bfz\in \mathbb{R}^d}\bigg|\frac{\partial^{|\bm{\gamma}|}g}{\partial z_1^{\gamma_1}\ldots \partial z_1^{\gamma_d}}(\bfz)\bigg|\leq C, \quad	 \textrm{ for all  } |\bm{\gamma}|\leq \floor{p},\nonumber
\end{eqnarray}
and
\begin{eqnarray}
	\bigg|\frac{\partial^{|\bm{\gamma}|}g}{\partial z_1^{\gamma_1}\ldots \partial z_1^{\gamma_d}}(\bfz)-\frac{\partial^{|\bm{\gamma}|}g}{\partial z_1^{\gamma_1}\ldots \partial z_1^{\gamma_d}}(\widetilde{\bfz})\bigg|\leq C\|\bfz-\widetilde{\bfz}\|_2^{p-\floor{p}}, \quad	 \textrm{ for all  } |\bm{\gamma}|=\floor{p} \textrm{ and } \bfz,\widetilde{\bfz}\in \mathbb{R}^d.\nonumber
\end{eqnarray}
Here  $|\bfgamma|=\sum_{i=1}^d\gamma_d$.  For convenience, we say $g$ is $(\infty, C)$-H\" older smooth convenience if $g$ is $(p, C)$-H\" older smooth for all $p>0$.  
\end{Definition}

\begin{Definition}\label{Definition:Convolution:f0}
A function $f: \mathbb{R}^d \to \mathbb{R}$ is said to have a compositional structure with parameters $(L_*, \bfd, \bft, \bfp, \bfa, \bfb, C)$ for $L_*\in \mathbb{Z}_+$, $\bfd=(d_0, \ldots, d_{L_*+1})\in \mathbb{Z}_+^{L_*+2}$ with $d_0=d, d_{L_*+1}=1$, $\bft=(t_0,\ldots, t_{L_*})\in \mathbb{Z}_+^{L_*+1}$, $\bfp=(p_0,\ldots, p_{L_*})\in \mathbb{R}_+^{L_*+1}$, $\bfa=(a_0,\ldots, a_{L_*+1}), \bfb=(b_0,\ldots, b_{L_*+1})\in \mathbb{R}^{L_*+2}$ and $C\in \mathbb{R}_+$, if 
\begin{equation*}
	f(\bfz)=\bfg_{L_*}\circ\ldots\circ \bfg_1 \circ \bfg_0(\bfz),\quad\quad \textrm{ for all } \bfz\in [a_0, b_0]^d
\end{equation*}
where $\bfg_i=(g_{i,1},\ldots, g_{i,d_{i+1}})^\top: [a_i, b_i]^{d_i}\to [a_{i+1}, b_{i+1}]^{d_{i+1}}$ for some $|a_i|, |b_i|\leq C$ and the functions $g_{i,j}: [a_i, b_i]^{t_i} \to [a_{i+1}, b_{i+1}]$ are $(p_i, C)$-H\" older smooth only relying on $t_i$ variables. We denote $\mcCS(L_*, \bfd, \bft, \bfp,\bfa, \bfb, C)$ as the class of functions defined above
\end{Definition}
Definition \ref{Definition:p:M:smooth} characterizes the smoothness of the regression function which is commonly used in the nonparametric literature (see \citealp{gkkw06,h03}). 
Definition \ref{Definition:Convolution:f0} requires that the function $f$ is a composition of $L_*+1$ {layers} with each {layer} being a vector-valued multivariate function which demonstrates a local connectivity structure.  
Such a compositional structure, also adopted by \cite{bk19}, \cite{s19} and \cite{mk19}, 
is naturally motivated from the structure of the neural network. The functions $g_{i,j}$'s can be viewed as  hidden features of the target function $f$, which makes up more complex features $g_{i+1,j}$'s in the next layer. One essence of deep neural network is to learn these hidden features from the data (\citealp{zf14}).

It is worthwhile to discuss the connection between Definitions \ref{Definition:p:M:smooth} and \ref{Definition:Convolution:f0}. For this purpose, we define the following two important quantities:
\begin{eqnarray}
	p^*=p_{i^*}^*\quad \textrm{ and }\quad t^*=t_{i^*},\nonumber
\end{eqnarray}
where $p_i^*=p_i\prod_{s=i+1}^{L_*}(p_s\wedge 1)$ for $i=0,\ldots, L_*$,
and $i^*=\argmin_{0\leq i \leq L_*}p_i^*/t_i$.
We will adopt the convention $\prod_{s=L_*+1}^{L_*}(p_s\wedge 1)=1$ for convenience. Similar to \cite{bk19} and \cite{s19}, $p^*$ and $t^*$ can be interpreted as the intrinsic smoothness and intrinsic dimension of a function satisfying Definition \ref{Definition:Convolution:f0}, and $t^*$ tends to be  smaller than the input dimension $d$ in several important models, as seen from examples below. 

It is not difficult to verify that a $(p, C)$-H\" older smooth function has a trivial compositional structure with $L_*=0$. On the other hand, the following lemma indicates that the functions with a compositional structure are also H\" older smooth.
\begin{lemma}\label{lemma:pc:smooth:degree}
\begin{enumerate}[label={(\roman*}),ref={(\roman*})]
\item \label{lemma:pc:smooth:degree:result:1} Suppose $g_1: \mathbb{R}^{q} \to \mathbb{R}$ is $(p_1, C)$-H\" older smooth and $\bfg_2=(g_{21},\ldots, g_{2q}): \mathbb{R}^{d} \to \mathbb{R}^q$ with  $g_{2i}$'s are all $(p_2, C)$-H\" older smooth, then degree of H\" older smoothness of $g_1\circ \bfg_2$ is $\min\{p_1p_2, p_1, p_2\}$.
\item  \label{lemma:pc:smooth:degree:result:2}  If $f\in \mcCS(L_*, \bfd, \bft, \bfp,\bfa, \bfb, C)$, then the degree of H\" older smoothness of $f$ is $p_H\leq p^*$. 
\end{enumerate}
\end{lemma}
The conclusion \ref{lemma:pc:smooth:degree:result:1} in Lemma \ref{lemma:pc:smooth:degree} was obtained by \cite{jlt09}. The conclusion \ref{lemma:pc:smooth:degree:result:2} of Lemma \ref{lemma:pc:smooth:degree} suggests that the  H\" older smoothness of a function with a compositional structure is smaller than its intrinsic smoothness. {An interesting implication is that if one ignores the compositional structure, the smoothness could be underestimated in the sense that it is not larger than the intrinsic dimension.}

The compositional structure specified in  \ref{Definition:Convolution:f0} covers many important models in statistics and economics, as demonstrated in the following examples. 
\begin{Example}\label{example:4} (Classical Nonparametric Regression) In classic nonparametric regression,
it is often assumed that the regression function $f(z_1,\ldots, z_d)$ is $(p, C)$-H\" older smooth
(see \citealp{h98, h03, c07}). Therefore, $f$ has a compositional structure with $L_*=0$, $\bfd=(d,1)$, $\bft=d$ and $\bfp=p$. Consequently, $p^*=p$, $t^*=d$ and $p_H=p$.
\end{Example}

\begin{Example}\label{example:2} (Generalized Additive Model) The generalized additive model assumes that the condition mean 
of the response given a set of predictors $z_1,\ldots,z_d$ has an expression $f(z_1, \ldots, z_d)=g(\sum_{j=1}^dh_j(z_j))$, where $g$ is $(p_g, C)$-H\" older smooth and $h_j$'s are $(p_h, C)$-H\" older smooth. It can be shown that $f$ has a compositional structure with $L_*=2$, $\bfd=(d, d, 1, 1)$, $\bft=(1, d, 1)$ and $\bfp=(p_h, \infty, p_g)$. Therefore,  $p^*=\min(p_g, g_h)$,  $t^*=1$ and $p_H=\min\{p_gp_h, p_g, p_h\}$.
\end{Example}

\begin{Example}\label{example:1} (Production Function with $d$ Inputs) In economic studies, the production function is often assumed to be of the form $f(z_1, \ldots, z_d)=A\prod_{j=1}^d z_i^{\lambda_i}$, in which $z_j$'s represent the quantities of production factors, the constant $A$ represents the factor productivity, and $\lambda_j$'s represent elasticities.
This is a generalization of the classical Cobb-Douglas production function (see \citealp{n65}). Thus, $f$ has a compositional structure with $L_*=1, \bfd=(d, d, 1), \bft=(1, d)$ and $\bfp=(\infty, \infty)$. If the domain of $\bfz$ is compact and does not contain zero, it can be shown that $p^*=\infty$, $t^*=1$ and $p_H=\infty$. 
\end{Example}
To establish the asymptotic theory, we need the following technical conditions. 
\begin{Assumption}\label{Assumption:AC}
\begin{enumerate}[label={(\roman*}),ref={(\roman*})]
\item \label{Assumption:A1}$\ev(e^{\kappa_1|X_s|})\leq \kappa_2$, for some $\kappa_1, \kappa_2>0$ and all $s=1,\ldots, q$.
\item \label{A1:b} 
The underlying  functions $f_{0,1},\ldots, f_{0,q}\in \mcCS(L_*, \bfd, \bft, \bfp, \bfa, \bfb, C)$. We assume that $d_0=d$ is allowed to diverge with $n$, while all the rest parameters are fixed constants.
\item \label{Assumption:A2:deep}The network structure satisfies the conditions that
\begin{eqnarray}
	(L, W)\to \infty, \quad L\geq L_*+\sum_{i=0}^{L_*}L_i\;\; \textrm{ and }\;\; W\geq \max_{0\leq i \leq L_*}qW_id_{i+1}\nonumber
\end{eqnarray} 
with $L_i= 216\ceil{p_i}^{2}+1$ and $W_i=81(\ceil{p_i}+t_i+2)^{t_i+1}3^{t_i+1}$.
\end{enumerate}
\end{Assumption}

Assumption \ref{Assumption:AC}\ref{Assumption:A1} requires exponential tail of $X_s$, which is a key assumption to obtain the optimal (up to a logarithm factor) convergence rate of $\widehat{f}_s$. This assumption is weaker than the sub-Gaussian condition proposed by \cite{bk19} and \cite{s19}. Moreover, Assumption \ref{Assumption:AC}\ref{Assumption:A1} can be relaxed to some moment conditions of $X_s$, but this relaxation sacrifices a polynomial rate in convergence. Assumption \ref{Assumption:AC}\ref{A1:b} imposes a compositional structure on $f_{0,s}$, which can be viewed as a neural network version of sparsity condition. In a similar spirit, \cite{bcch12} assume that $f_{0,s}$ has a sparse basis expression. Assumption \ref{Assumption:AC}\ref{Assumption:A2:deep} is to specify what kind of neural networks can accurately approximate the functions in $\mcCS(L_*, \bfd, \bft, \bfp,\bfa, \bfb, C)$, and the lower bounds of depth and width are relying on the parameters $p_i$'s and $t_i$'s, {which are assumed to be fixed}. Therefore, it is worth mentioning that there are three types of diverging behaviours of $(L, W)$ {based on Assumption \ref{Assumption:AC}\ref{Assumption:A2:deep}}, namely (a) diverging $L$ and fixed $W$; (b) fixed $L$ and diverging $W$; (c) diverging $L$ and $W$. Different choices of $(L, W)$ correspond to different network architectures. In particular, the first type network corresponds to the so-called  \textit{fixed-width DNN} (\citealp{mk19}), while the second one is called  \textit{fixed-depth DNN} (\citealp{bk19}). Our general theory covers all three types (a)-(c).

The following theorem states the convergence rate of $\widehat{f}_s$.
\begin{theorem}\label{theorem:rate:deep}
Under Assumption \ref{Assumption:AC}, if $LW=o(\sqrt{n})$ and $LWd=o(n)$, then for all $s=1,\ldots, q$, it follows that
\begin{equation*}
	\|\widehat{f}_s-f_{0,s}\|_n=O_P\bigg(\log^5(n)\sqrt{\frac{L^2W^2+LWd}{n}}+\log^{\frac{4p^*}{t^*}}(n)(LW)^{-\frac{2p^*}{t^*}}\bigg).
\end{equation*}
As a consequence, if $d=O(LW)$ and $LW\asymp n^{\frac{t^*}{2(2p^*+t^*)}}$, then
\begin{equation*}
	\|\widehat{f}_s-f_{0,s}\|_n=O_P\bigg([\log(n)]^{5+\frac{4p^*}{t^*}}n^{-\frac{p^*}{2p^*+t^*}}\bigg),\quad s=1,\ldots, q.
\end{equation*}
\end{theorem}
Theorem \ref{theorem:rate:deep} provides a convergence rate for $\widehat{f}_s$ under the norm $\|\cdot\|_n$ in terms of $L,W,d$. This rate consists of two parts: the first part $\sqrt{(L^2W^2+LWd)n^{-1}}$ corresponds to the estimation error which relies on the entropy of $\mcF_{d, q}(L, W)$; the second part $(LW)^{-2p^*/t^*}$ corresponds to the approximation error of $\mcF_{d,q}(L, W)$ to $\bff_0:=(f_{0,1},\ldots, f_{0,1})^\top$. Note that the approximation error only depends on $p^*$ and $t^*$, which is free of the input dimension $d$, and it also implies that to approximate less smooth function with higher intrinsic dimension, the neural network needs to be more complicated, namely, with large depth or width. 

The proof of Theorem \ref{theorem:rate:deep} relies on the recent results in \cite{lsyz20} that approximate functions in Sobolev space by the fully connected neural network with general choices of $L$ and $W$. We extend their results to approximate the H\" older smooth functions and the functions with a compositional structure. By a suitable choice of $L$ and $W$, the DNN estimators can achieve the convergence rate $n^{-p^*/(2p^*+t^*)}$ (up to a logarithm factor), which is minimax optimal according to \cite{s19}, as long as $d$ does not grow faster than $LW$. Recently, \cite{bk19} and \cite{s19} obtain a similar convergence rate by considering a {\em sparse} network for a {\em fixed} $d$. More recently, \cite{mk19} extend the result of \cite{bk19} to fully connected neural networks but still with a fixed $W$ based on a new approximation theorem. Note that their result is a special case of our Theorem \ref{theorem:rate:deep}. It is worth mentioning that, using series or kernel based methods, the optimal convergence rate is $n^{-\frac{d}{2p_H+d}}$ when estimating a H\" older smooth function (e.g. see \citealp{s94,s19}). Since $p_H\leq p^*$ and $d \geq t^*$, the neural network estimator has a faster convergence rate by capturing the intrinsic smoothness and dimension.

In the end, we would like to stress that different network structures will have different consequences on the optimization. From a theoretical point of view, for $\bff \in \mcF_{d,q}(L, W)$, there are $W(L+d+q)+(L-1)W^2$ parameters to be estimated, including all the shift vectors and weight matrices. If $d$ is fixed, the number of parameters is of an order of $LW^2$. In particular, to achieve the optimal convergence rate, it requires to estimate about $n^{\frac{t^*}{2(2p^*+t^*)}}$ parameters for fixed-width DNN and about $n^{\frac{t^*}{2p^*+t^*}}$ parameters for fixed-depth DNN. Therefore, increasing depth is more ``economical" to obtain the optimal convergence rate in terms of the number of parameters, which is due to the fact that depth is more effective than width for the expressiveness of
ReLU networks (e.g., see \citealp{lpwhw17,yz19,lsyz20}). On the other hand, training very deep neural network is numerically more challenging due to the vanishing gradient issue (see \citealp{sgs15}), and thus fixed-depth DNN or neural networks with less depth are also of practical importance.

\subsection{Asymptotic Distribution}
In this section, we will show that the second stage estimator $\widehat{\beta}$ is asymptotically normal and moreover, achieves the semiparametric efficiency bound \citep{n90}. Let $\bfD:=(f_{0,1}(\bfZ),\ldots, f_{0,q}(\bfZ))^\top \in \mathbb{R}^q$ and assume the following regularity conditions. 
\begin{Assumption}\label{Assumption:A3}
\begin{enumerate}[label={(\roman*}),ref={(\roman*})]
\item \label{A3:0} $\ev(\epsilon|Z_j)=0$ for $j=1,\ldots, d$.
\item \label{A3:a} $\ev(e^{\kappa_3|\epsilon|})\leq \kappa_4$ for some $\kappa_3, \kappa_4>0$ and $\ev(\epsilon^2|\bfZ)=\sigma_\epsilon^2$.
\item \label{A3:b} The matrix $\ev(\bfD\bfD^\top)$ is positive definite.
\end{enumerate}
\end{Assumption}

Assumptions \ref{Assumption:A3}\ref{A3:0} and \ref{Assumption:A3}\ref{A3:a} are both standard in the IV literature. Assumption \ref{Assumption:A3}\ref{A3:b}, called as the strong-instrument condition in \cite{bcch12}, guarantees the invertibility of $\ev(\bfD\bfD^\top)$.

\begin{theorem}\label{theorem:asymptotic:distribution:deep}
Under Assumptions \ref{Assumption:AC} and \ref{Assumption:A3},  if $L^2W^2\log^8(n)=o(n^{1/2})$, $LWd\log^8(n)=o(n^{1/2})$ and $(LW)^{1-2p^*/t^*}[\log(n)]^{3+4p^*/t^*}=o(1)$, then  it follows that
\begin{equation*}
	\sqrt{n}(\widehat{\beta}-\beta_0)\cid \textrm{N}(0, \sigma_\epsilon^2 \ev^{-1}(\bfD\bfD^\top)).
\end{equation*}
\end{theorem}
Theorem \ref{theorem:asymptotic:distribution:deep} establishes the asymptotic distribution of $\widehat{\beta}$ with general choices of depth and width for the neural network. We allow the number of explanatory variables $d$ to be possibly diverging. In particular, when $d=o(LW)$, the rate conditions of Theorem \ref{theorem:asymptotic:distribution:deep} can be further
simplified as $L^2W^2\log^8(n)=o(n^{1/2})$ and $(LW)^{1-2p^*/t^*}[\log(n)]^{3+4p^*/t^*}=o(1)$, which specify upper and lower bounds for $LW$. It is also worthwhile to mention that, since $LW$ is diverging, to satisfy the rate condition $(LW)^{1-2p^*/t^*}[\log(n)]^{3+4p^*/t^*}=o(1)$, one needs $p^*>t^*/2$. In other words,  to apply Theorem \ref{theorem:asymptotic:distribution:deep}, a sufficient condition is that the underlying functions $f_{0,s}$'s need to be smooth enough in the sense that the degree of intrinsic smoothness should be larger than half of the intrinsic dimension. In addition, this condition is weaker than that in literature, e.g., \cite{bcch12}, since in general $p_H\leq p^*$ and $d\geq t^*$. For the same reason, our neural network estimate has a faster convergence rate than 
series or kernel estimators; see Theorem~\ref{theorem:rate:deep}. An implication is that, according to Section 3 in \cite{cheng2008general} {and the discussion in Section \ref{sec:model}}, $\sqrt{n}(\widehat{\beta}-\beta_0)$ converges to its Gaussian limit at a faster rate than the resulting second-stage estimators from series or kernel estimators.

We next discuss how to consistently estimate the unknown asymptotic covariance matrix of $\widehat{\beta}$.

\begin{lemma}\label{lemma:estimation:variance}
Under conditions of Theorem \ref{theorem:asymptotic:distribution:deep}, it holds that $$\widehat{\bfV}^2\cip \sigma_\epsilon^2\ev^{-1}(\bfD\bfD^\top),$$ where $\widehat{\bfV}^2= (\sum_{i=1}^n\widehat{\bfX}_i\bfX_i^\top) ^{-1}\sum_{i=1}^n\widehat{\epsilon}_i^2$ with $\widehat{\epsilon}_i=Y_i-\widehat{\beta}^\top \bfX_i$.
\end{lemma}
Combining Theorem \ref{theorem:asymptotic:distribution:deep} and Lemma \ref{lemma:estimation:variance}, we can construct a $100\times(1-\alpha)\%$ confidence interval for $\beta_0$ as follows:
\begin{equation}\label{eq:confidence:interval}
	\widehat{\beta}\pm \frac{z_{\alpha/2}^*}{\sqrt{n}}\textrm{Diag}(\widehat{\bfV}),
\end{equation}
where $z_{\alpha/2}^*$ is the $\alpha/2$ upper percentile of the standard normal distribution and $\textrm{Diag}(A)$ is the vector consisting of the diagonal elements of squared matrix $A$.

\subsection{Theoretical Benefits}\label{sec:boosting:efficiency}

In this section, we highlight the theoretical advantages of $\widehat{\beta}$ when compared with the existing IV estimators. In summary, due to the faster convergence rate and the ability to capture the intrinsic structure of the first-stage estimator, the second-stage estimator can achieve the efficiency bound with a smaller (second order) estimation error under a weaker smoothness condition.

For illustration, we assume $q=1$ so that $\bfX$ is a scalar. Let $\mcF$ denote some function class and $f_\mcF\in\mcF$ denote the projection of $\bfX$ onto $\mcF$:
\begin{eqnarray}\label{eq:population:minimizer}
	f_\mcF=\argmin_{f\in \mcF}\ev[(\bfX-f(\bfZ))^2].\nonumber
\end{eqnarray}
The form of $f_\mcF$ relies on the choice of $\mcF$; see the following examples.
\begin{enumerate}
\item\label{eq:mcFC} $\mcF_{\textrm{C}}=\{f: \mathbb{R}^d\to \mathbb{R}\;: f\in \mcCS(L_*, \bfd, \bft, \bfp,\bfa, \bfb, C) \textrm{ for parameters } L_*, \bfd, \bft, \bfp,\bfa, \bfb, C.\}$, the class of  functions with a compositional structure.
By Theorem \ref{theorem:rate:deep}, neural network estimator proposed in (\ref{eq:estimator:step:1}) is a consistent estimator of $f_{\mcF_\textrm{C}}$. 
\item \label{eq:mcFS}  $\mcF_{\textrm{S}}=\{f: \mathbb{R}^d\to \mathbb{R}\;: f \textrm{ is } (p, C)\textrm{-H\" older smooth for parameters } p, C\}$,  the class of H\"older smooth functions. By \cite{s94} and \cite{h03}, spline estimators can effectively estimate $f_{\mcF_\textrm{S}}$. 
\item \label{eq:mcFA}  $\mcF_{\textrm{A}}=\{f: \mathbb{R}^d\to \mathbb{R}\;: f(\bfz)=\sum_{i=1}^df_j(z_j), \textrm{ each } f_j \textrm{ is } (p_j, C)\textrm{-H\" older smooth for parameters } \newline p_j, C \}$, the class of  functions with smooth additive components.
\cite{h03} and \cite{hhw10} consistently estimate $f_{\mcF_\textrm{A}}$ for both fixed $d$ and increasing $d$.

\item \label{eq:mcFL} $\mcF_{\textrm{L}}=\{f: \mathbb{R}^d\to \mathbb{R}\;: f(\bfz)=u^\top \bfz \textrm{ for some } u\in \mathbb{R}^d\}$, the class of linear functions. $f_{\mcF_\textrm{L}}$ can be consistently estimated by the standard linear least squares regression.
\end{enumerate}

Suppose an estimator of  $f_\mcF$ has been obtained, denoted by $\widehat{f}_\mcF$. 
Similar to (\ref{eq:estimator:step:1}), we can define an estimator of $\beta_0$ based on $\widehat{f}_\mcF$ as follows.
\begin{equation}\label{eq:general:beta:mcF}
	\widehat{\beta}_{\mcF}=\bigg(\frac{1}{n}\sum_{i=1}^n\widehat{f}_\mcF(\bfZ_i) \bfX_i\bigg)^{-1}\frac{1}{n}\sum_{i=1}^n\widehat{f}_\mcF(\bfZ_i) Y_i.
\end{equation}
As discussed in Section \ref{sec:model}, $\widehat{\beta}_{\mcF}$ is essentially the solution to (\ref{eq:empirical:moment:condition}) by replacing $\widehat{\bfh}_\textrm{opt}$ with $\widehat{f}_{\mcF}$. Under certain assumptions, it can be shown that
\begin{eqnarray}
	\sqrt{n}(\widehat{\beta}_{\mcF}-\beta_0)\cid N(0, \sigma_\epsilon^2D_\mcF^{-2}),\label{eq:general:beta:asymptotic}
\end{eqnarray}
where $D_\mcF^2=\ev[f_\mcF^2(\bfZ)]$. Note that $\widehat{\beta}_{\mcF_{\textrm{C}}}$ is our second-stage estimator defined in  (\ref{eq:estimator:step:1}); $\widehat{\beta}_{\mcF_{\textrm{S}}}$ is the efficient estimator proposed in \cite{n90};
$\widehat{\beta}_{\mcF_{\textrm{A}}}$  is the nonparametric additive instrumental variables estimator proposed
in \cite{fz18}; and $\widehat{\beta}_{\mcF_{\textrm{L}}}$ is asymptotically equivalent to the classical 2SLS estimator commonly used in the economics literature (e.g, see \citealp{ak91}). The last three estimators, $\widehat{\beta}_{\mcF_{\textrm{S}}}$, $\widehat{\beta}_{\mcF_{\textrm{A}}}$ and $\widehat{\beta}_{\mcF_{\textrm{L}}}$, can be incorporated into the general  framework in \cite{bcch12}. 

{We are now ready to summarize three benefits of Deep IV in comparison to its competitors. First, it is easy to see that $\mcF_{\textrm{L}}\subseteq \mcF_{\textrm{A}} \subseteq \mcF_{\textrm{S}} \subseteq \mcF_{\textrm{C}}$, and thus $D_{\mcF_\textrm{L}}^2\leq D_{\mcF_\textrm{A}}^2\leq D_{\mcF_\textrm{S}}^2\leq D_{\mcF_\textrm{C}}^2$. This indicates that our estimator $\widehat{\beta}_{\mcF_{\textrm{C}}}$ has the {\em smallest} variance. In fact, $\widehat{\beta}_{\mcF_\textrm{C}}$ turns out to be semiparametric efficient since Assumption \ref{Assumption:AC}\ref{A1:b} essentially requires the conditional mean $\ev(\bfX | \bfZ=\bfz)=f_0(\bfz)\in \mcF_\textrm{C}$. Second, the discussions right after Theorem \ref{theorem:asymptotic:distribution:deep} reveal that the semiparametric efficient  $\widehat{\beta}_{\mcF_{\textrm{C}}}$ requires a {\em weaker} condition $p^*>t^*/2$ on the smoothness of the underlying function when compared with the condition $p_H>d/2$ for $\widehat{\beta}_{\mcF_{\textrm{S}}}$. Finally, $\widehat{\beta}_{\mcF_{\textrm{C}}}$ has a smallest (second order) estimation error in the sense of \cite{cheng2008general} due to the fastest convergence rate of $\widehat{f}_{\textrm{C}}$ obtained by neural networks}

The following concrete example clearly illustrates these three advantages.
\begin{Example}\label{example:5}
Consider $\bfX=\prod_{j=1}^dg(Z_j)$ with $g$ being $(p, C)$-H\" older smooth and $Z_j$ being i.i.d. such that $\ev(g(Z_j))=0$. It can be shown that $f_{\mcF_{\textrm{A}}}=f_{\mcF_{\textrm{L}}}=0$. As a consequence, neither $\widehat{\beta}_{\mcF_{\textrm{L}}}$ nor $\widehat{\beta}_{\mcF_{\textrm{A}}}$ in \cite{fz18} is consistent. Notice  the function $f(\bfz)=\prod_{j=1}^dg(z_j)$  has a compositional  structure with $p^*=p$ and $t^*=1$. Moreover, by the conclusion \ref{lemma:pc:smooth:degree:result:1} in Lemma \ref{lemma:pc:smooth:degree},  we can verify  that the H\" older smoothness of $f$ is $p_H=p$. Therefore, the necessary conditions for $\widehat{\beta}_{\mcF_\textrm{S}}$ and $\widehat{\beta}_{\mcF_\textrm{C}}$ to guarantee (\ref{eq:general:beta:asymptotic}) are $p>d/2$ and $p>1/2$, respectively (see discussion of Theorem \ref{theorem:asymptotic:distribution:deep}). Finally, according to the discussion of Theorem \ref{theorem:rate:deep}, the convergence rate of $\widehat{f}_{\mcF_\textrm{C}}$ is $n^{-\frac{p}{2p+1}}$ that is faster than $n^{-\frac{p}{2p+d}}$. Hence, the estimation error of the second-stage estimator $\widehat{\beta}_{\mcF_\textrm{C}}$ will be smaller.
\end{Example}

\section{Some Auxiliary Results}\label{sec:extra:results}

The general theoretical results developed in previous sections are useful in other statistical inference problems in the field of IV. In this section, we present these auxiliary but useful results including split sample estimate, specification test and an extension of our models to contain exogenous variables, with different inferential purposes.

\subsection{Split-sample Estimator}
As mentioned in the discussion of Theorem \ref{theorem:asymptotic:distribution:deep}, a necessary condition for our theoretical results is that $f_{0,s}$'s must be sufficiently smooth, say $p^*>t^*/2$. Our goal in this section is to relax this condition by proposing a four-stage estimator through splitting the samples, as motivated by \cite{ak95, bcch12}.

We randomly divide the samples into two groups: group $a$ of size $n_a=\floor{n/2}$ and group $b$ of size $n_b=n-n_a$. Correspondingly, we define $(\bfX_i^a. \bfX_i^b)$ and $(\bfZ_i^a,\bfZ_i^b)$. 

Stage 1: Use the data in each group to construct the first-stage neural network estimators as in (\ref{eq:estimator:step:1}), and denote them by $\bff^k=(\widehat{f}_1^k,\ldots, \widehat{f}_q^k)$ for $k=a,b$.  

Stage 2: Given any $C_n>0$, we define the truncated neural network estimators:  
\begin{eqnarray}
	\widecheck{f}_s^k=\widehat{f}_s^kI(|\widehat{f}_s^k|\leq C_n),\quad \textrm{ for } k=a,b \textrm{ and } s=1,\ldots, q. \nonumber
\end{eqnarray}

Stage 3: Let $\widecheck{\bfX}_i^a=(\widecheck{f}_1^b(\bfZ_i^a),\ldots, \widecheck{f}_q^b(\bfZ_i^a))^\top$ and $\widecheck{\bfX}_i^b=(\widecheck{f}_1^a(\bfZ_i^b),\ldots, \widecheck{f}_q^a(\bfZ_i^b))^\top$, and we define the following two estimators:
\begin{equation*}
	\widecheck{\beta}^a=\bigg(\frac{1}{n_a}\sum_{i=1}^{n_a}\widecheck{\bfX}_i^a\bfX_i^{a\top}\bigg)^{-1}\frac{1}{n_a}\sum_{i=1}^{n_a}\widecheck{\bfX}_i^a Y_i^a \quad \textrm{ and } \quad \widecheck{\beta}^b=\bigg(\frac{1}{n_b}\sum_{i=1}^{n_b}\widecheck{\bfX}_i^b\bfX_i^{b\top}\bigg)^{-1}\frac{1}{n_b}\sum_{i=1}^{n_b}\widecheck{\bfX}_i^b Y_i^b.
\end{equation*}

Stage 4: Combining $\check{\beta}^a$ and $\check{\beta}^b$, we construct the following split-sample estimator:
\begin{equation*}
	\widecheck{\beta}^{ab}=\bigg(\sum_{i=1}^{n_a}\widecheck{\bfX}_i^a\bfX_i^{a\top}+\sum_{i=1}^{n_b}\widecheck{\bfX}_i^b\bfX_i^{b\top}\bigg)^{-1}\bigg(\sum_{i=1}^{n_a}\widecheck{\bfX}_i^a\bfX_i^{a\top}\widecheck{\beta}^a+\sum_{i=1}^{n_b}\widecheck{\bfX}_i^b\bfX_i^{b\top}\widecheck{\beta}^b\bigg).
\end{equation*}
The truncation idea in the second stage is inspired by \cite{gkkw06} and \cite{bk19}, and aims to obtain the convergence results in terms of norm $\|\cdot\|$, which is critical to the success of the final split-sample estimator. Without bounding the estimator, only the convergence rate in terms of  norm $\|\cdot\|_n$ can be derived. {In practice, the truncation parameter $C_n$ can be chosen as $c\log(n)$ for some constant $c$.}

Theorems \ref{theorem:split:sample:DIVE:deep} reveals that the truncated neural network estimators $\widecheck{f}_{s}^k$'s converge at the optimal rate (up to a logarithm term) and the split-sample estimator $\widecheck{\beta}^{ab}$ achieves semiparametric efficiency, as long as $C_n$ grows slowly at a $\log n$-rate and the network structure is properly specified. In comparison with Theorem \ref{theorem:asymptotic:distribution:deep}, the smoothness condition $p^*>t^*/2$ is neither explicitly nor implicitly required in Theorem \ref{theorem:split:sample:DIVE:deep} due to sample splitting. 
\begin{theorem}\label{theorem:split:sample:DIVE:deep}
Under Assumption \ref{Assumption:AC}, if $C_n \to \infty$, $C_n=O(\log(n))$, $LW=o(\sqrt{n})$ and $LWd=o(n)$, then for $k=a,b$ and  $s=1,\ldots, q$, it follows that
\begin{equation}
	\|\widecheck{f}_s^k-f_{0,s}\|=O_P\bigg(\log^5(n)\sqrt{\frac{L^2W^2+LWd}{n}}+\log^{\frac{4p^*}{t^*}}(n)(LW)^{-\frac{2p^*}{t^*}}\bigg),\; \nonumber
\end{equation}
As a consequence, if $d=O(LW)$ and $LW\asymp n^{\frac{t^*}{2(2p^*+t^*)}}$, then for $k=a,b$ and  $s=1,\ldots, q$, the following holds:
\begin{equation*}
	\|\widecheck{f}_s^k-f_{0,s}\|=O_P\bigg([\log(n)]^{5+\frac{4p^*}{t^*}}n^{-\frac{p^*}{2p^*+t^*}}\bigg).
\end{equation*}
In addition, if Assumption \ref{Assumption:A3} holds and $\log^6(n)LW=o(\sqrt{n})$, $\log^{12}(n)LWd=o(n)$, $\log^{\frac{t^*+4p^*}{2p^*}}(n)=o(LW)$, then we have
\begin{equation*}
	\sqrt{n}(\widecheck{\beta}^{ab}-\beta_0)\cid  \textrm{N}(0, \sigma_\epsilon^2 \ev^{-1}(\bfD\bfD^\top)).
\end{equation*}
\end{theorem}


To conduct statistical inference using the split-sample estimator, we propose the following estimators for the asymptotic covariance:
\begin{eqnarray}
	\widecheck{\bfV}^2_{ab}= \bigg(\sum_{i=1}^{n_a}\widecheck{\bfX}_i^a\bfX_i^{a\top}+\sum_{i=1}^{n_b}\widecheck{\bfX}_i^b\bfX_i^{b\top}\bigg) ^{-1}\bigg(\sum_{i=1}^{n_a}|\widecheck{\epsilon}_i^{a}|^2+\sum_{i=1}^{n_b}|\widecheck{\epsilon}_i^{b}|^2\bigg),\nonumber
\end{eqnarray}
where $\widecheck{\epsilon}_i^{k}=Y_i^k-\bfX_i^{k\top}\widecheck{\beta}^{ab}$  for $k=a,b$. Lemma \ref{lemma:split:sample:estimation:variance} below show that the above covariance estimator is consistent.
\begin{lemma}\label{lemma:split:sample:estimation:variance}
Under conditions of Theorems \ref{theorem:split:sample:DIVE:deep}, it holds that $\widecheck{\bfV}_{ab}^2\cip \sigma_\epsilon^2\ev^{-1}(\bfD\bfD^\top)$.
\end{lemma}

\subsection{Specification Test}

One fundamental problem in the field of IV is whether or not the instrumental variables are indeed exogenous. In this section, following \cite{h78} and \cite{bcch12}, we propose a Hausman-type testing procedure to address this issue. 

Suppose that we have several baseline instrumental variables and also that the first $d_b$ instruments are valid, denoted as $\widetilde{\bfZ}=(Z_1, \ldots, Z_{d_b})^\top$. The goal is to test whether the rest variables $Z_{d_b+1},\ldots, Z_{d}$ are also valid instruments or not. Let $\widehat{\beta}$ and $\widetilde{\beta}$ be the estimators proposed in (\ref{eq:estimator:step:2}) based on the potential instruments $\bfZ=(Z_1,\ldots, Z_{d_b}, Z_{d_b+1},\ldots, Z_d)^\top$ and baseline instruments $\widetilde{\bfZ}$, respectively. Define $g_{0,s}(\widetilde{\bfz})=\ev(X_{s}|\widetilde{\bfZ}=\widetilde{\bfz})$ for $\widetilde{\bfz}\in \mathbb{R}^{d_b}$ and $\widehat{g}_s$ the DNN estimator in (\ref{eq:estimator:step:1}) by replacing $\bfZ$ with $\widetilde{\bfZ}$. Similarly, define $\widetilde{\bfD}=(g_{0,1}(\widetilde{\bfZ}),\ldots, g_{0,q}(\widetilde{\bfZ}))^\top$ and $\widetilde{\bfX}_i=(\widehat{g}_1(\widetilde{\bfZ}_i),\ldots, \widehat{g}_q(\widetilde{\bfZ}_i))^\top$. With these notation, the estimands of  $\widehat{\beta}$ and $\widetilde{\beta}$ are essentially $$\beta_{\bfZ}:=\ev^{-1}(\bfD\bfD^\top)\ev(\bfD Y)\quad\mbox{ and }\quad \beta_{\widetilde{\bfZ}}:=\ev^{-1}(\widetilde{\bfD}\widetilde{\bfD}^\top)\ev(\widetilde{\bfD} Y),$$ respectively.
Since $\widetilde{\bfZ}$ is a vector of valid instruments, it can be shown that
\begin{eqnarray*}
	\beta_{\bfZ}=\beta_0+\ev^{-1}(\bfD\bfD^\top)\ev(\bfD \epsilon)\quad \textrm{ and }\quad  \beta_{\widetilde{\bfZ}}=\beta_0.
\end{eqnarray*}
Therefore, if  all the elements in $\bfZ$ are  also valid instruments, then $\beta_{\bfZ}=\beta_{\widetilde{\bfZ}}$ such that the difference between $\widehat{\beta}$ and $\widetilde{\beta}$ is expected to be small. 

Based on the above intuition, we propose to test $$H_0: \beta_{\bfZ}=\beta_{\widetilde{\bfZ}}\;\;\mbox{versus}\;\;H_1: \beta_{\bfZ}\neq \beta_{\widetilde{\bfZ}}$$ using the following test statistic
\begin{equation}\label{eq:test:statistics}
	J=\frac{n}{\widetilde{\sigma}_\epsilon^{2}} (\widehat{\beta}-\widetilde{\beta})^\top\bigg[\bigg(\frac{1}{n}\sum_{i=1}^n\widetilde{\bfX}_i\bfX_i^\top\bigg)^{-1}-\bigg(\frac{1}{n}\sum_{i=1}^n\widehat{\bfX}_i\bfX_i^\top\bigg)^{-1}\bigg]^{-1} (\widehat{\beta}-\widetilde{\beta}),
\end{equation}
where $\widetilde{\sigma}_\epsilon^2=n^{-1}\sum_{i=1}^n(Y_i-\widetilde{\beta}^\top\bfX_i)^2$ and the two matrices in (\ref{eq:test:statistics}) are essential the estimators of $\ev^{-1}(\widetilde{\bfD}\widetilde{\bfD}^\top)$ and $\ev^{-1}({\bfD}{\bfD}^\top)$. 
Some additional regularity conditions are introduced for studying the proposed test statistic.
\begin{Assumption}\label{Assumption:A7}
\begin{enumerate}[label={(\roman*}),ref={(\roman*})]
\item \label{A7.a} The underlying  functions $g_{0,1},\ldots, g_{0,q}\in \mcCS(L_*, \bfd, \bft, \bfp, \bfa, \bfb, C)$, with all the parameters fixed constants, except for $d_0=d_b$ which could be diverging.
\item \label{A7.b} The matrices $\ev(\widetilde{\bfD}\widetilde{\bfD}^\top)$ and $\ev^{-1}(\widetilde{\bfD}\widetilde{\bfD}^\top)-\ev^{-1}({\bfD}{\bfD}^\top)$ are positive definite.
\end{enumerate}
\end{Assumption}

The compositional structure Assumption \ref{Assumption:A7}\ref{A7.a} is similar to Assumption \ref{A1:b}. Assumption \ref{Assumption:A7}\ref{A7.b} is a regularity condition for the asymptotic covariance matrices.  Since $\widetilde{\bfZ}$ is a subvector of $\bfZ$, we can show that the matrix $\ev^{-1}(\widetilde{\bfD}\widetilde{\bfD}^\top)-\ev^{-1}({\bfD}{\bfD}^\top)$ is at least nonnegative definite.

The following theorem justifies the proposed test statistic $J$: reject $H_0$ if $J>\chi^2_\alpha(q)$, where $\chi^2_\alpha(q)$ is the $\alpha$-th upper percentile of $\chi_q^2$. 

\begin{theorem}\label{theorem:specification:test}
\begin{enumerate}
\item Under the conditions in Theorem \ref{theorem:asymptotic:distribution:deep} and Assumption \ref{Assumption:A7}, we have $J\cid \chi^2(q)$.
\item  Suppose that the conditions in Theorem \ref{theorem:asymptotic:distribution:deep} and Assumption \ref{Assumption:A7} hold, except that Assumption \ref{Assumption:A3}\ref{A3:0} therein is replaced by $E(\epsilon|Z_j)=0$ for $j=1,\ldots, d_b$. Furthermore, if $\|\ev(\bfD\epsilon)\|_2>0$, then $J \to \infty$ in probability.
\end{enumerate}
\end{theorem}

\subsection{Model with Exogenous Variables}\label{sec:extension}
In this section, we discuss an extension of the model (\ref{eq:model:1}) containing both endogenous and exogenous variables. To be more specific, we consider i.i.d. observations $(Y_i, \bfX_i, \bfR_i, \bfZ_i)$ generated from the following model:
\begin{eqnarray}
	Y=\beta_0^\top \bfX+ \alpha_0^\top\bfR+\epsilon,\nonumber
\end{eqnarray}
where $Y$ is the response, $\bfX\in \mathbb{R}^{q_1}$ are the endogenous explanatory variables, $\bfR\in \mathbb{R}^{q_2}$ are the exogenous explanatory variables, $\epsilon\in \mathbb{R}$ is the random noise, and $\bfZ\in \mathbb{R}^d$ are the instrumental variables. Since $\bfR$ are exogenous such that $\ev(\epsilon | \bfR)=0$, we can add them to the instruments set and define $\widetilde{\bfZ}=(\bfR^\top, \bfZ^\top)^\top\in \mathbb{R}^{q_2+d}$. 

Given the above setup, we propose the following two-stage estimator for $(\beta_0^\top, \alpha_0^\top)^\top \in \mathbb{R}^{q_1+q_2}$. In the first stage, we fit a neural network using $\bfX$ as response and $\widetilde{\bfZ}$ as the explanatory variables:
\begin{equation}
	\widehat{\bff}:=\argmin_{\substack{\bff\in \mcF_{d,q_1}(L, W)}}\frac{1}{n}\sum_{i=1}^n\|\bfX_i-\bff(\widetilde{\bfZ}_i)\|_2^2.\nonumber
\end{equation}
Furthermore, we denote the $q_1$ outputs of $\widehat{\bff}$ as $\widehat{f}_1,\ldots, \widehat{f}_{q_1}$. In the second stage,  we can estimate $\beta$ and $\alpha$ as follows:
\begin{equation}
	\begin{pmatrix}
	\widehat{\beta}\\
	\widehat{\alpha}
	\end{pmatrix}=\bigg(\frac{1}{n}\sum_{i=1}^n\widehat{\bfD}_i\bfD_i^\top\bigg)^{-1}\frac{1}{n}\sum_{i=1}^n\widehat{\bfD}_i Y_i.\nonumber
\end{equation}
where $\widehat{\bfD}_i:=(\widehat{\bff}^\top(\bfZ_i), \bfR_i^\top)^\top=(\widehat{f}_1(\bfZ_i),\ldots, \widehat{f}_{q_1}(\bfZ_i), \bfR_i^\top)^\top\in \mathbb{R}^{q_1+q_2}$ and $\bfD_i:=(\bfX_i^\top, \bfR_i^\top)^\top \in \mathbb{R}^{q_1+q_2}$ for $i=1,\ldots, n$. Our previous theory can naturally carry over to this extension.

\section{Monte Carlo Simulation}\label{sec:simulation}
In this section, we provide several  simulation studies to demonstrate the finite-sample performance of the proposed procedure.  We consider the following two data generating processes (DGP):

\begin{itemize}[wide, labelwidth=!, labelindent=0pt]
\item[DGP 1 (Weak IV): ]  $Y=3x+20\epsilon$ and $X=f_0(\bfZ)+\epsilon$, with $f_0(\bfZ)=Z_1\sin(Z_2)+Z_3Z_4$. Here $Z_i$'s are i.i.d. uniformly distributed in $[-3, 3]$ and $\epsilon\sim N(0, 1)$ independent of $Z_i$'s;
\item[DGP 2 (Linear Reduced Form):]  $Y=3x+20\epsilon$ and $X=f_0(\bfZ)+\epsilon$, with $f_0(\bfZ)=3Z_1+4Z_2-2Z_3+Z_4$. Here $Z_i$'s and $\epsilon$ are generated similarly as DGP 1.
\end{itemize}

DGP 1 corresponds to the weak IV case and is a special case of Example \ref{example:5}. DGP 2 requires a linear reduced form equation. In our simulation settings, the sample size was chosen to be $n=100, 200, 500, 1000, 2000$, and each experiment was repeated 1000 times. 

\subsection{First-Stage Estimator}
We consider the following four types of nonparametric and parametric estimation procedures discussed in Section \ref{sec:boosting:efficiency} to obtain the first-stage estimators.
\begin{enumerate}
\item Deep Neural Network (DNN): This estimator is constructed as in (\ref{eq:estimator:step:1}) using deep neural network with depth $L=3$ and width $W=10$.
\item Penalized Tensor Product Spline (P-Spline): The univariate cubic polynomial basis on $[-3, 3]$  is chosen to be $\bfB(z)=(1, z, z^2, z^3, (z-t_1)^3_+, \ldots, (z-t_{20})^3_+)$, where $t_i$'s are the equally-spaced points in $[-3, 3]$. The tensor product spline basis on $[-3, 3]^4$ is defined as a collection of all the interactions between $B_{j_1}(z_1)B_{j_2}(z_2)B_{j_3}(z_3)B_{j_4}(z_4)$, where $B_j(z)$ is the $j$-th element of $\bfB(z)$. Based on the cubic tensor product spline basis, we apply the lasso estimation procedure to select the optimal instruments.
\item Additive Spline (A-Spline): We use  $(\bfB(z_1), \bfB(z_2), \bfB(z_3), \bfB(z_4))$ as the additive spline basis, and  apply the lasso estimation procedure to select the optimal instruments.
\item Linear Regression (LR): The first-stage is the simple linear regression estimator using $\bfX$ as the response and $\bfZ$ as the explanatory variables.
\end{enumerate}

For the DNN estimator,  a widely used and effective algorithm to solve the optimization problem in (\ref{eq:estimator:step:1}) is Stochastic Gradient Descent (SGD). We randomly divided the observations into the training set with sample size $\floor{0.8n}$ and testing set with sample size $n-\floor{0.8n}$. The training set is used to update the weights of the neural network by SGD, while the testing set is used to calculate the testing error. The TensorFlow package in python was applied to obtain the numerical results. 

We used the root mean square error (RMSE) to evaluate the first-stage estimator $\widehat{f}$. Figure \ref{fig:rmse:dnn1} reveals that for DGP 1, the RMSEs of the A-Spline and LR estimators do not decrease even the sample size increases since $f_0$ in DGP 1 does not have the linear or additive structure. At the same time, the DNN estimator has a significantly lower RMSE than the P-Spline estimator because the latter may not be able to effectively capture the intrinsic structure of $f_0$ in DGP 1. For DGP 2, it can be seen from Figure \ref{fig:rmse:dnn2} that the estimator $\widehat{f}$ obtained by LR has the smallest RMSE, while A-Spline and DNN estimators have slightly larger RMSEs. However,  the RMSE of the P-Spline estimator decreases slowly when the sample size increases.

\begin{figure}[H]
\centering
\begin{minipage}{.45\textwidth}
  \centering
  \includegraphics[width=3.2 in, height=3 in]{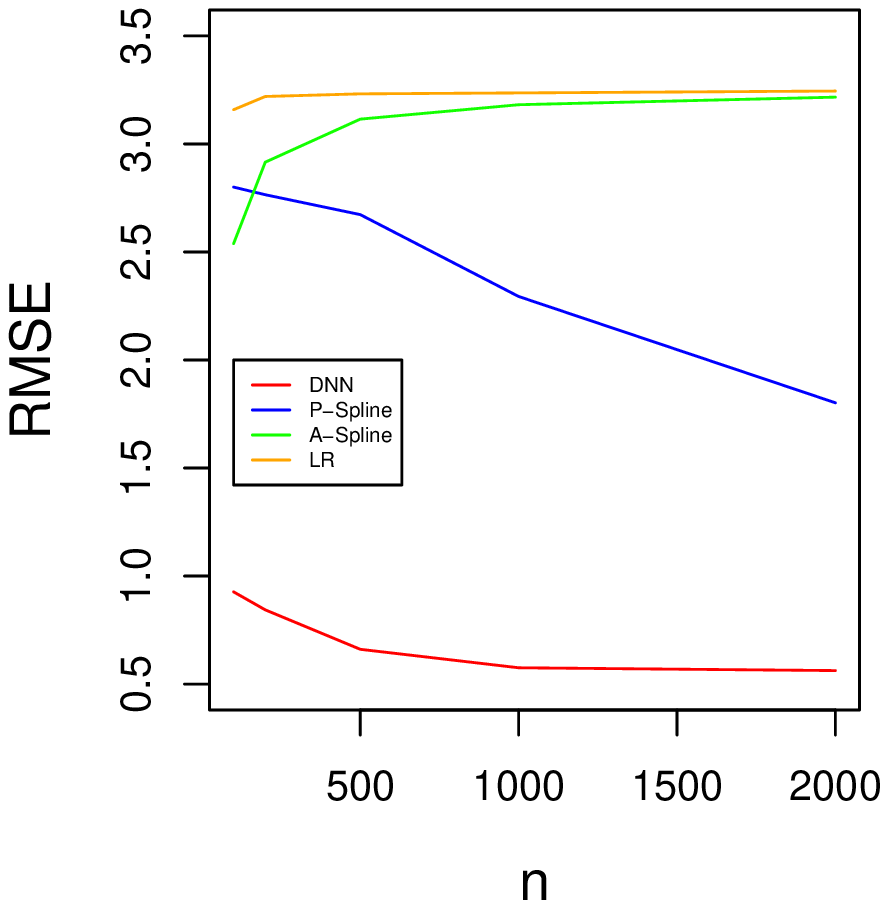}
  \captionof{figure}{RMSE of first-stage estimator for DGP  1}
  \label{fig:rmse:dnn1}
\end{minipage}%
\hfill
\begin{minipage}{.45\textwidth}
  \centering
  \includegraphics[width=3.2 in, height=3 in]{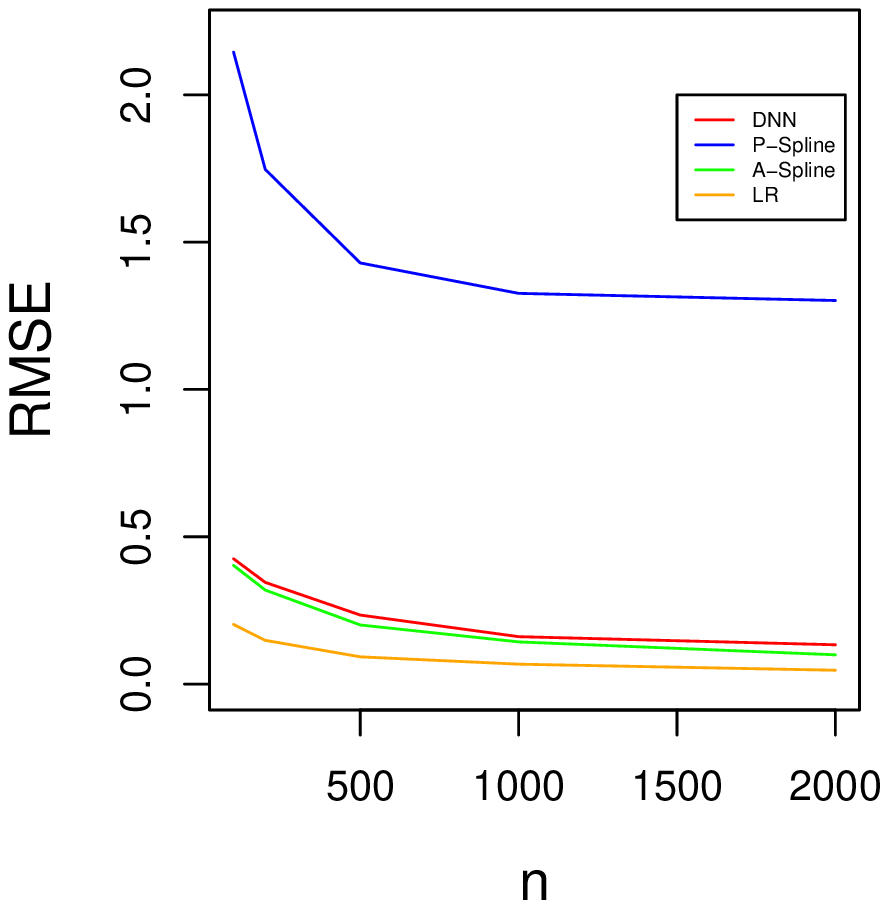}
  \captionof{figure}{RMSE of first-stage estimator for DGP  2}
  \label{fig:rmse:dnn2}
\end{minipage}%
\end{figure}

We next evaluate the performance of the DNN estimators with different neural network structures by investigating the RMSE of $\widehat{f}$ in (\ref{eq:estimator:step:1}) with all the combinations of $W=5, 10, 20$ and $L=3, 5, 10$. It can be observed from Figures \ref{fig:rmse:wl:dnn1} and \ref{fig:rmse:wl:dnn2} that, the errors decrease when the sample size increases regardless of the choices of $W$ and $L$. Moreover, in terms of RMSE, the performance of $\widehat{f}$ is quite similar, especially when the sample size is large ($n\geq 1000$).

\begin{figure}[H]
\includegraphics[width=2.1 in]{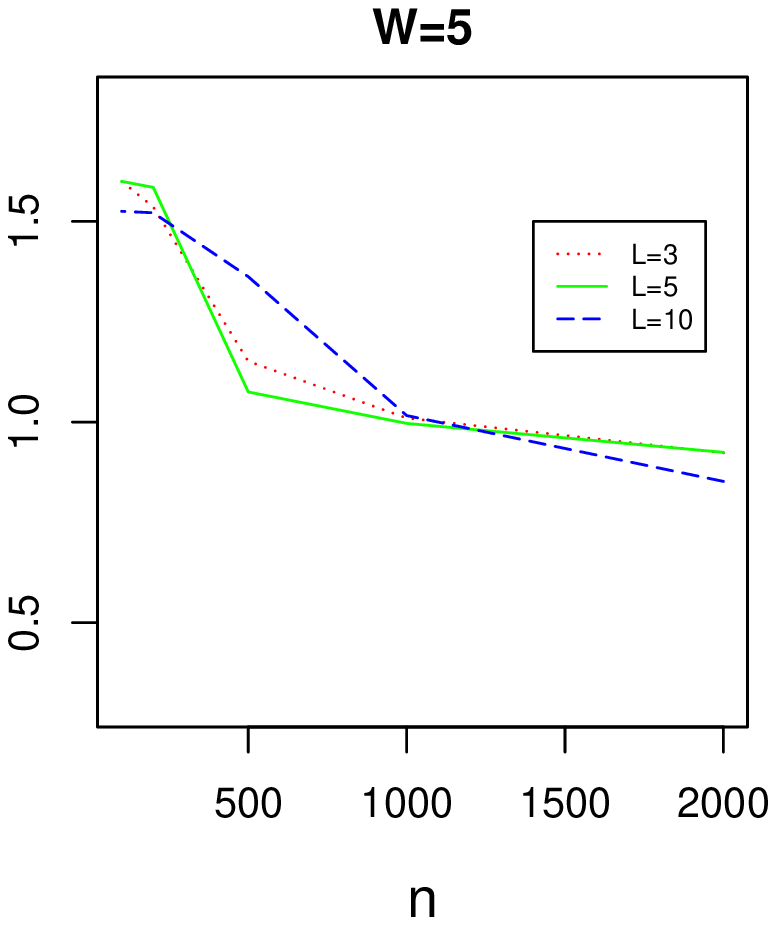}
\includegraphics[width=2.1 in]{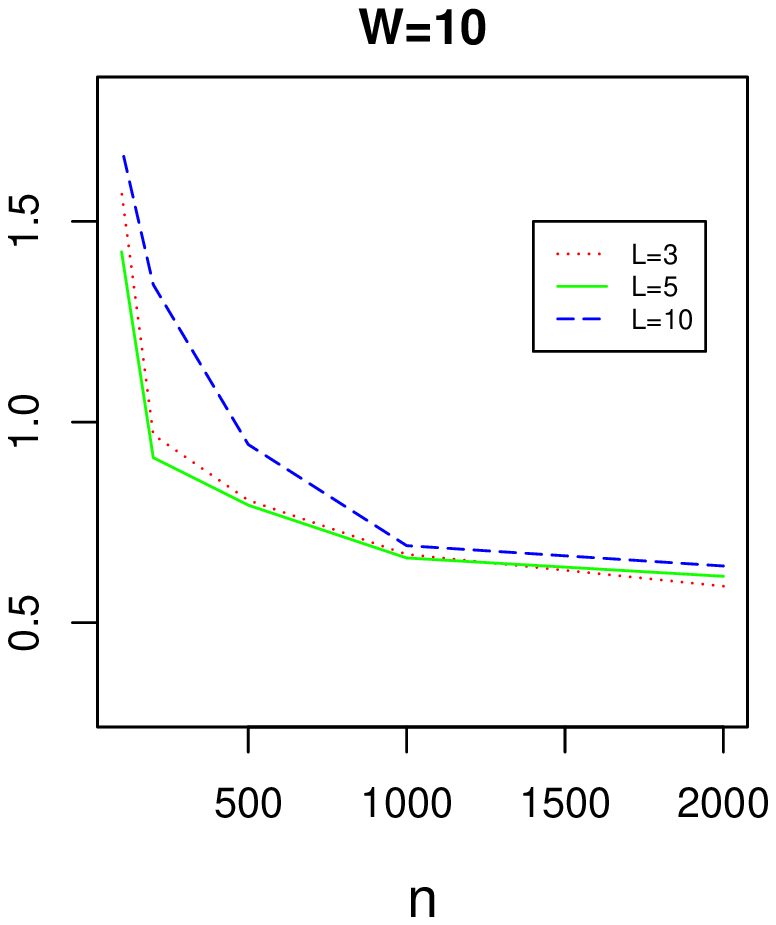}
\includegraphics[width=2.1 in]{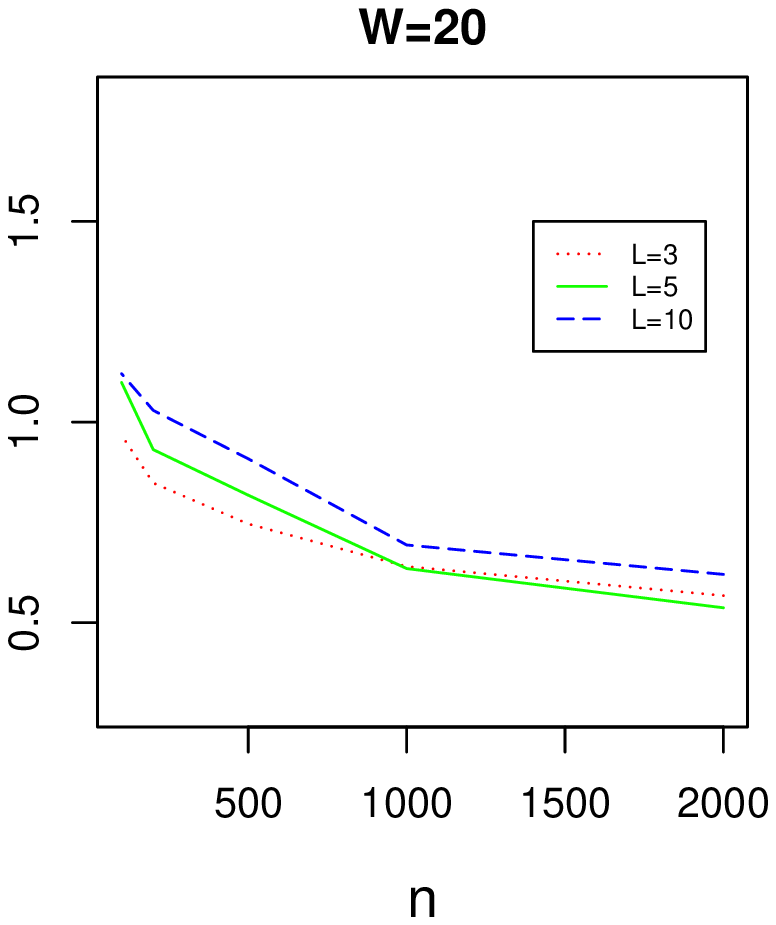}
\caption{RMSE of the DNN estimator $\widehat{f}$ with different $(W, L)$ under DGP 1}
\label{fig:rmse:wl:dnn1}
\end{figure}

\begin{figure}[H]
\includegraphics[width=2.1 in]{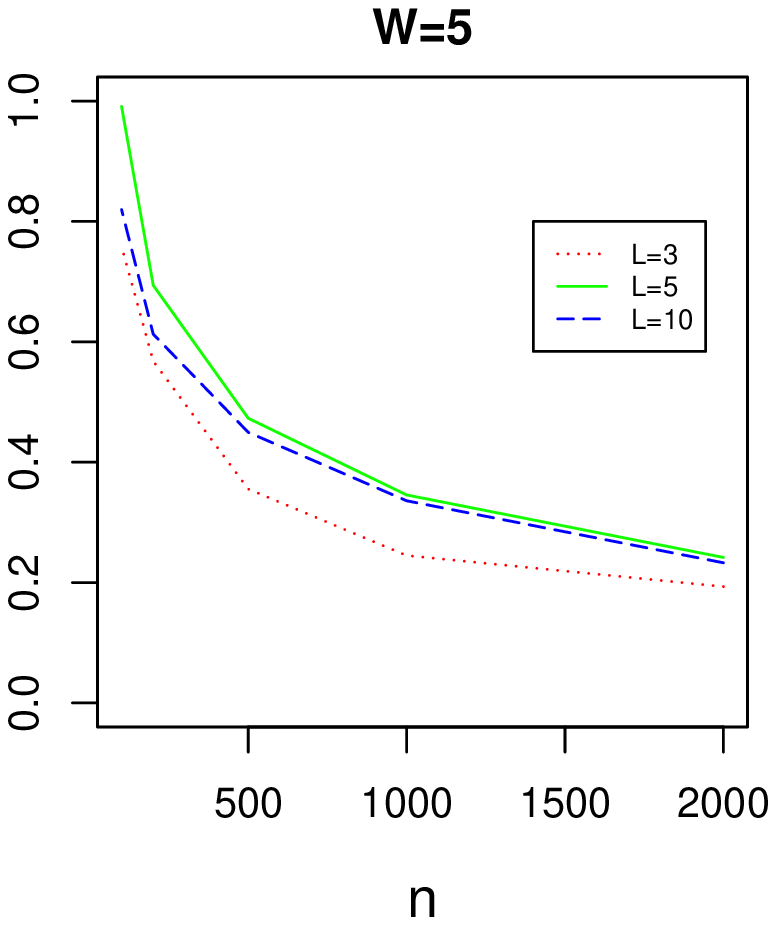}
\includegraphics[width=2.1 in]{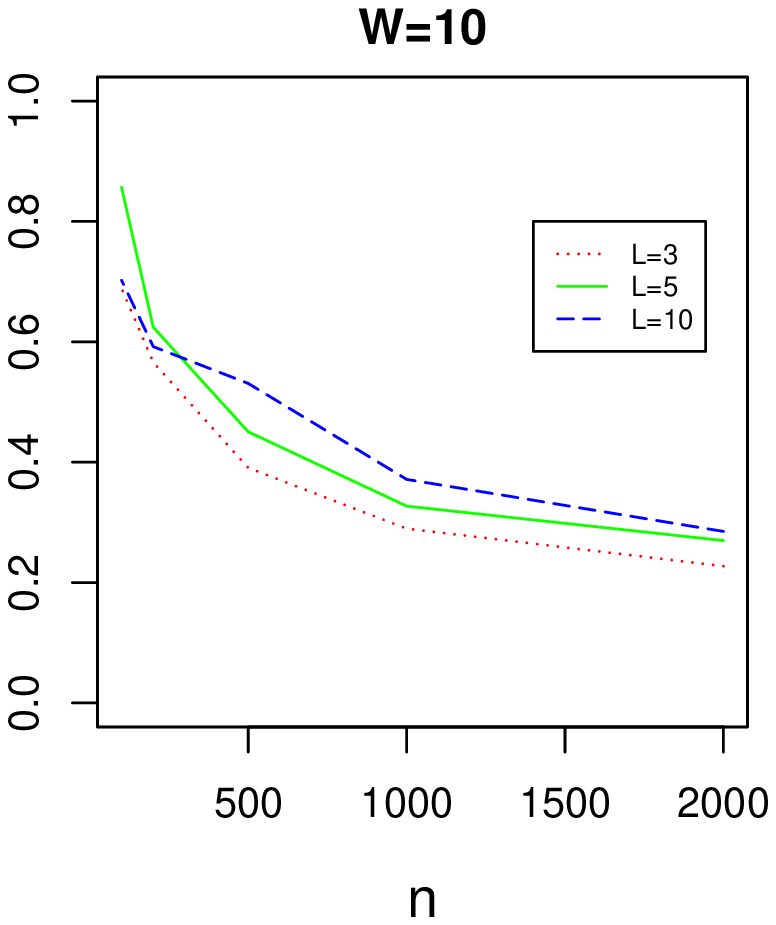}
\includegraphics[width=2.1 in]{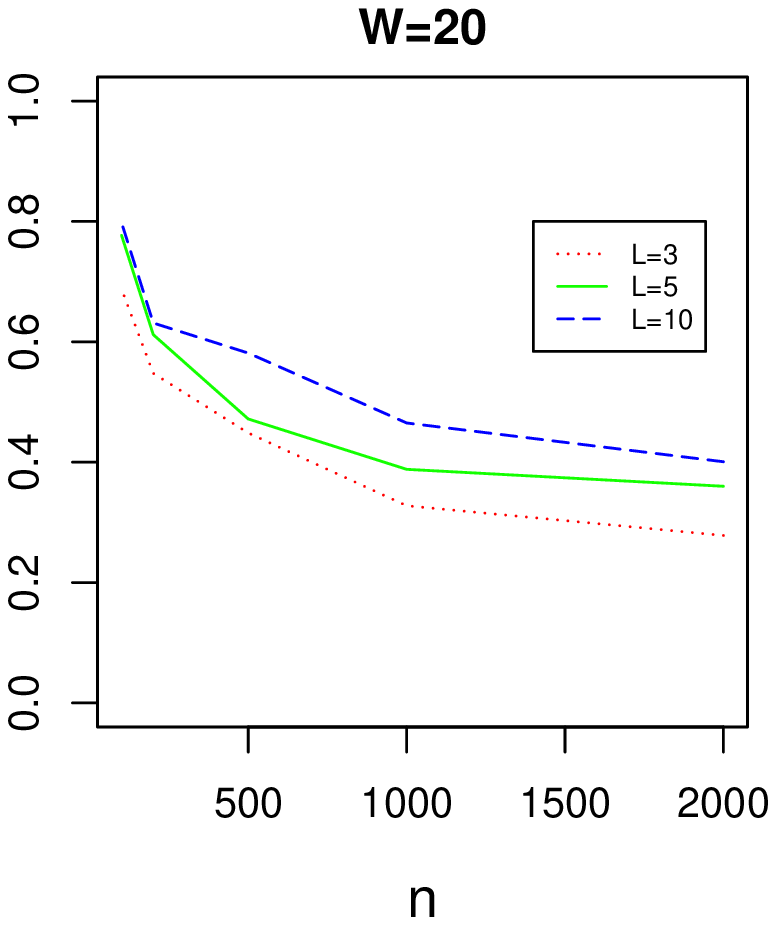}
\caption{RMSE of the DNN estimator $\widehat{f}$ with different $(W, L)$ under DGP 2}
\label{fig:rmse:wl:dnn2}
\end{figure}

\subsection{Second-Stage Estimator}

In addition to the second-stage estimator $\widehat{\beta}$ obtained through (\ref{eq:general:beta:mcF}), we consider two other estimators of $\beta_0$ for comparison purpose. The first one is the ORACLE estimator that was obtained using $Y$ as the response and $f_0(\bfZ)$ as the explanatory variable, and the second one is the naive OLS estimator that regressed $Y$ with respect to $X$. The former is unrealistic but used as a benchmark, while the second is known to be inconsistent due to endogeneity. 

We use RMSE  of $\widehat{\beta}$ as a criterion to evaluate the finite sample performance. Figure \ref{fig:rmse:dive:1} suggests that in DGP 1,  the LR, A-Spline, and OLS estimators have large RMSEs, which are not decreasing even when the sample size is relatively large ($n=2000$). This coincides with the theoretical analysis of the weak IV case in Example \ref{example:5}. In contrast,  the P-Spline, DNN, and ORACLE estimators have relatively smaller error. Besides, the DNN estimator has a similar performance as the ORACLE estimator with a large sample size ($n\geq 1000$). However, when compared with the P-Spline estimator, the DNN estimator is uniformly better regardless of the sample size. In DGP 2, Figure \ref{fig:rmse:dive:2} shows that the errors of all the estimators, except for the OLS, are significantly reduced when increasing the sample size. Moreover, the ORACLE, LR, A-Spline and DNN estimators have comparable performance when the sample size is great than $500$. However, when compared with the P-Spline estimator, the DNN estimator stands out under different sample sizes.

\begin{figure}[H]
\centering
\begin{minipage}{.45\textwidth}
  \centering
  \includegraphics[width=3.2 in, height=3 in]{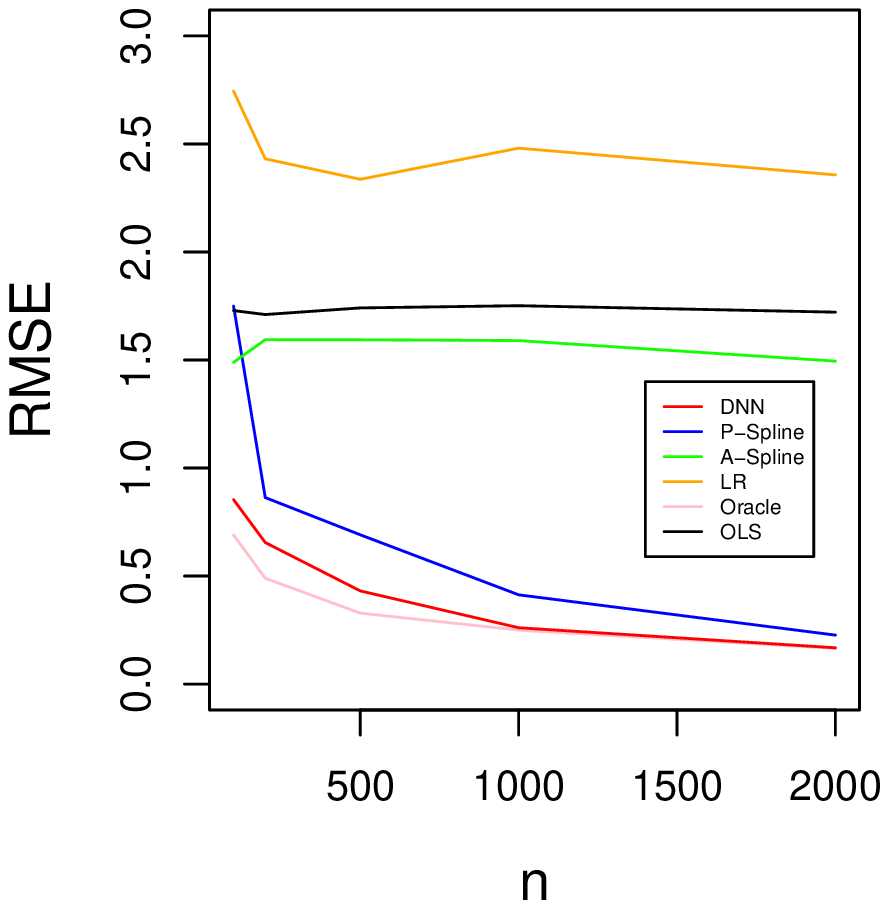}
  \captionof{figure}{RMSE of second-stage estimator for DGP 1}
  \label{fig:rmse:dive:1}
\end{minipage}
\hfill
\begin{minipage}{.45\textwidth}
  \centering
  \includegraphics[width=3.2 in, height=3 in]{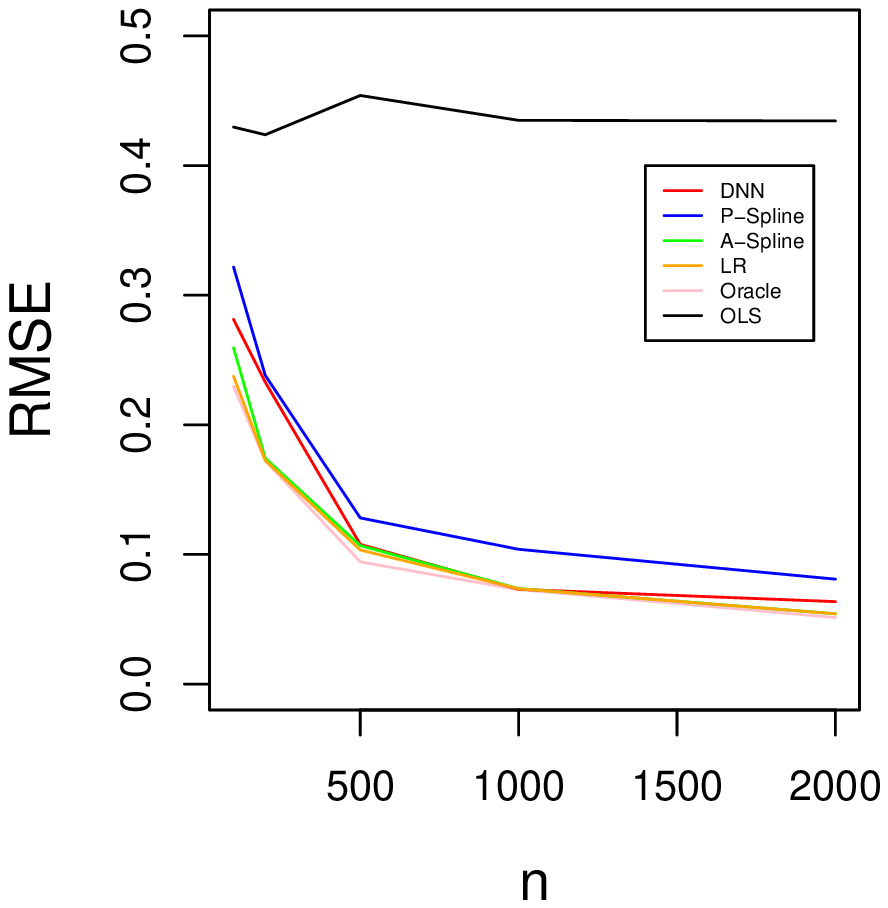}
  \captionof{figure}{RMSE of second-stage estimator for DGP 2}
  \label{fig:rmse:dive:2}
\end{minipage}
\end{figure}

We further conduct additional studies to evaluate the performance of $\widehat{\beta}$ using deep neural network with different $W$ and $L$. Figures \ref{fig:rmse:wl:beta1} and \ref{fig:rmse:wl:beta2} reveal that the DNN estimator is fairly stable to the choices of network structure. When sample size is great than $1000$, the RMSEs are very close for different $W$ and $L$.

\begin{figure}[H]
\includegraphics[width=2.1 in]{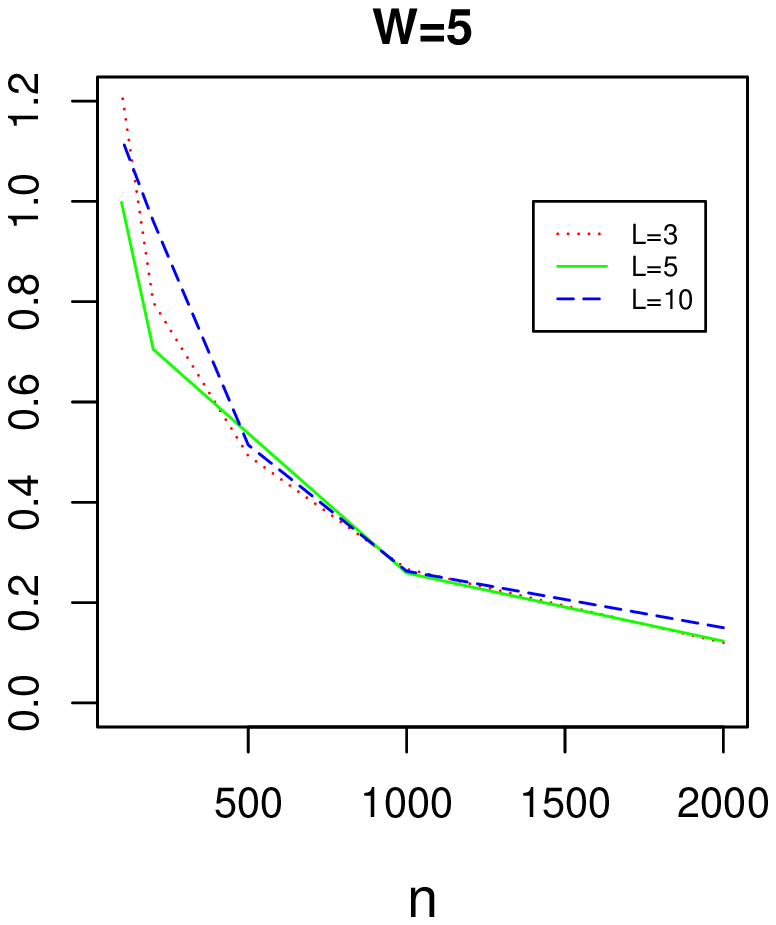}
\includegraphics[width=2.1 in]{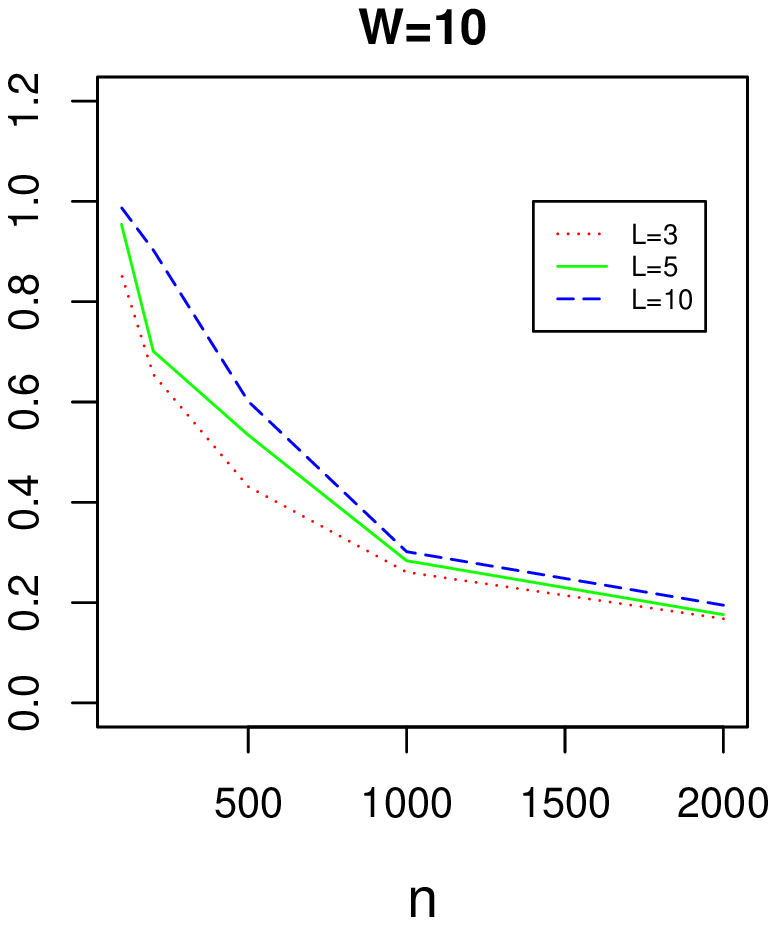}
\includegraphics[width=2.1 in]{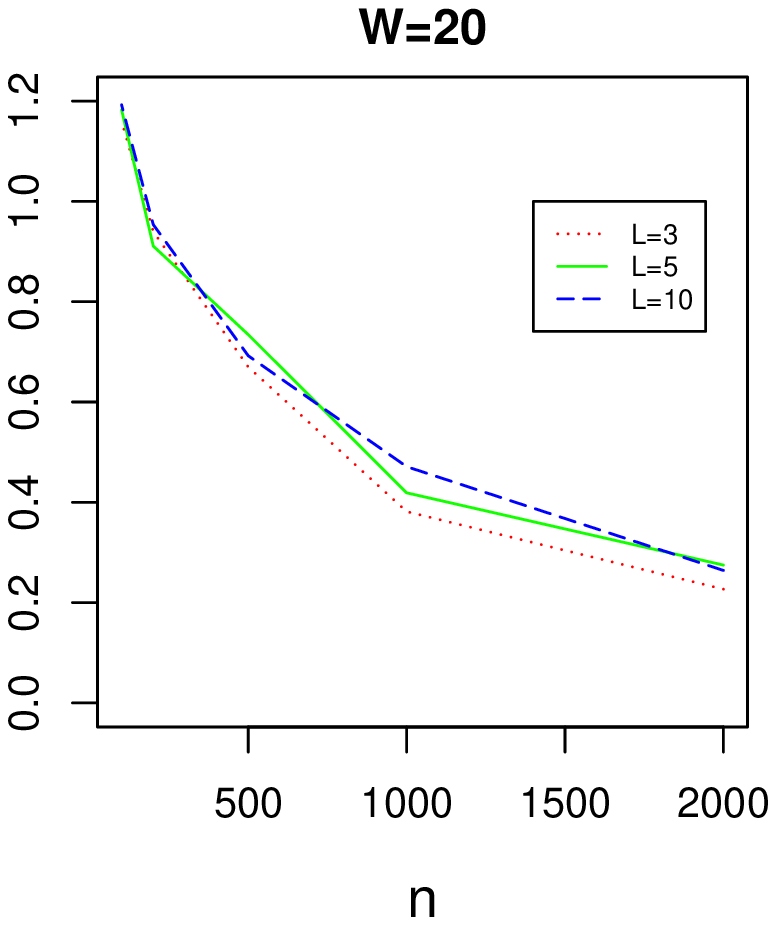}
\caption{RMSE of the DNN estimator $\widehat{\beta}$ with different $(W, L)$ for DGP 1}
\label{fig:rmse:wl:beta1}
\end{figure}

\begin{figure}[H]
\includegraphics[width=2.1 in]{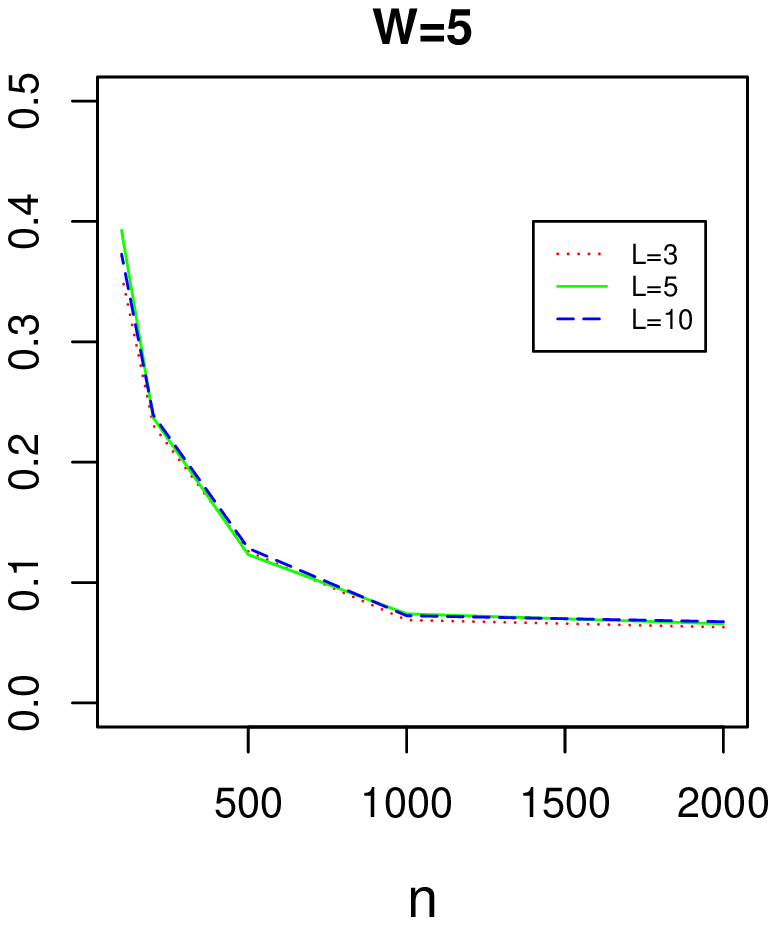}
\includegraphics[width=2.1 in]{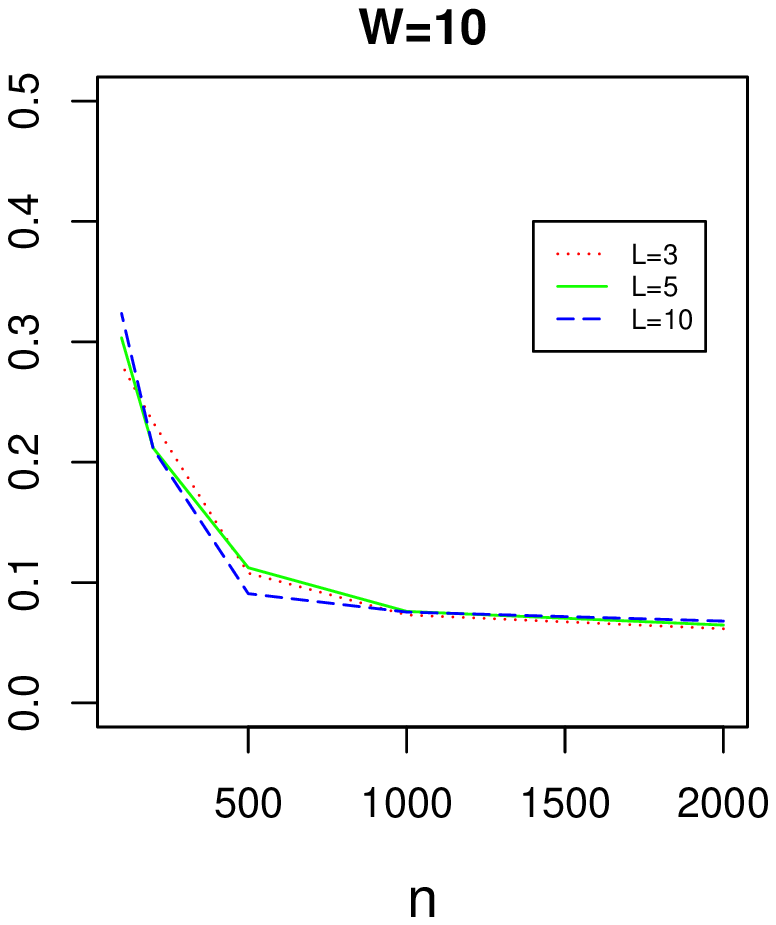}
\includegraphics[width=2.1 in]{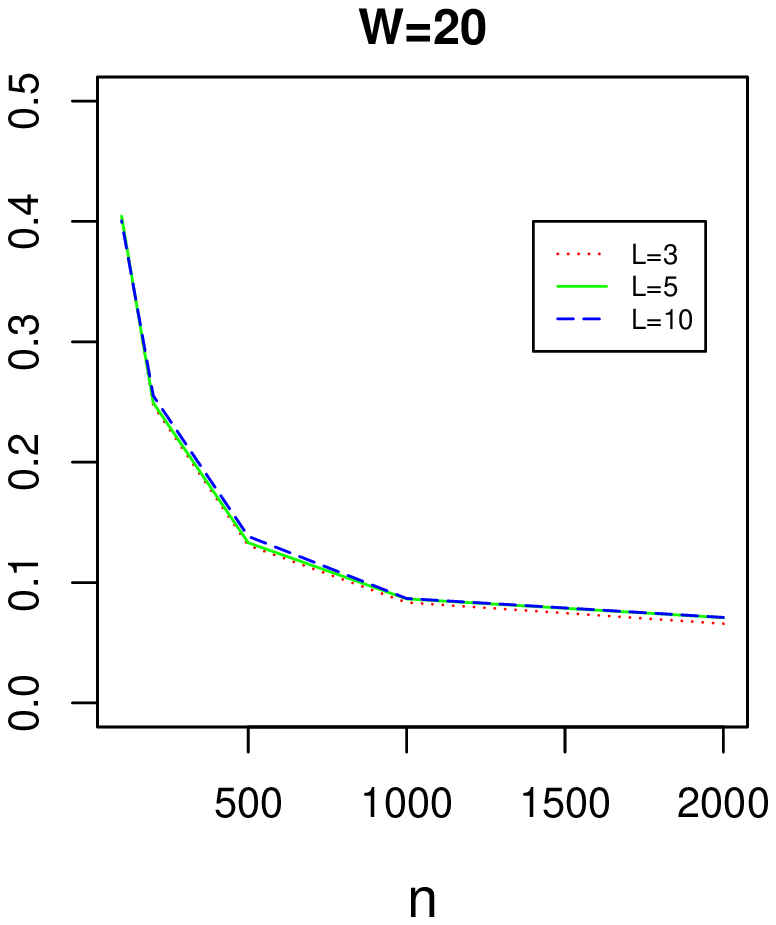}
\caption{RMSE of the DNN estimator $\widehat{\beta}$ with different $(W, L)$ for DGP 2}
\label{fig:rmse:wl:beta2}
\end{figure}


\subsection{Coverage Probability}
We calculated the coverage probabilities of the proposed confidence interval in (\ref{eq:confidence:interval}) to examine its empirical performance. The benchmark for comparison is based on the ORACLE estimator. Figures \ref{fig:ci:dgp1} and  \ref{fig:ci:dgp2}  report the coverage probabilities of the 95\% confidence intervals based on DNN estimator and the ORACLE estimator under different sample sizes. In DGP 1, Figure \ref{fig:ci:dgp1} shows that when sample size is relatively large ($n\geq 1000$), the coverage probability of the DNN estimator is around $95\%$, while it is about $93\%$ for small sample. 
Figure \ref{fig:ci:dgp2} reveals that, in DGP 2, the performance of the DNN estimator and the ORACLE estimator are comparable, even when the sample size is around $500$. When the sample size is large ($n=2000$), the coverage probability stays around the 95\% nominal level. The above findings confirm the validity of our theoretical results.
\begin{figure}[H]
\centering
\begin{minipage}{.45\textwidth}
  \centering
  \includegraphics[width=3.2 in, height=3 in]{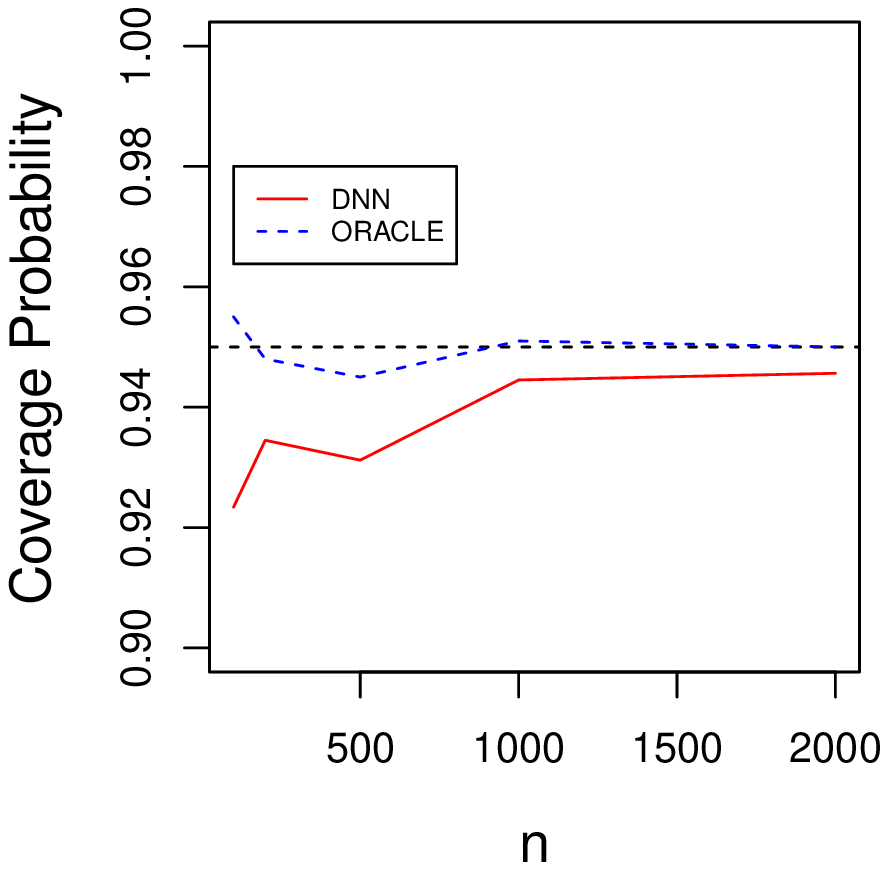}
  \captionof{figure}{Coverage Probability for DGP 1}
  \label{fig:ci:dgp1}
\end{minipage}
\hfill
\begin{minipage}{.45\textwidth}
  \centering
  \includegraphics[width=3.2 in, height=3 in]{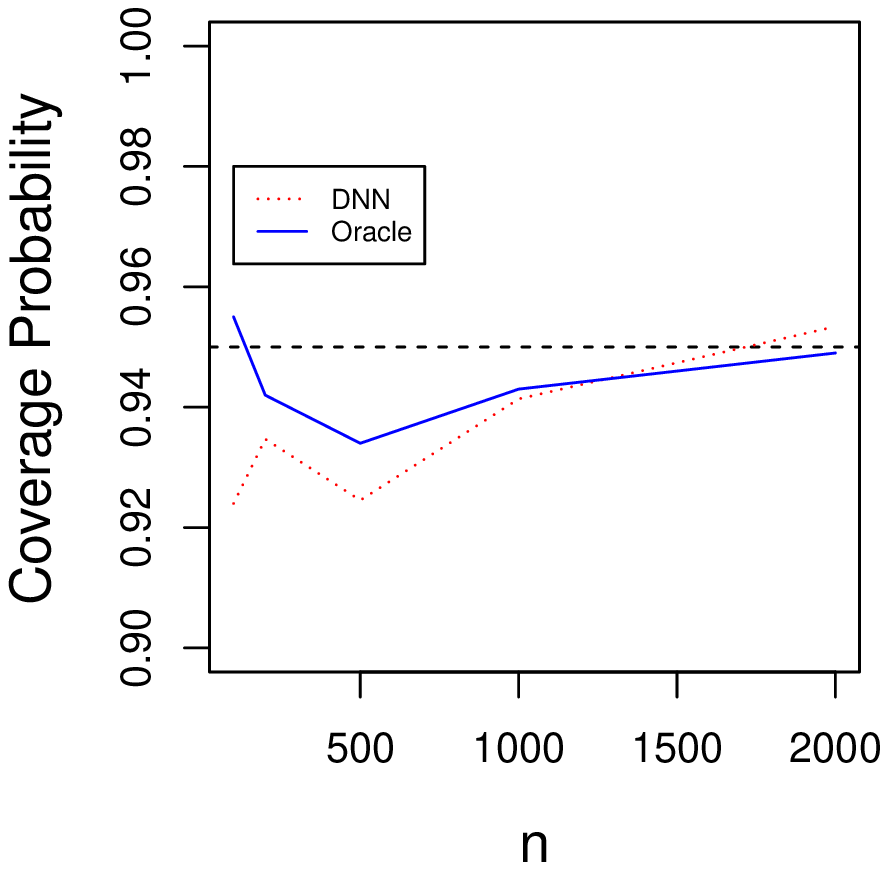}
  \captionof{figure}{Coverage Probability for DGP 2}
    \label{fig:ci:dgp2}
\end{minipage}

\end{figure}

Additional simulation studies were conducted to evaluate the stability of the DNN estimator with different  structures in terms of coverage probability. It can be seen from Figures \ref{fig:rmse:wl:ci1} and \ref{fig:rmse:wl:ci2} that, for various choices of $W$ and $L$, the difference of coverage probabilities is quite small (less than $2\%$).

\begin{figure}[H]
\includegraphics[width=2.1 in]{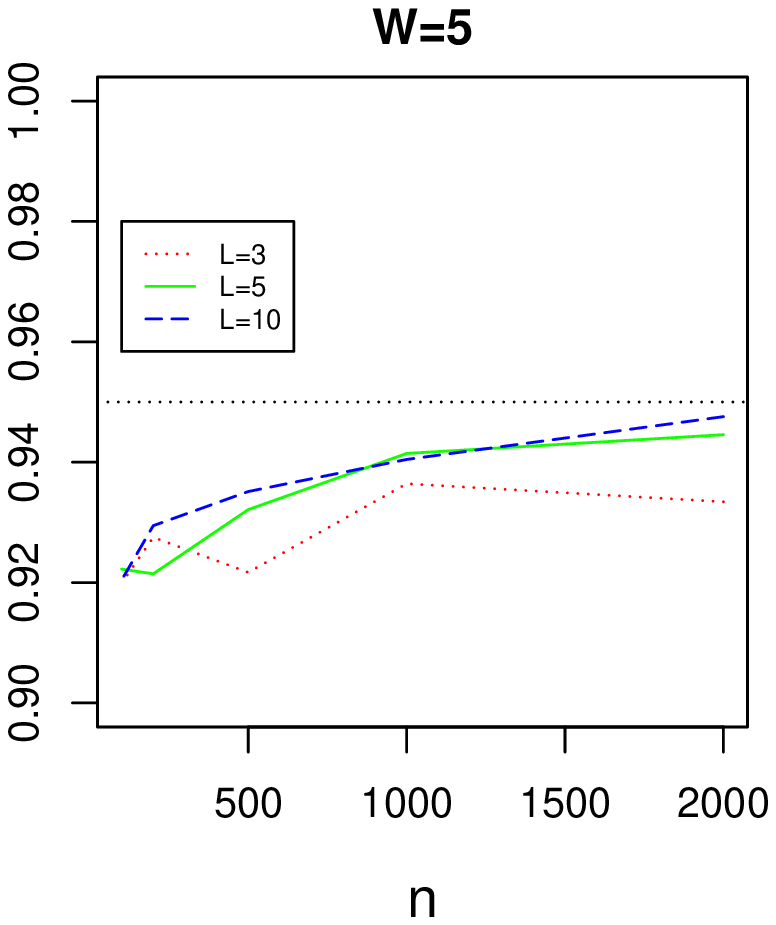}
\includegraphics[width=2.1 in]{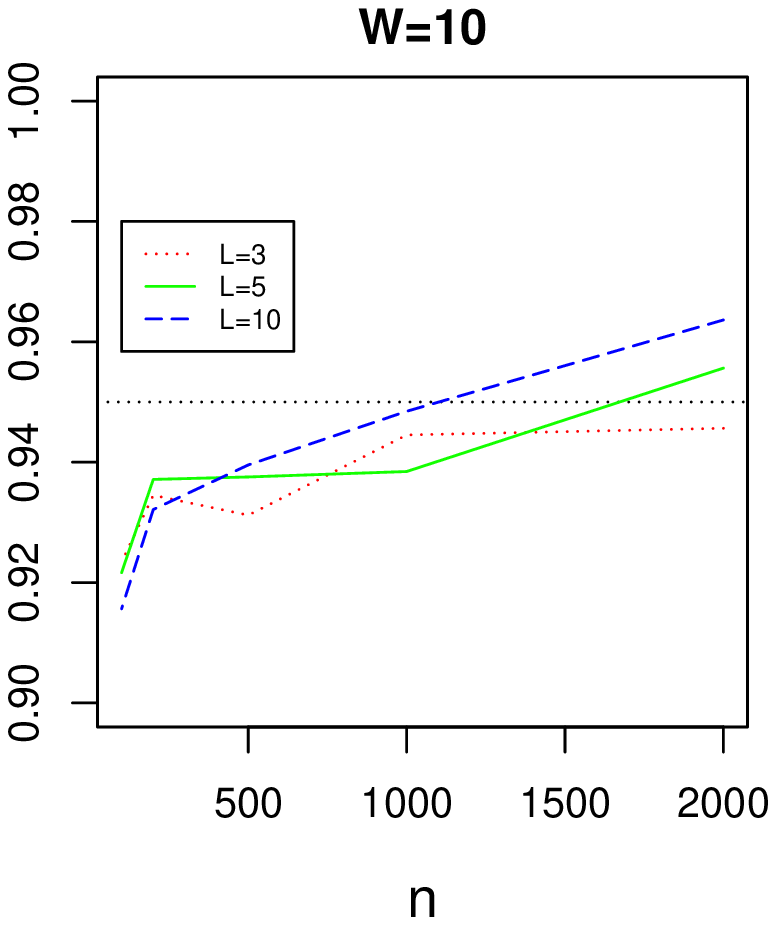}
\includegraphics[width=2.1 in]{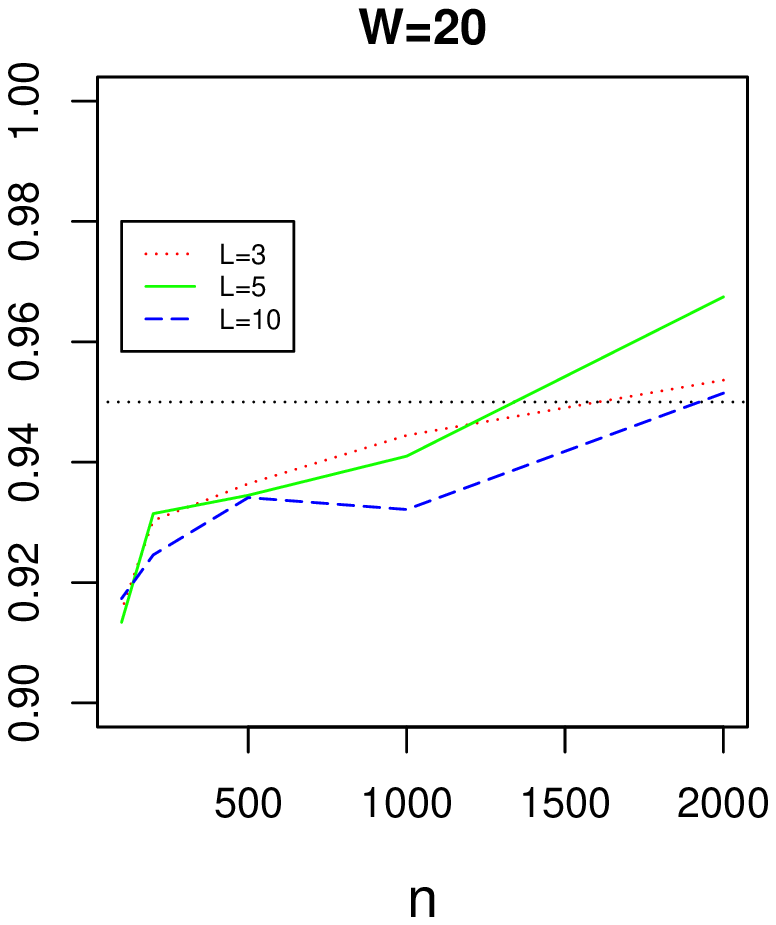}
\caption{Coverage probability with different $(W, L)$ for DGP 1}
\label{fig:rmse:wl:ci1}
\end{figure}
\begin{figure}[H]
\includegraphics[width=2.1 in]{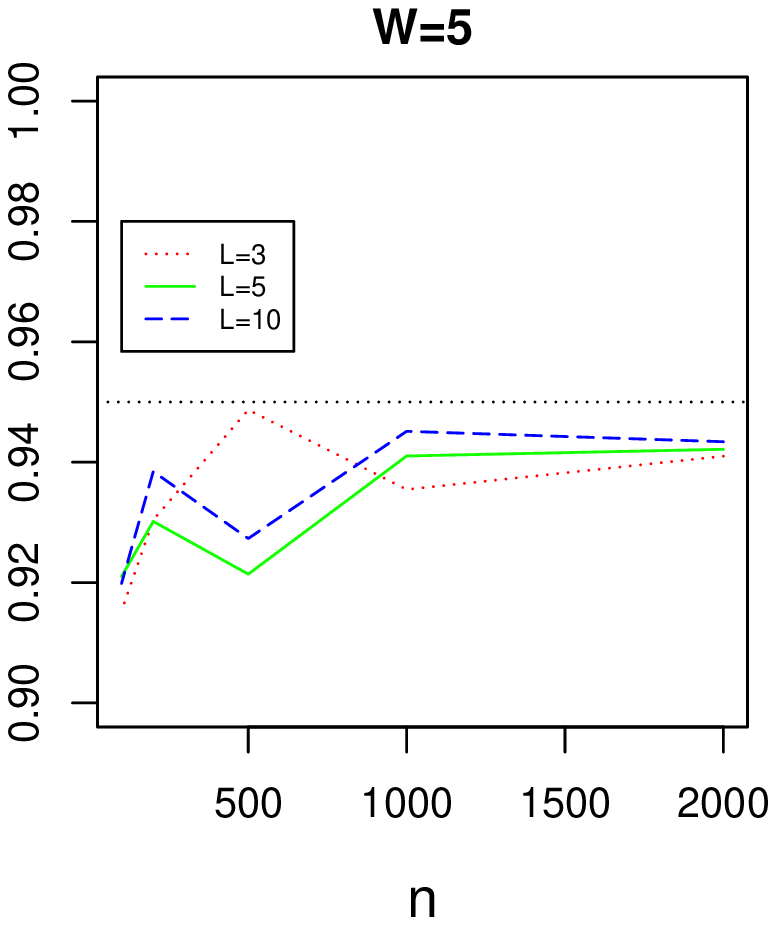}
\includegraphics[width=2.1 in]{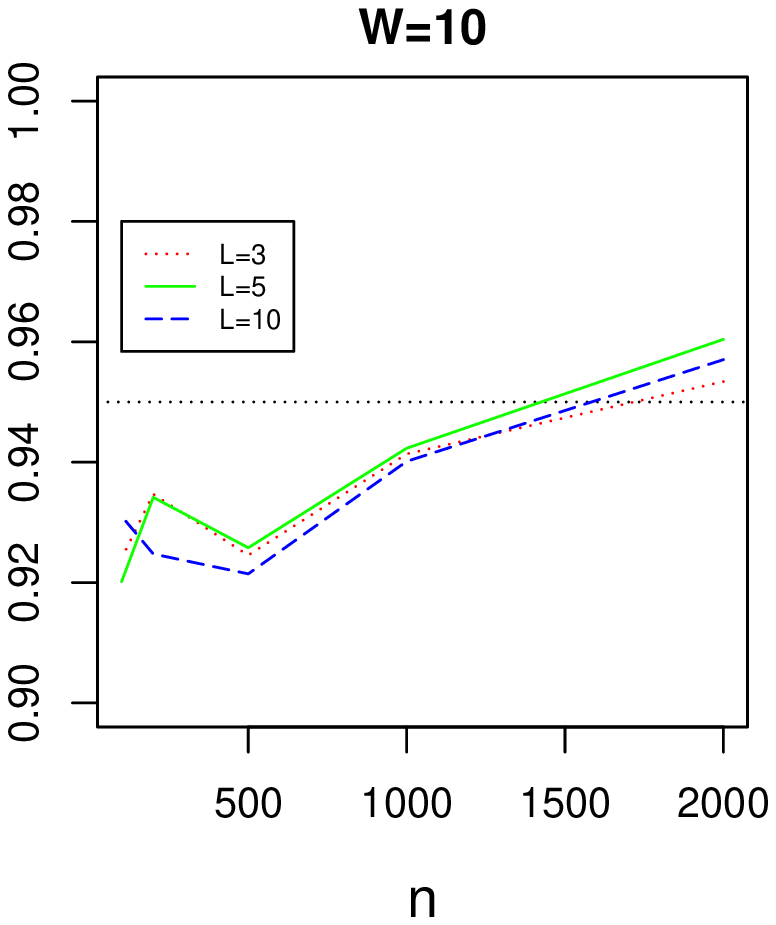}
\includegraphics[width=2.1 in]{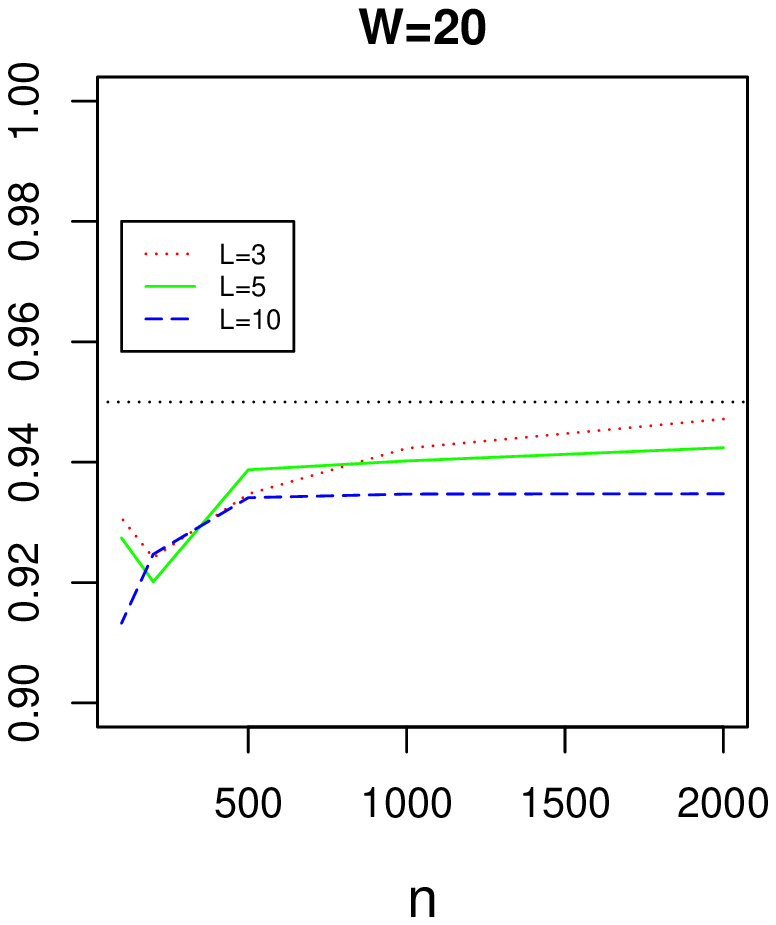}
\caption{Coverage probability with different $(W, L)$ for DGP 2}
\label{fig:rmse:wl:ci2}
\end{figure}

%
%
%
%
\section{Empirical Application}\label{sec:empirical}
In this section, we apply the proposed estimation procedure to study automobile sales and price. Specifically, we consider the following model:
\begin{equation*}
	\textrm{sales}_i=\beta\times \textrm{price}_i+\epsilon,
\end{equation*}
where sales$_i$ and price$_i$ indicate the market sales and price of vehicle of type $i$. The instrumental variables we adopt are (a) a 0-1 valued variable indicating the air conditioning; (b) horsepower divided by weight; (c) miles per dollar measuring the routine costs; (d) size of the vehicle. Similar settings were also considered in \cite{blp95}. After eliminating missing values, we keep 2217 types of automobiles in the dataset. We fit simple OLS, LR, A-Spline, P-Spline, and DNN based on the 2217 observations and summarize the results in Table  \ref{table:automobile:price}. Several interesting findings can be observed. First, all the estimators are highly significant at 1\% significance level. Second, without using instrumental variables, OLS gives an estimator of the coefficient $-0.0840$. Among the rest four instrumental variables estimators, the P-Spline estimator is more elastic than the OLS estimator, while LR, A-Spline, and DNN estimators are less elastic. Finally, we observe that LR and  DNN estimators are almost equal, but the standard deviation of the DNN estimator is slightly smaller.

\begin{table}[htp!]
\centering
\caption{Estimators of automobile price}
\label{table:automobile:price}
\begin{tabular}{cccccc}
\hline\hline
      & OLS                                                        & LR                                                       & A-Spline                                                      & P-Spline                                                       & DNN                                                       \\ \cline{2-6} 
\;\;\;$price$\;\;\;  & \begin{tabular}[c]{@{}c@{}}-0.0840$^{***}$\\ (0.0029)\end{tabular} & \begin{tabular}[c]{@{}c@{}}-0.0804$^{***}$\\ (0.0038)\end{tabular} & \begin{tabular}[c]{@{}c@{}}-0.0814$^{***}$\\ (0.0037)\end{tabular} & \begin{tabular}[c]{@{}c@{}}-0.0909$^{***}$\\ (0.0055)\end{tabular} & \begin{tabular}[c]{@{}c@{}}-0.0805$^{***}$\\ (0.0036)\end{tabular}\\\hline
\end{tabular}
\vspace*{1mm}

Note: *, **, and *** refer to significance at 10\%, 5\% and 1\% level, respectively.
\end{table}



\newpage

\bibliography{ref}{}
\bibliographystyle{apalike}
\newpage

\section*{Appendix}
\setcounter{subsection}{0}
\renewcommand{\thesubsection}{A.\arabic{subsection}}
\setcounter{subsubsection}{0}
\renewcommand{\thesubsubsection}{\textbf{A.\arabic{subsection}.\arabic{subsubsection}}}
\setcounter{equation}{0}
\renewcommand{\theequation}{A.\arabic{equation}}
\setcounter{lemma}{0}
\renewcommand{\thelemma}{A.\arabic{lemma}}
\setcounter{proposition}{0}
\renewcommand{\theproposition}{A.\arabic{proposition}}
For any $k$-dimensional vector $\bfv=(v_1,\ldots, v_k)^\top$ and real number $a$, we denote $\bfv+a=(v_1+a,\ldots, v_k+a)^\top$. Let $\bfW=(\bfX, \bfZ)$ be an independent copy of $\bfW_i=(\bfX_i, \bfZ_i)$ and for any function $f:\mathbb{R}^d \to \mathbb{R}$, we define norms $\|f\|_\infty=\sup_{\bfz}|f(\bfz)|$ and $\|f\|_{n,\infty}=\sup_{1\leq i\leq n}|f(\bfZ_i)|$. For matrix $A$, we define its Frobenius norm $\|A\|_F=\sqrt{\textrm{Tr}(A^\top A)}$. We recall the Rademacher Complexity (see \citealp{w19}) of function class $\mcF$ is defined as
\begin{equation*}
	\mcRn\mcF=\ev\bigg(\sup_{f\in \mcF}\frac{1}{n}\sum_{i=1}^n \eta_i f(\bfZ_i) \bigg| \bfZ_1,\ldots, \bfZ_n\bigg),
\end{equation*}
where $\eta_i$'s are i.i.d. Rademacher random variables which are independent of $\bfZ_i$'s. Moreover, we denote $\textrm{Pdim}(\mcF)$ as the pseudo-dimension of $\mcF$ (see, \citealp{ab09nn}). For functions $f,g: \mathbb{R}^d\to \mathbb{R}$ and random variables $\zeta$ and $\xi$, we denote $\langle f, g\rangle_n=n^{-1}\sum_{i=1}^nf(\bfZ_i)g(\bfZ_i)$,  $\langle \zeta, \xi\rangle_n=n^{-1}\sum_{i=1}^n\zeta_i\xi_i$ and $\langle \zeta, f\rangle_n=n^{-1}\sum_{i=1}^n\zeta_if(\bfZ_i)$, where $\zeta_i, \xi_i$ are the observations of $\zeta, \xi$.

\subsection{Some Preliminary Lemmas}\label{sec:preliminary:lemma}
\begin{lemma}\label{prop:conditional:indpendence:1}
Suppose $(\zeta_i, \bfZ_i), i=1,\ldots, n$ are i.i.d. observations and $\ev(|\zeta_i|)<\infty$, then
$\ev(\zeta_i|\bfZ_1,\ldots, \bfZ_n)=\ev(\zeta_i|\bfZ_i)$ for all $i=1,\ldots, n$.
\end{lemma}
\begin{proof}
We prove the case when $n=2$, and the extension to the case of general $n$ can be made analogically. We define $h_1=\ev(\zeta_1 | \bfZ_1)$, $\mcD=\{A\in \mathcal{S}(\bfZ_1, \bfZ_2) : \ev(\zeta_1 I_A)=\ev(h_1I_A)\}$, and $\mcP=\{A\times B : A\in \mathcal{S}(\bfZ_1), B\in \mathcal{S}(\bfZ_2)\}$. Clearly, $\mcP$ is a $\pi$-system. Since $\ev(|\zeta_1|)<\infty$, by Lebesgue's dominated convergence theorem, it is not difficult to see $\mcD$ is a $\lambda$-system. Moreover, for $A\times B \in \mcP$, we can see
\begin{eqnarray}
	\ev(\zeta_1I_{A\times B})=\ev(\zeta_1I_{A}I_{B})=\ev(\zeta_1I_{A})\ev(I_B)=\ev(h_1I_A)\ev(I_B)=\ev(h_1I_AI_B)=\ev(h_1I_{A\times B}),\nonumber
\end{eqnarray}
where we use the facts that $A\in \mathcal{S}(\bfZ_1)$, $B\in \mathcal{S}(\bfZ_2)$ and their independence. As a consequence of $\pi-\lambda$ monotone class theorem, we have $\mcP\subset \mcD$ and $\mathcal{S}(\bfZ_1, \bfZ_2)=\mathcal{S}(\mcP)\subset \mcD\subset \mathcal{S}(\bfZ_1, \bfZ_2)$. By the definition of $\mcD$, we see that $\ev(\zeta_1 | \bfZ_1, \bfZ_2)=h_1$. 
\end{proof}
\begin{lemma}\label{prop:conditional:indpendence:2}
Suppose $(\zeta_i, \bfZ_i), i=1,\ldots, n$ are i.i.d. observations and $\ev(|\zeta_i|)<\infty$, then the random variables $\zeta_1,\ldots, \zeta_n$ are  conditionally mutually independent and identically distributed given $\bfZ_1,\ldots, \bfZ_n$
\end{lemma}
\begin{proof}
We prove the case when $n=2$, and the extension to the case of general $n$ can be made analogically. For any $a_1, a_2\in \mathbb{R}$, we define $\delta_i=I(\zeta_i\leq a_i)$ for $i=1,2$. Also we denote $h_i=\ev(\delta_i | \bfZ_1, \bfZ_2)$, $\mcD=\{A\in \mathcal{S}(\bfZ_1, \bfZ_2) : \ev(\delta_1\delta_2 I_A)=\ev(h_1h_2I_A)\}$, and $\mcP=\{A\times B : A\in \mathcal{S}(\bfZ_1), B\in \mathcal{S}(\bfZ_2)\}$.  It is not difficult to verify that $\mcP$ is a $\pi$-system and $\mcD$ is a $\lambda$-system. Moreover, due to Lemma \ref{prop:conditional:indpendence:1}, we have $h_i \in \mathcal{S}(Z_i)$. Therefore, for $A\times B \in \mcP$, we can verify
\begin{align*}
\ev(\delta_1\delta_2I_{A\times B})=\ev(\delta_1\delta_2I_{A}I_{B})=\ev(\delta_1I_{A})\ev(\delta_2I_{B})
&=\ev(h_1I_{A})\ev(h_2I_{B})\nonumber\\
&=\ev(h_1h_2I_{A}I_{B})\nonumber\\
&=\ev(h_1h_2I_{A\times B}),\nonumber
\end{align*}
where we use the facts that $A\in \mathcal{S}(\bfZ_1)$, $B\in \mathcal{S}(\bfZ_2)$ and their independence. As a consequence of $\pi-\lambda$ monotone class theorem, we have $\mcP\subset \mcD$ and $\mathcal{S}(\bfZ_1, \bfZ_2)=\mathcal{S}(\mcP)\subset \mcD\subset \mathcal{S}(\bfZ_1, \bfZ_2)$. By the definition of $\mcD$ and $h_i$, we see that $\pr(\zeta_1\leq a_1, \zeta_2\leq a_2 | \bfZ_1, \bfZ_2)=h_1h_2=\pr(\zeta_1\leq a_1| \bfZ_1, \bfZ_2)\pr(\zeta_2\leq a_2| \bfZ_1, \bfZ_2)$. 
\end{proof}
\begin{lemma}\label{lemma:contraction:inequality}
Let $\mcF$ be a class of functions and $\phi: \mathbb{R}\to \mathbb{R}$ be a Lipschitz function with  Lipschitz constant $K$, then 
\begin{eqnarray}
	\ev(\mcRn  \phi\circ\mcF)\leq K\ev(\mcRn\mcF),\nonumber
\end{eqnarray}
where $\phi\circ\mcF=\{\phi\circ f: f\in \mcF\}$.
\end{lemma}
\begin{proof}
This is Proposition 5.28 in \cite{w19}.
\end{proof}
\begin{lemma}\label{lemma:concentration:inequality}
Let $\mcF$ be a class of functions such that $|f|\leq B$ and $\textrm{Var}(f)\leq V$ for all $f\in \mcF$ and some $B, V \in \mathbb{R}$. Then for any $\eta>0$, with probability at least $1-2e^{-\eta}$,
\begin{eqnarray}
	\sup_{f\in \mcF}\bigg|(\pr-\pr_n)[f]\bigg|\leq 3\ev(\mcRn\mcF)+\sqrt{\frac{2V\eta}{n}}+\frac{4B\eta}{3n}.\nonumber
\end{eqnarray}  
\end{lemma}
\begin{proof}
This is Theorem 2.1 in \cite{bbm05}.
\end{proof}
\begin{lemma}\label{lemma:Massart:inequality}
Suppose $(\zeta_i, \bfZ_i), i=1,\ldots, n$ are i.i.d. observations with $\ev(\zeta_i | \bfZ_i)=0$. Let $\mcF$ be a class of function with finite elements, and  conditioning on $\mathbb{Z}$, $\|f\|_n\leq r$ for all $f\in \mcF$ and some $r>0$. Furthermore, if $\zeta_1,\zeta_2, \ldots, \zeta_n$ are conditionally independent given $\mathbb{Z}$ and $|\zeta_i|\leq B$ for some $B>0$, then 
\begin{equation*}
\ev_{\mathbb{Z}}(\sup_{f\in \mcF}\langle f, \zeta\rangle_n)\leq Br\sqrt{\frac{2\log(|\mcF|)}{n}}
\end{equation*}
\end{lemma}
\begin{proof}
Conditioning on $\mathbb{Z}$, $\langle f, \zeta\rangle_n$ is $n^{-1/2}B\|f\|_n$-subgaussian due to the conditional independence of $\zeta_1,\ldots, \zeta_n$. Furthermore, for any $\lambda>0$, it follows from Jensen’s inequality that
\begin{eqnarray*}
\exp\bigg(\lambda\ev_{\mathbb{Z}}(\sup_{f\in \mcF}\langle f, \zeta\rangle_n)\bigg)&\leq& \ev_{\mathbb{Z}}\bigg(\exp(\lambda\sup_{f\in \mcF}\langle f, \zeta\rangle_n)\bigg)\nonumber\\
&=& \ev_{\mathbb{Z}}\bigg(\sup_{f\in \mcF}\exp(\lambda\langle f, \zeta\rangle_n)\bigg)\nonumber\\
&\leq& \ev_{\mathbb{Z}}\bigg(\sum_{f\in \mcF}\exp(\lambda\langle f, \zeta\rangle_n)\bigg)\nonumber\\
&=& \sum_{f\in \mcF}\ev_{\mathbb{Z}}\bigg(\exp(\lambda\langle f, \zeta\rangle_n)\bigg)\nonumber\\
&\leq&  \sum_{f\in \mcF}\exp\bigg(\frac{\lambda^2B^2\|f\|_n^2}{2n}\bigg)\leq |\mcF|\exp\bigg(\frac{\lambda^2B^2r^2}{2n}\bigg).
\end{eqnarray*}
As a consequence, it follows that
\begin{eqnarray*}
	\ev_{\mathbb{Z}}(\sup_{f\in \mcF}\langle f, \zeta\rangle_n)\leq \frac{\lambda B^2r^2}{2n}+\frac{\log(|\mcF|)}{\lambda},\textrm{ for all } \lambda>0.
\end{eqnarray*}
We complete the proof by setting $\lambda=\frac{\sqrt{2n\log(|\mcF|)}}{Br}$.
\end{proof}

\begin{lemma}\label{lemma:dudley:inequality}
Suppose $(\zeta_i, \bfZ_i), i=1,\ldots, n$ are i.i.d. observations with $\ev(\zeta_i | \bfZ_i)=0$. Let $\mcF$ be a class of functions such that conditioning on $\mathbb{Z}$, $\|f\|_n\leq r$ for all $f\in \mcF$ and some $r>0$. Furthermore, if $\zeta_1,\zeta_2, \ldots, \zeta_n$ are conditionally independent given $\mathbb{Z}$ and $|\zeta_i|\leq B$ for some $B>0$, then it follows that
\begin{eqnarray}
	\ev_{\mathbb{Z}}(\sup_{f\in \mcF_r}\langle f, \zeta \rangle_n)\leq \inf_{0<x<r}\bigg\{4x\sqrt{\frac{1}{n}\sum_{i=1}^n\ev(\zeta_i^2 |\bfZ_i)}+12B\int_{x}^{r}\sqrt{\frac{\log \mcN(u, \mcF_r, \|\cdot\|_n)}{n}}du\bigg\}.\nonumber
\end{eqnarray}
As a consequence, we have
\begin{eqnarray}
	\mcRn\mcF_r\leq \inf_{0<x<r}\bigg\{4x+12\int_{x}^{r}\sqrt{\frac{\log \mcN(u, \mcF_r, \|\cdot\|_n)}{n}}du\bigg\}.\nonumber
\end{eqnarray}
\end{lemma}
\begin{proof}
Let $\alpha_j=2^{-j}r$ and $T_j$ be the a proper $alpha_j$-covering of $\mcF_r$ with respect to $\|\cdot\|_n$. For $f\in \mcF_r$, denote $\tau_j(f)\in T_j$ as the function such that $\|\tau_j(f)-f\|_n\leq \alpha_j$. For integer $N>0$, we have
\begin{equation*}
	f=f-\tau_N(f)+\sum_{j=1}^N(\tau_j(f)-\tau_{j-1}(f)),
\end{equation*}
here we denote $g_0=0$ for simplicity. Therefore, it follows that
\begin{eqnarray}
	\langle f, \zeta \rangle_n&\leq&	\langle f-\tau_N(f), \zeta \rangle_n+\sum_{j=1}^N \langle \tau_j(f)-\tau_{j-1}(f), \zeta \rangle_n\nonumber\\
	&\leq& \|\zeta\|_n \|f-\tau_N(f)\|_n+\sum_{j=1}^N \langle \tau_j(f)-\tau_{j-1}(f), \zeta \rangle_n\nonumber\\
	&\leq&  \alpha_N\|\zeta\|_n+\sum_{j=1}^N \langle \tau_j(f)-\tau_{j-1}(f), \zeta \rangle_n. \label{eq:lemma:dudley:inequality:eq1}
\end{eqnarray}
Now notice that
\begin{eqnarray*}
	\|\tau_j(f)-\tau_{j-1}(f)\|_n\leq \|\tau_j(f)-f\|_n+\|\tau_{j-1}(f)-f\|_n\leq \alpha_j+\alpha_{j-1}= 3\alpha_j.
\end{eqnarray*}
Therefore, apply  Lemma \ref{lemma:Massart:inequality} to the class $\{f_j-f_{j-1}: \|f_j-f_{j-1}\|_n\leq 3\alpha_j, f_j\in T_j, f_{j-1}\in T_{j-1}\}$, it follows that
\begin{eqnarray}
	\ev_{\mathbb{Z}}\bigg(\sup_{f\in \mcF_r}\langle \tau_j(f)-\tau_{j-1}(f), \zeta \rangle_n\bigg)&\leq&   3\alpha_j B\sqrt{\frac{2\log(|T_j||T_{j-1}|)}{n}}\nonumber\\
	&\leq& 6\alpha_j B\sqrt{\frac{\log(|T_j|)}{n}}\nonumber\\
	&\leq& 12B(\alpha_j-\alpha_{j+1})\sqrt{\frac{\log(|T_j|)}{n}}\nonumber\\
	&\leq& 12B(\alpha_j-\alpha_{j+1})\sqrt{\frac{\log \mcN(\alpha_j, \mcF_r, \|\cdot\|_n)}{n}}\nonumber\\
	&\leq&12B\int_{\alpha_{j+1}}^{\alpha_j}\sqrt{\frac{\log \mcN(u, \mcF_r, \|\cdot\|_n)}{n}}du.\nonumber
\end{eqnarray}
Above inequality and (\ref{eq:lemma:dudley:inequality:eq1}) together imply that
\begin{eqnarray}
	\ev_{\mathbb{Z}}(\sup_{f\in \mcF_r}\langle f, \zeta \rangle_n)\leq \alpha_N\ev_\bfZ(\|\zeta\|_n)+12B\int_{\alpha_{N+1}}^{\alpha_0}\sqrt{\frac{\log \mcN(u, \mcF_r, \|\cdot\|_n)}{n}}du.\nonumber
\end{eqnarray}
Now pick $N$ such that, $\alpha_N> 2\epsilon$ and $\alpha_{N+1}\leq 2\epsilon$. Therefore, $\alpha_N=2\alpha_{N+1}\leq 4\epsilon$ and $\alpha_{N+1}=\alpha_N/2>\epsilon$. Hence, we conclude that
\begin{eqnarray}
	\ev_{\mathbb{Z}}(\sup_{f\in \mcF_r}\langle f, \zeta \rangle_n)\leq 4\epsilon\sqrt{\frac{1}{n}\sum_{i=1}^n\ev(\zeta_i^2 |\bfZ_i)}+12B\int_{\epsilon}^{r}\sqrt{\frac{\log \mcN(u, \mcF_r, \|\cdot\|_n)}{n}}du,\nonumber
\end{eqnarray}
where we use  the inequality $\ev_{\mathbb{Z}}(\|\zeta\|_n)\leq \sqrt{\ev_{\mathbb{Z}}(\|\zeta\|_n^2)}$ and Lemma \ref{prop:conditional:indpendence:1}. Since $\epsilon>0$ is arbitrary, we finish the proof.
\end{proof}

\begin{lemma}\label{lemma:entroy:bound}
For a class of function $\mcF$, if $n\geq \textrm{Pdim}(\mcF)$ and $\max_{1\leq i\leq  n}\|f(\bfZ_i)\| \leq C$ for all $f \in \mcF$ and some constant $C>0$, then it follows that
\begin{equation*}
	\log N(x, \mcF, \|\cdot\|_{n,\infty})\leq \textrm{Pdim}(\mcF)\log\bigg(\frac{2enC}{x\textrm{Pdim}(\mcF)}\bigg)\leq \textrm{Pdim}(\mcF)\log\bigg(\frac{2enC}{x}\bigg).
\end{equation*}
Moreover for deep neural network class, it holds that
\begin{equation*}
	\textrm{Pdim}(\mcF_{d,1}(L,W))\leq \textrm{VC}(\mcF_{d+1,1}(L+1,W+1))\leq cL(LW^2+Wd)\log(LW^2+Wd),
\end{equation*}
for some universal constant .
\end{lemma}
\begin{proof}
The first result is Theorem 12.2 in \cite{ab09nn}. The second one follows from Theorem 14.1 in \cite{ab09nn} and
 Theorem 6 in \cite{bhlm19}.
\end{proof}
\begin{lemma}\label{proposition:cauchy:F:norm}
	Let $\bfa_1, \bfb_2,\ldots, \bfa_n, \bfb_n$ be real vectors of the same dimension. Then it follows that
	\begin{eqnarray}
		\bigg\|\frac{1}{n}\sum_{i=1}^n \bfa_i\bfb_i^\top\bigg\|_F^2\leq \frac{1}{n}\sum_{i=1}^n\|\bfa_i\|_2^2\times \frac{1}{n}\sum_{i=1}^n\|\bfb_i\|_2^2.\nonumber
	\end{eqnarray}
\end{lemma}
\begin{proof}
Denote $\bfa_i=(a_{i1},\ldots, a_{ik})^\top$ and $\bfb_i=(b_{i1},\ldots, b_{ik})^\top$ for $i=1,\ldots, n$. By definition of Frobenius norm and Cauchy–Schwarz inequality, we have
\begin{eqnarray}
	\bigg\|\sum_{i=1}^n \bfa_i\bfb_i^\top\bigg\|_F^2=\sum_{j=1}^k\sum_{s=1}^k \bigg(\sum_{i=1}^na_{ij}b_{is}\bigg)^2\leq\sum_{j=1}^k\sum_{s=1}^k  \sum_{i=1}^na_{ij}^2\sum_{i=1}^nb_{is}^2=\sum_{i=1}^n\|\bfa_i\|_2^2\times \sum_{i=1}^n\|\bfb_i\|_2^2.\nonumber
\end{eqnarray}
\end{proof}

\subsection{DNN Approximation}\label{sec:dnn:approximation}
The proof of approximation results of DNN, we mainly rely on the result in \cite{lsyz20}. We borrow the following notation from their paper. For given $K\in \mathbb{Z}_+$ and $\delta>0$ with $\delta<\frac{1}{K}$, define the  trifling region of $[0, 1]^d$ as follows:
\begin{eqnarray}
	\Omega(K, \delta, d):=\bigcup_{i=1}^d\bigg\{\bfz=(z_1,\ldots, z_d)^\top : x_i\in \bigcup_{k=1}^{K-1}\bigg(\frac{k}{K}-\delta, \frac{k}{K}\bigg)\bigg\}.
\end{eqnarray}
\begin{lemma}\label{lemma:multiplication:network}
For any $N, L\in \mathbb{Z}_+$, there exists a network $\phi_{\times}\in \mcF_{2,1}(L+1,9N)$ such that 
\begin{enumerate}
\item $|\phi_{\times}(x,y)-xy|\leq 12N^{-L}$ for all $x,y\in [0,1]$;
\item $|\phi_{\times}(x,y)|\leq 1$ for all $x,y\in [0,1]$.
\end{enumerate}
\end{lemma}
\begin{proof}
By Lemma 5.2 in \cite{lsyz20}, there exists a $\widetilde{\phi}_{\times} \in \mcF_{2,1}(L, 9N)$ such that
\begin{eqnarray}
		|\widetilde{\phi}_{\times}(x,y)-xy|\leq 6N^{-L}, \;\;\textrm{ for all } x,y\in [0,1].\label{eq:lemma:multiplication:network:eq1}
\end{eqnarray}
Let $\delta:=6N^{-L}$, and we further define
\begin{eqnarray}
	\phi_{\times}(x,y)=\frac{\sigma(\widetilde{\phi}_{\times}(x,y))}{1+\delta}.\nonumber
\end{eqnarray}
Direct examination implies
\begin{eqnarray}
	|\phi_{\times}(x,y)-xy|=\bigg|\frac{\sigma(\widetilde{\phi}_{\times}(x,y))}{1+\delta}-xy\bigg|=\begin{cases}
	|xy|&\textrm{ if } \widetilde{\phi}_{\times}(x,y)<0\\
	\bigg|\frac{\widetilde{\phi}_{\times}(x,y)}{1+\delta}-xy\bigg| &\textrm{ if } \widetilde{\phi}_{\times}(x,y)\geq 0
	\end{cases}.\nonumber
\end{eqnarray}
Notice by (\ref{eq:lemma:multiplication:network:eq1}), if $ \widetilde{\phi}_{\times}(x,y)<0$ implies $|xy|\leq 6N^{-L}=\delta$, and
\begin{eqnarray}
	\bigg|\frac{\widetilde{\phi}_{\times}(x,y)}{1+\delta}-xy\bigg|=\bigg|\frac{\widetilde{\phi}_{\times}(x,y)-xy-\delta xy}{1+\delta}\bigg|\leq \frac{2\delta}{1+\delta}\leq 2\delta.\nonumber
\end{eqnarray}
Therefore, we show that $|\widetilde{\phi}_{\times}(x,y)-xy|\leq 2\delta=12N^{-L}$ for all $x,y\in [0, 1]$. Moreover, since $\widetilde{\phi}_{\times}(x,y)\leq 1+6N^{-L}\leq 1+\delta$, it holds that $\phi_{\times}(x,y)\leq 1$. Finally, it is trivial to see that $\phi_{\times}(x,y) \in \mcF_{2,1}(L+1, 9N)$.
\end{proof}

\begin{lemma}\label{lemma:multiplication:d:approximation}
For any $N, L, d\in \mathbb{Z}_+$, there exists a network $\phi_{\times}^d\in \mcF_{d,1}((L+1)(d-1)+d-2,9N+d-2)$ such that 
\begin{enumerate}
\item $|\phi_{\times}^d(\bfz)-\prod_{i=1}^d z_i|\leq 12(d-1)N^{-L}$ for all $\bfz=(z_1,\ldots, z_d)^\top \in [0, 1]^d$;
\item $|\phi_{\times}^d(\bfz)|\leq 1$ for all $\bfz=(z_1,\ldots, z_d)^\top \in [0, 1]^d$.
\end{enumerate}
\end{lemma}
\begin{proof}
For simplicity, we prove the case when $d=3$. By Lemma \ref{lemma:multiplication:network}, there exist a network $\phi_{\times}\in \mcF_{2,1}(L+1,9N)$ such that $|\phi_{\times}(z_1,z_2)-z_1z_2|\leq 12N^{-L}$. To pass the value of $z_3$ to next layer, we can add one more channel to store the $z_3$. Therefore,  we use $\phi_{\times}(\phi_{\times}(z_1,z_2),z_3)$ to approximate the products, and the error can be calculated as follows:
\begin{eqnarray}
	|\phi_{\times}(\phi_{\times}(z_1,z_2),z_3)-z_1z_2z_3|&\leq& |\phi_{\times}(\phi_{\times}(z_1,z_2),z_3)-\phi_{\times}(z_1,z_2)z_3|+|\phi_{\times}(z_1,z_2)z_3-z_1z_2z_3|\nonumber\\
	&\leq&12N^{-L}+12N^{-L}\leq 24N^{-L}.\nonumber
\end{eqnarray}
Finally, it is not difficult to verify that this neural network has $2(L+1)+1$ hidden layers and width $9N+1$.
\end{proof}

\begin{lemma}\label{lemma:multinomial:approximation}
For any $N, L\in \mathbb{Z}_+$ and $\bfgamma=(\gamma_1,\ldots, \gamma_d)\in \mathbb{Z}^{d}$, there exists a network $P_{\bfgamma}\in \mcF_{d,1}$ with $(L+2)(|\bfgamma|-1)$ hidden layers and width $9N+|\bfgamma|-2$ such that 
\begin{enumerate}
\item $|P_{\bfgamma}(\bfz)-\bfz^{\bfgamma}|\leq 12(|\bfgamma|-1)N^{-L}$ for all $\bfz=(z_1,\ldots, z_d)^\top \in [0, 1]^d$;
\item $|P_{\bfgamma}(\bfz)|\leq 1$ for all $\bfz=(z_1,\ldots, z_d)^\top \in [0, 1]^d$.
\end{enumerate}
\end{lemma}
\begin{proof}
First step we pass the input value $\bfz$ to next layer as follows:
\begin{eqnarray}
	\bfz\to (\underbrace{z_1,\ldots, z_1}_{\gamma_1 \textrm{ times }}, \underbrace{z_2,\ldots, z_2}_{\gamma_2 \textrm{ times }},\ldots,  \underbrace{z_{d-1},\ldots, z_{d-1}}_{\gamma_{d-1} \textrm{ times }}, \underbrace{z_d,\ldots, z_d}_{\gamma_{d} \textrm{ times }}).\nonumber
\end{eqnarray}
Next we apply the neural network defined in Lemma \ref{lemma:multiplication:d:approximation} to obtain the desired result.  Clearly, this neural network consistent of $(L+1)(|\bfgamma|-1)+|\bfgamma|-1$ hidden layers and width $9N+|\bfgamma|-2$.
\end{proof}

\begin{lemma}\label{lemma:mapping:k:K}
For any $N,L, d\in \mathbb{Z}_+$ and $\delta>0$ with $K=\floor{N^{1/d}}^2\floor{L^{2/d}}$ and $\delta\leq \frac{1}{3K}$, there exits a neural network $\psi_{\textrm{map}}\in \mcF_{1,1}(4L+4, 4N+5)$ such that
\begin{eqnarray}
	\psi_{\textrm{map}}(x)=\frac{k}{K}, \textrm{ if } x\in \bigg[\frac{k}{K},\frac{k+1}{K}-\delta I(x<K-1)\bigg] \textrm{ for } k=0,1,\ldots, K-1.\nonumber
\end{eqnarray}
\end{lemma}
\begin{proof}
This is Proposition 4.3 in \cite{lsyz20}. 
\end{proof}

\begin{lemma}\label{lemma:point:fitting}
Given any $N, L, m \in \mathbb{Z}_+$ and $\xi_i\in [0, 1]$ for $i=0, 1,\ldots, N^2L^2-1$, there exists a neural network $\phi_{\textrm{fit}}$ with depth $(5L+8)\ceil{\log_2(2L)}$ and width $8m(2N+3)\ceil{\log_2(4N)}$ such that
\begin{enumerate}
\item $|\phi_{\textrm{fit}}(i)-\xi_i|\leq N^{-2m}L^{-2m}$, for $i=0, 1,\ldots, N^2L^2-1$;
\item $0\leq \phi_{\textrm{fit}}(x)\leq 1$ for all $x\in \mathbb{R}$.
\end{enumerate}
\end{lemma}
\begin{proof}
This is Proposition 4.4 in \cite{lsyz20}. 
\end{proof}

\begin{lemma}\label{lemma:approximation:p:c:smooth}
Suppose $f: [0, 1]^d \to \mathbb{R}$ be $(p, 1)$-H\" older smooth for some $p>0$, then for all positive integers $M, N$, there exists a neural network $g\in \mcF_{d,1}(L, W)$ with 
\begin{equation*}
	L=18\ceil{p}^{2}(M+3)\ceil{\log_2(2M)}
\end{equation*}
and 
\begin{equation*}
	W=40(\ceil{p}+d+2)^{d+1}(N+2)\ceil{\log_2(4N)}
\end{equation*}
such that 
\begin{equation*}
	|f(\bfz)-g(\bfz)|\leq  (86+d8^p)(p+1)^{d+1}N^{-2p/d}M^{-2p/d}\quad \textrm{ for all }\; \bfz\in [0, 1]^d\backslash \Omega(K, \delta, d),
\end{equation*}
where $K=\floor{N^{1/d}}^2\floor{M^{2/d}}$ and $0<\delta<\frac{1}{3K}$ can be chosen arbitrary.
\end{lemma}
\begin{proof}
\noindent \textbf{Step 1:} For notational simplicity, we denote $m=\floor{p}$, the largest integer strictly smaller than $p$. 
For each $\bftheta=(\theta_1, \ldots, \theta_d)^\top \in \{0, 1, \ldots, K-1\}^d$, we define 
\begin{eqnarray}
	Q_{\bftheta}:=\bigg\{\bfz=(z_1, \ldots, z_d)^\top: z_i\in \bigg[\frac{\theta_i}{K}, \frac{\theta_i+1}{K}-\delta I(\theta_i<K-1) \bigg], i=1,\ldots, d\bigg\}.\nonumber
\end{eqnarray}
Clearly, $[0,1]^d\backslash \Omega(K, \delta, d)=\cup_{\bftheta}Q_{\bftheta}$.
By Lemma \ref{lemma:mapping:k:K}, there exists a neural network $\psi_{\textrm{map}}\in \mcF_{1,1}(4M+4, 4N+5)$ such that
\begin{eqnarray}
	\psi_{\textrm{map}}(z)=\frac{k}{K}, \textrm{ if } z\in \bigg[\frac{k}{K},\frac{k+1}{K}-\delta I(x<K-1)\bigg] \textrm{ for } k=0,1,\ldots, K-1.\nonumber
\end{eqnarray}
Therefore, by parallelizing $d$ above networks, we can obtain a neural network
\begin{eqnarray}
	\bm{\psi}_{\textrm{map}}(\bfz):=(\psi_{\textrm{map}}(z_1),\ldots, \psi_{\textrm{map}}(z_d))^\top \in \mcF_{d, d}(4M+4, d(4N+5))\nonumber
\end{eqnarray} 
such that
\begin{eqnarray}
	\bm{\psi}_{\textrm{map}}(\bfz)=\frac{\bftheta}{K}, \textrm{ if } \bfz\in Q_{\bftheta} \textrm{ for all } \bftheta\in \{0, 1,\ldots, K-1\}^d.\label{eq:lemma:approximation:p:c:smooth:step:1:conclusion}
\end{eqnarray}
\noindent \textbf{Step 2:}  For all $\bfgamma=(\gamma_1, \ldots, \gamma_d)^\top \in \mathbb{Z}^d$ with $|\bfgamma|\leq m$, by Lemma \ref{lemma:multinomial:approximation}, there exists a neural network $P_{\bfgamma}\in \mcF_{d, 1}$ with depth $(2(m+1)(M+1)+2)(m-1)$ and width $9N+m-2$ such that
\begin{eqnarray}
	|P_{\bfgamma}(\bfz)-\bfz^{\bfgamma}|\leq 12(m-1)N^{-2(m+1)(M+1)} \; \textrm{ for all } \bfz=(z_1,\ldots, z_d)^d \in [0, 1]^d.\label{eq:lemma:approximation:p:c:smooth:step:2:eq0}
\end{eqnarray}

Moreover, for each $i=0, 1,\ldots, K^d-1$, we define mapping 
\begin{eqnarray}
	\bm{\vartheta}(i):=(\vartheta_1,\ldots, \vartheta_d)^\top \in \{0, 1, \ldots, K-1\}^d \label{eq:lemma:approximation:p:c:smooth:step:2:eq1}
\end{eqnarray}
such that $\sum_{j=1}^d\vartheta_jK^{j-1}=i$. This can be done, since this is a bijection. For each $\bfgamma=(\gamma_1, \ldots, \gamma_d)^\top \in \mathbb{Z}^d$ with $|\bfgamma|\leq m$, define
\begin{eqnarray}
	\xi_{\bfgamma,i}=\bigg[\partial^{\bfgamma}f\bigg(\frac{\bm{\vartheta(i)}}{K}\bigg)+1\bigg]/2\; \textrm{ for } i=0,1,\ldots, K^d-1.\label{eq:lemma:approximation:p:c:smooth:step:2:eq1.5}
\end{eqnarray}
We can see that $\xi_{\bfgamma, i}\in [0, 1]$, as $f$ is $(p, 1)$-H\" older smooth.
By Lemma \ref{lemma:point:fitting} and the fact that $K^d\leq N^2M^2$, there exist a  neural network  $\phi_{\bfgamma}$ with depth $(5M+8)\ceil{\log_2(2M)}$ and width $8(m+1)(2N+3)\ceil{\log_2(4N)}$ that
\begin{eqnarray}
	|\phi_{\bfgamma}(i)-\xi_{\bfgamma, i}|\leq N^{-2(m+1)}M^{-2(m+1)}\; \textrm{ for all } i=0,1,\ldots, K^d-1 \textrm{ and } |\bfgamma|\leq m.\label{eq:lemma:approximation:p:c:smooth:step:2:eq2}
\end{eqnarray}

\noindent \textbf{Step 3:} Using above networks, we further define
\begin{eqnarray}
	\Phi_{\bfgamma}(\bfz)=\phi_{\bfgamma}(\sum_{j=1}^dz_jK^j)\; \textrm{ for all } \bfz=(z_1,\ldots, z_d)^\top \in [0, 1]^d.\label{eq:lemma:approximation:p:c:smooth:step:3:eq1}
\end{eqnarray}
Clearly, $\Phi_{\bfgamma}\in \mcF_{d,1}$ has depth at most $(5M+8)\ceil{\log_2(2M)}+2$ and width at most $8(m+1)(2N+3)\ceil{\log_2(4N)}+d$.

Notice if $\bfz \in Q_{\bftheta}$ for some $\bftheta=(\theta_1,\ldots, \theta_d)^\top \in \{0, 1,\ldots, K-1\}^d$, then by (\ref{eq:lemma:approximation:p:c:smooth:step:1:conclusion}) and  (\ref{eq:lemma:approximation:p:c:smooth:step:2:eq1}), we have following mapping:
\begin{eqnarray}
	\bfz \xrightarrow{\bm{\psi}_{\textrm{map}}}\frac{\bftheta}{K} \xrightarrow{\Phi_{\bfgamma}}\phi_{\bfgamma}(i),\nonumber
\end{eqnarray}
where the integer $i=\sum_{j=1}^d\theta_jK^{j-1}$ by (\ref{eq:lemma:approximation:p:c:smooth:step:2:eq1}).  As a consequence of  (\ref{eq:lemma:approximation:p:c:smooth:step:2:eq1.5}) and  (\ref{eq:lemma:approximation:p:c:smooth:step:2:eq2}), it follows that
\begin{eqnarray}
	\bigg|\Phi_{\bfgamma}\circ \bm{\psi}_{\textrm{map}}(\bfz)-\bigg[\partial^{\bfgamma}f\bigg(\frac{\bftheta}{K}\bigg)+1\bigg]/2\bigg|\leq N^{-2(m+1)}M^{-2(m+1)}\;\; \textrm{ if } \bfz \in Q_{\bftheta}.\label{eq:lemma:approximation:p:c:smooth:step:3:conclusion}
	\end{eqnarray}
Finally, we define our network to be 
\begin{eqnarray}
	g(\bfz)=2\sum_{|\bfgamma|\leq m}\phi_{\times}\bigg(\frac{\Phi_{\bfgamma}\circ \bm{\psi}_{\textrm{map}}(\bfz)}{\bfgamma !}, P_{\bfgamma}(\bfz- \bm{\psi}_{\textrm{map}}(\bfz))\bigg)-\sum_{|\bfgamma|\leq m}\frac{P_{\bfgamma}(\bfz- \bm{\psi}_{\textrm{map}}(\bfz))}{\bfgamma !},\nonumber
\end{eqnarray}
where $\phi_\times\in \mcF_{2,1}(2(m+1)(M+1)+1, 9N)$ is the product neural network in Lemma \ref{lemma:multiplication:network} such that 
\begin{eqnarray}
	|\phi_\times(x, y)-xy|\leq 12N^{-2(m+1)(M+1)}.\label{eq:lemma:approximation:p:c:smooth:step:3:conclusion:2}
\end{eqnarray}
\noindent \textbf{Step 4:}
For any $\bfz \in Q_{\bftheta}$, let $\bfh:=\bfz-\frac{\bftheta}{K}=\bfz-\frac{\bm{\psi}_{\textrm{map}}(\bfz)}{K}$. By Taylor's Theorem,  we can quantify the bound of $|f(\bfz)-g(\bfz)|$ by
\begin{eqnarray}
	&&\bigg|\sum_{|\bfgamma|\leq m}\frac{\partial^{\bfgamma}f(\bm{\psi}_{\textrm{map}}(\bfz))}{\bfgamma !}\bfh^{\bfgamma}+\sum_{|\bfgamma|= m}\frac{\partial^{\bfgamma}f(\bm{\psi}_{\textrm{map}}(\bfz)+\xi_{\bfz}\bfh)}{\bfgamma !}\bfh^{\bfgamma}-\sum_{|\bfgamma|= m}\frac{\partial^{\bfgamma}f(\bm{\psi}_{\textrm{map}}(\bfz))}{\bfgamma !}\bfh^{\bfgamma}-g(\bfz)\bigg|\nonumber\\
	&\leq& S_1+S_2,\nonumber
\end{eqnarray}
where $\xi_{\bfz}\in [0, 1]$ and
\begin{eqnarray}
	S_1=\bigg|\sum_{|\bfgamma|= m}\frac{\partial^{\bfgamma}f(\bm{\psi}_{\textrm{map}}(\bfz)+\xi_{\bfz}\bfh)}{\bfgamma !}\bfh^{\bfgamma}-\sum_{|\bfgamma|= m}\frac{\partial^{\bfgamma}f(\bm{\psi}_{\textrm{map}}(\bfz))}{\bfgamma !}\bfh^{\bfgamma}\bigg|, \; S_2= \bigg|\sum_{|\bfgamma|\leq m}\frac{\partial^{\bfgamma}f(\bm{\psi}_{\textrm{map}}(\bfz))}{\bfgamma !}\bfh^{\bfgamma}-g(\bfz)\bigg|.\nonumber
\end{eqnarray}
By the definition of $(p, 1)$-H\" older smooth, it follows that
\begin{eqnarray}
	S_1\leq \sum_{|\bfgamma|=m}\frac{\|\xi_{\bfz}\bfh\|_2^{p-m}}{\bfgamma !}\bfh^{\bfgamma}\leq \frac{d(m+1)^d}{\bfgamma !}K^{-p}\leq d(m+1)^dK^{-p}.\nonumber
\end{eqnarray}
For the term $S_2$, we have $S_2=S_{21}+S_{22}$ with
\begin{eqnarray}
	S_{21}&=&2\sum_{|\bfgamma|\leq m}\frac{\partial^{\bfgamma}f(\bm{\psi}_{\textrm{map}}(\bfz))+1}{2\bfgamma !}\bfh^{\bfgamma}-\sum_{|\bfgamma|\leq m}2\phi_{\times}\bigg(\frac{\Phi_{\bfgamma}\circ \bm{\psi}_{\textrm{map}}(\bfz)}{\bfgamma !}, P_{\bfgamma}(\bfz- \bm{\psi}_{\textrm{map}}(\bfz))\bigg),\nonumber\\
	S_{22}&=&\sum_{|\bfgamma|\leq m}\frac{1}{\bfgamma !}\bfh^{\bfgamma}-\sum_{|\bfgamma|\leq m}\frac{P_{\bfgamma}(\bfz- \bm{\psi}_{\textrm{map}}(\bfz))}{\bfgamma !}.\nonumber
\end{eqnarray}
Direct examination leads to the following bound:
\begin{eqnarray}
	S_{21}&\leq&  2\bigg|\sum_{|\bfgamma|\leq m}\frac{\partial^{\bfgamma}f(\bm{\psi}_{\textrm{map}}(\bfz))+1}{2\bfgamma !}\bfh^{\bfgamma}-\sum_{|\bfgamma|\leq m}\frac{\Phi_{\bfgamma}\circ \bm{\psi}_{\textrm{map}}(\bfz)}{\bfgamma !} P_{\bfgamma}(\bfz- \bm{\psi}_{\textrm{map}}(\bfz))\bigg|\nonumber\\
	&&+2\bigg|\sum_{|\bfgamma|\leq m}\frac{\Phi_{\bfgamma}\circ \bm{\psi}_{\textrm{map}}(\bfz)}{\bfgamma !} P_{\bfgamma}(\bfz- \bm{\psi}_{\textrm{map}}(\bfz))-\sum_{|\bfgamma|\leq m}2\phi_{\times}\bigg(\frac{\Phi_{\bfgamma}\circ \bm{\psi}_{\textrm{map}}(\bfz)}{\bfgamma !}, P_{\bfgamma}(\bfz- \bm{\psi}_{\textrm{map}}(\bfz))\bigg)\bigg|\nonumber\\
	&\leq&2\sum_{|\bfgamma|\leq m}\bigg|\frac{\partial^{\bfgamma}f(\bm{\psi}_{\textrm{map}}(\bfz))+1}{2\bfgamma !}\bfh^{\bfgamma}-\frac{\Phi_{\bfgamma}\circ \bm{\psi}_{\textrm{map}}(\bfz)}{\bfgamma !}\bfh^{\bfgamma}\bigg|\nonumber\\ 
	&&+2\sum_{|\bfgamma|\leq m}\bigg|\frac{\Phi_{\bfgamma}\circ \bm{\psi}_{\textrm{map}}(\bfz)}{\bfgamma !}\bfh^{\bfgamma}-\frac{\Phi_{\bfgamma}\circ \bm{\psi}_{\textrm{map}}(\bfz)}{\bfgamma !} P_{\bfgamma}(\bfz- \bm{\psi}_{\textrm{map}}(\bfz))\bigg|\nonumber\\ 
	&&+2\sum_{|\bfgamma|\leq m}\sup_{x,y \in [0, 1]}|xy-\phi_{\times}(x,y)|\nonumber\\ 
	&\leq& 2(m+1)^dN^{-2(m+1)}M^{-2(m+1)}+24(m+1)^d(1+N^{-2(m+1)}M^{-2(m+1)})(m-1)N^{-2(m+1)(M+1)}\nonumber\\
	&&+24(m+1)^dN^{-2(m+1)(M+1)}\nonumber\\
	&\leq&2(m+1)^dN^{-2(m+1)}M^{-2(m+1)}+72(m+1)^{d+1}N^{-2(m+1)(M+1)},\nonumber
\end{eqnarray}
where we use (\ref{eq:lemma:approximation:p:c:smooth:step:2:eq0}), (\ref{eq:lemma:approximation:p:c:smooth:step:3:conclusion}) and (\ref{eq:lemma:approximation:p:c:smooth:step:3:conclusion:2}). Similarly, using (\ref{eq:lemma:approximation:p:c:smooth:step:2:eq0}), we have
\begin{eqnarray}
	S_{22}&\leq&\sum_{|\bfgamma|\leq m}\bigg|\bfh^{\bfgamma}-P_{\bfgamma}(\bfz- \bm{\psi}_{\textrm{map}}(\bfz))\bigg|\leq 12(m+1)^{d+1}N^{-2(m+1)(M+1)}\nonumber.
\end{eqnarray}
As a consequence of the bounds of $S_1, S_{21}$ and $S_{22}$, we conclude that for $\bfz \in Q_{\bftheta}$, it holds that
\begin{eqnarray}
	|f(\bfz)-g(\bfz)|&\leq&  d(m+1)^dK^{-p}+2(m+1)^dN^{-2(m+1)}M^{-2(m+1)}+84(m+1)^{d+1}N^{-2(m+1)(M+1)}\nonumber\\
	&\leq& d8^p(m+1)^dN^{-2p/d}M^{-2p/d}+2(m+1)^dN^{-2(m+1)}M^{-2(m+1)}+\nonumber\\
	&&+84(m+1)^{d+1}N^{-2(m+1)}N^{-2(m+1)M}\nonumber\\
	&\leq&d8^p(m+1)^dN^{-2p/d}M^{-2p/d}+2(m+1)^dN^{-2(m+1)}M^{-2(m+1)}+\nonumber\\
	&&+84(m+1)^{d+1}N^{-2(m+1)}M^{-2(m+1)}\nonumber\\
	&\leq& (86+d8^p)(m+1)^{d+1}N^{-2p/d}M^{-2p/d},\nonumber
\end{eqnarray}
where we use the fact that $K=\floor{N^{1/d}}^2\floor{M^{2/d}}\geq N^{2/d}M^{2/d}/8$.

\noindent \textbf{Step 5:} We will demonstrate how to implement $g$ using neural network. We denote the set $\Gamma=\{\bfgamma\in \mathbb{N}^d : |\bfgamma|\leq m\}$. By basic Combinatorics, we now $|\Gamma|={m+d-1 \choose m-1}$. For simplicity, we set $|\Gamma|=J$ and index the element in $\Gamma$ as $\bfgamma_1,\ldots, \bfgamma_J$. For each $\bfgamma in \Gamma$, we define
\begin{eqnarray}
	h_{\bfgamma}(\bfz)=\frac{\Phi_{\bfgamma}\circ \bm{\psi}_{\textrm{map}}(\bfz)}{\bfgamma !}\quad \quad\textrm{ and }\quad\quad f_{\bfgamma}(\bfz)=P_{\bfgamma}(\bfz- \bm{\psi}_{\textrm{map}}(\bfz)).\nonumber
\end{eqnarray}
By (\ref{eq:lemma:approximation:p:c:smooth:step:1:conclusion}), (\ref{eq:lemma:approximation:p:c:smooth:step:2:eq0}) and (\ref{eq:lemma:approximation:p:c:smooth:step:3:eq1}), it is not difficult to see $h_{\bfgamma}$ has at most depth
\begin{eqnarray}
	4M+4+(5L+8)\ceil{\log_2(2M)}+2+2\leq 10(M+2)\ceil{\log_2(2M)}\nonumber
\end{eqnarray}
and width
\begin{eqnarray}
	d(4N+5)+8(m+1)(2N+3)\ceil{\log_2(4N)}+d\leq 16(m+d+1)(N+2)\ceil{\log_2(4N)}.\nonumber
\end{eqnarray}
Similarly, we can show $f_{\bfgamma}$ has at most depth
\begin{eqnarray}
	4M+4+1+[2(m+1)(M+2)+2](m-1)+1\leq 6(m+1)^2(M+3)\nonumber
\end{eqnarray}
and width
\begin{eqnarray}
	d(4N+5)+d+9N+m-2\leq (4d+9)(N+2)+m.\nonumber
\end{eqnarray}
Now, we parallelize all the $h_{\bfgamma}$'s and $f_{\bfgamma}$'s together, we can obtain a neural network 
\begin{eqnarray}
	g_1:=(h_{\bfgamma_1},\ldots, h_{\bfgamma_J},f_{\bfgamma_1},\ldots, f_{\bfgamma_J})^\top \in \mcF_{d, 2J}\nonumber
\end{eqnarray}
with at most $16(m+1)^{2}(M+3)\ceil{\log_2(2M)}$ hidden layers and at most width $40J(m+d+3)(N+2)\ceil{\log_2(4N)}$.

Next for each $\bfgamma \in \Gamma$, using the product neural network in (\ref{eq:lemma:approximation:p:c:smooth:step:3:conclusion:2}), we define 
\begin{eqnarray}
	s_{\bfgamma}=2\phi_{\times}(h_{\bfgamma}, f_{\bfgamma}),\nonumber
\end{eqnarray}
and use the outputs of $g_1$ to construct a neural network $g_2\circ g_1$ such that
\begin{eqnarray}
	g_2\circ g_1:=(s_{\bfgamma_1},\ldots, s_{\bfgamma_J},f_{\bfgamma_1},\ldots, f_{\bfgamma_J})^\top \in \mcF_{d, 2J},\nonumber
\end{eqnarray}
which will have at most  depth
\begin{eqnarray}
	16(m+1)^{2}(M+3)\ceil{\log_2(2M)}+2(m+1)(M+1)+1+1\leq 18(m+1)^{2}(M+3)\ceil{\log_2(2M)}\nonumber
\end{eqnarray}
and at most width $40J(m+d+3)(N+2)\ceil{\log_2(4N)}$.

Finally, we make a linear combination using the outputs of $g_2\circ g_1$ to obtain the final neural network, which will require one more hidden layer. We will finish the proof  if we notice $J\leq (m+1)^d$.
\end{proof}

\begin{lemma}\label{lemma:continuous:extension:approximation}
Given $\varepsilon>0$, $L, W, K\in \mathbb{Z}_+$, and $\delta>0$ with $\delta<\frac{1}{3K}$, assume $f$ is continuous on $[0, 1]^d$. Suppose there is a neural network $\widetilde{g}\in \mcF_{d,1}(L, W)$ such that
\begin{eqnarray}
	|f(\bfz)-\widetilde{g}(\bfz)|\leq \varepsilon \quad \textrm{ for all } \bfz\in [0,1]^d\backslash \Omega(K, \delta, d),\nonumber
\end{eqnarray}
then there exists a neural network $g\in \mcF_{d,1}(L+2d, 3^d(N+3))$ such that
\begin{eqnarray}
	|f(\bfz)-g(\bfz)|\leq \varepsilon+d\omega_f(\delta) \quad \textrm{ for all } \bfz\in [0,1]^d,\nonumber
\end{eqnarray}
where $\omega_f(\delta)=\sup\{|f(\bfx)-f(\bfy)|: \|\bfx-\bfy\|_2\leq \delta, \bfx, \bfy\in [0, 1]^d\}$.
\end{lemma}

\begin{lemma}\label{lemma:mk19}
Suppose $f: [0, 1]^d \to \mathbb{R}$ be $(p, C)$-H\" older smooth for some $p, C>0$, then for all positive intergers $M, N$, there exists a neural network $g\in \mcF_{d,1}(L, W)$ with 
\begin{equation*}
	L=54\ceil{p}^{2}(M+3)\ceil{\log_2(2M)}\;\;\textrm{ and }\;\;W=40(\ceil{p}+d+2)^{d+1}3^d(N+2)\ceil{\log_2(4N)}+3^{d+1}
\end{equation*}
such that
\begin{equation*}
	\sup_{\bfz\in [0, 1]^d}|f(\bfz)-g(\bfz)|\leq C(86+d8^p+d)(p+1)^{d+1}N^{-2p/d}M^{-2p/d}.
\end{equation*}
\end{lemma}
\begin{proof}
Define $\widetilde{f}=C^{-1}f$, so $\widetilde{f}$ is $(p, 1)$-H\" older smooth. Set $K=\floor{N^{1/d}}^2\floor{M^{2/d}}$ and choose $\delta\in (0, \frac{1}{3K})$ such that $\omega_{\widetilde{f}}(\delta)\leq N^{-2p/d}L^{-2p/d}$. This can be done, as 
\begin{eqnarray}
	\omega_f(\delta)\leq \begin{cases}
	\delta &\textrm{ if } p\geq 1\\
	\delta^p &\textrm{ if } p<  1\\
	\end{cases}.\nonumber
\end{eqnarray}
By Lemma \ref{lemma:approximation:p:c:smooth}, there exists a neural network $\widetilde{g}_1$ with $18\ceil{p}^{2}(M+3)\ceil{\log_2(2M)}$ hidden layers and width $40(\ceil{p}+d+2)^{d+1}(N+2)\ceil{\log_2(4N)}$ such that $|\widetilde{f}(\bfz)-\widetilde{g}_1(\bfz)|\leq (86+d8^p)(m+1)^{d+1}N^{-2p/d}M^{-2p/d}$ for all $\bfz\in [0, 1]^d\backslash \Omega(K, \delta, d)$. By Lemma \ref{lemma:continuous:extension:approximation},  there exists a neural network $\widetilde{g}_2$ with $54\ceil{p}^{2}(M+3)\ceil{\log_2(2M)}$ hidden layers and width $40(\ceil{p}+d+2)^{d+1}3^d(N+2)\ceil{\log_2(4N)}+3^{d+1}$ such that 
\begin{eqnarray}
	|\widetilde{f}(\bfz)-\widetilde{g}_2(\bfz)|&\leq& (86+d8^p)(p+1)^{d+1}N^{-2p/d}M^{-2p/d}+dN^{-2p/d}L^{-2p/d}\nonumber\\
	&\leq& (86+d8^p+d)(p+1)^{d+1}N^{-2p/d}M^{-2p/d}\;\;\textrm{ for all } \bfz\in [0, 1]^d.\nonumber
\end{eqnarray}
As a consequence, the neural network $g:=Cg_2$, which has the same architecture as $g_2$, is the desired neural network.
\end{proof}

\begin{lemma}\label{lemma:offset:to:unit:interval}
Suppose $h: [0, 1]^d \to [0, 1]$ be $(p, C)$-H\" older smooth for some $p, C>0$, then for all positive integers $M, N$, there exists a neural network $h_{\textrm{net}}\in \mcF_{d,1}(L, W)$ with 
\begin{equation*}
	L=54\ceil{p}^{2}(M+3)\ceil{\log_2(2M)}+1\quad \textrm{ and }\quad W=40(\ceil{p}+d+2)^{d+1}3^d(N+2)\ceil{\log_2(4N)}+3^{d+1}
\end{equation*}
such that $0\leq h_{\textrm{net}}(\bfz)\leq 1$ and
\begin{equation*}
	\sup_{\bfz\in [0, 1]^d}|h(\bfz)-h_{\textrm{net}}(\bfz)|\leq 2(86+d8^p+d)(p+1)^{d+1}N^{-2p/d}M^{-2p/d}.
\end{equation*}
\end{lemma}
\begin{proof}
 By Lemma \ref{lemma:mk19}, there exists $\widetilde{h}_{\textrm{net}}\in \mcF_{d,1}(\widetilde{L}, \widetilde{W})$ with
\begin{equation*}
	\widetilde{L}=54\ceil{p}^{2}(M+3)\ceil{\log_2(2M)}\;\;\textrm{ and }\;\; \widetilde{W}=40(\ceil{p}+d+2)^{d+1}3^d(N+2)\ceil{\log_2(4N)}+3^{d+1}
\end{equation*}
such that
\begin{equation}\label{eq:lemma:offset:to:unit:interval:eq1}
	\sup_{\bfz\in [0, 1]^d}|\widetilde{h}_{\textrm{net}}(\bfz)-h(\bfz)|\leq(86+d8^p+d)(p+1)^{d+1}N^{-2p/d}M^{-2p/d}\;\; \textrm{ for all } N, M\in \mathbb{Z}_+.  
\end{equation}
Let $\delta=(86+d8^p+d)(p+1)^{d+1}N^{-2p/d}M^{-2p/d}$ and define $h_{\textrm{net}}(\bfz)=\sigma(\widetilde{h}(\bfz))/(1+2\delta)$, which is always non-negative. If $\widetilde{h}(\bfz)\leq 0$, then by (\ref{eq:lemma:offset:to:unit:interval:eq1}), we have $h(\bfz)-\delta\leq \widetilde{h}(\bfz)\leq 0$, which further leads to $0\leq h(\bfz)\leq \delta$. Therefore, it follows that
\begin{eqnarray*}
	|h_{\textrm{net}}(\bfz)-h(\bfz)|=&\begin{cases}
	\bigg|\frac{\widetilde{h}(\bfz)}{1+\delta}-h(\bfz)\bigg| &\textrm{ if } \widetilde{h}(\bfz)>0\\
	h(\bfz)&\textrm{ if } \widetilde{h}(\bfz)\leq0\\
	\end{cases}.
\end{eqnarray*}
Combining above, we conclude that $\sup_{\bfz\in [0, 1]^d}|h_{\textrm{net}}(\bfz)-h(\bfz)|\leq 2 \delta$. Also by (\ref{eq:lemma:offset:to:unit:interval:eq1}), it yields that $h_{\textrm{net}}(\bfz)\leq |\widetilde{h}(\bfz)|/(1+2\delta)\leq 1$. Clearly, $h_{\textrm{net}}$ can be implemented by adding one more layer to $\widetilde{h}(\bfz)$.
\end{proof}

\begin{lemma}\label{theorem:approximation:deep}
Suppose $f\in  \mcCS(L_*, \bfd, \bft, \bfp, \bfa, \bfb, C)$, for all positive integers $M_0, N_0, \ldots, M_{L_*}, N_{L_*}$, there exists $\widetilde{f} \in \mcF_{d, 1}(L_*+\sum_{i=0}^{L_*}L_i, \max_{0\leq i \leq L_*}W_id_{i+1})$ with
\begin{equation*}
	L_i= 54\ceil{p_i}^{2}(M_i+3)\ceil{\log_2(2M_i)}+1\;\;\textrm{ and }\; W_i=40(\ceil{p_i}+t_i+2)^{t_i+1}3^{t_i}(N_i+2)\ceil{\log_2(4N_i)}+3^{t_i+1} 
\end{equation*}
such that 
\begin{eqnarray*}
	\|\widetilde{f}-{f}\|_\infty\leq  c \sum_{i=0}^{L_*}(M_iN_i)^{-2p_i\prod_{s=i+1}^{L_*}(p_s\wedge 1)/t_i},
\end{eqnarray*}
where constant $c$ is free of $d, M$ and $N$. As a consequence, if $(L, W) \to \infty$, then there exists $\widetilde{f} \in \mcF_{d,1}(L, W)$ such that
\begin{eqnarray*}
\|\widetilde{f}-{f}\|_\infty\leq  c\max_{1\leq i\leq L_*}\bigg(\frac{LW}{\log(L)\log(W)}\bigg)^{-\frac{2p_i\prod_{s=i+1}^{L_*}(p_s\wedge 1)}{t_i}}=c\bigg(\frac{LW}{\log(L)\log(W)}\bigg)^{-\frac{2p^*}{t^*}}.
\end{eqnarray*}
where the constant $c$ is free of $d, L$ and $W$. Moreover, if $\bff=(f_1, f_2, \ldots, f_q)^\top$ with each $f_i \in \mcCS(L_*, \bfd, \bft, \bfp, \bfa, \bfb, C)$, then there exists $\bff^*=(f_1^*, f_2^*,\ldots, f_q^*)^\top \in \mcF_{d, q}(L, W)$ suhc that
\begin{eqnarray*}
	\|{f}_s^*-{f}_s\|_\infty\leq  c\bigg(\frac{LW}{\log(L)\log(W)}\bigg)^{-\frac{2p^*}{t^*}}\quad \textrm{ for all } s=1,\ldots, d,
\end{eqnarray*}
where the constant $c$ is free of $d, L$ and $W$.
\end{lemma}
\begin{proof}
\textbf{Step 1:} We will rewrite $f$ as composition of functions $\bfh_i$'s, which are defined as follows: 
\begin{eqnarray*}
	\bfh_i(\bfz)&=&\frac{\bfg_i((b_i-a_i)\bfz+a_i)-a_{i+1}}{b_{i+1}-a_{i+1}}\;\;\;\textrm{ for } \bfz \in [0,1]^{d_i} \textrm{ and } i=0, 1,\ldots, L_*-1,\nonumber\\
	\bfh_{L_*}(\bfz)&=&\bfg_{L_*}((b_{L_*}-a_{L_*})\bfz+a_{L_*})\;\;\;\textrm{ for } \bfz \in [0,1]^{d_{L_*}}.
\end{eqnarray*} 
By above definition, we can see the range of $\bfh_i$ is $[0, 1]^{d_{i+1}}$ for $i=0,\ldots, L_*-1$ and the domain of $\bfh_i$ is $[0, 1]^{d_i}$ for $i=0,\ldots, L_*$. It is not difficult to verify following equality:
\begin{equation*}
	f(\bfz)=\bfh_{L_*}\circ\ldots\circ \bfh_0\bigg(\frac{\bfz-a_0}{b_0-a_0}\bigg) \;\;\;\textrm{ for } \bfz \in [a_0, b_0]^d.
\end{equation*}
\noindent\textbf{Step 2:} Let  $h_{i,1},\ldots, h_{i,d_{i+1}}$ be the elements of $\bfh_i$. Since all $g_{i,j}$'s are $(p_i, C)$-H\" older smooth and only relying on $t_i$ variables, we can see that $h_{i,j}$'s are $(p_i, K)$-H\" older smooth with $K=C\sum_{i=0}^{L_*-1}(b_i-a_i)/(b_{i+1}-a_{i+1})+b_{L_*}-a_{L_*}$. By Lemmas \ref{lemma:mk19} and \ref{lemma:offset:to:unit:interval}, there exists $\widetilde{h}_{i,j}\in \mcF_{t_i, 1}(L_i, W_i)$ with 
\begin{equation*}
	L_i= 54\ceil{p_i}^{2}(M_i+3)\ceil{\log_2(2M_i)}+1\;\;\textrm{ and }\; W_i=40(\ceil{p_i}+t_i+2)^{t_i+1}3^{t_i}(N_i+2)\ceil{\log_2(4N_i)}+3^{t_i+1} 
\end{equation*}
such that 
\begin{equation}\label{eq:theorem:approximation:deep:step2:result}
\sup_{\bfz\in [0,1]^{t_i}}|\widetilde{h}_{i,j}(\bfz)-h_{i,j}(\bfz)|\leq c_iN^{-2p_i/t_i}M^{-2p_i/t_i},
\end{equation}
where $c_i=K(86+t_i8^{p_i}+t_i)(p_i+1)^{t_i+1}$.
\newline
\noindent\textbf{Step 3:}  For each $i=0,\ldots, L_*$, we further define $\widetilde{\bfh}_i=(\widetilde{h}_{i,1},\ldots, \widetilde{h}_{i, d_{i+1}})^\top$. Moreover, let 
\begin{equation*}
	\widetilde{f}(\bfz)=\widetilde{\bfh}_{L_*}\circ\ldots\circ \widetilde{\bfh}_0\bigg(\frac{\bfz-a_0}{b_0-a_0}\bigg)\;\;\; \textrm{ for } \bfz\in [a_0, b_0]^d.
\end{equation*}
Next we will quantify the difference between $\widetilde{f}(\bfz)$ and $f(\bfz)$. 

\begin{eqnarray*}
	|\widetilde{f}(\bfz)-{f}(\bfz)|&=&\bigg|\widetilde{\bfh}_{L_*}\circ \widetilde{\bfh}_{L_*-1}\circ\ldots\circ \widetilde{\bfh}_0\bigg(\frac{\bfz-a_0}{b_0-a_0}\bigg)-\bfh_{L_*}\circ \bfh_{L_*-1}\circ\ldots\circ \bfh_0\bigg(\frac{\bfz-a_0}{b_0-a_0}\bigg)\bigg|\nonumber\\
	&\leq& \bigg|\widetilde{\bfh}_{L_*}\circ \widetilde{\bfh}_{L_*-1}\circ\ldots\circ \widetilde{\bfh}_0\bigg(\frac{\bfz-a_0}{b_0-a_0}\bigg)-\bfh_{L_*}\circ \widetilde{\bfh}_{L_*-1}\circ\ldots\circ \widetilde{\bfh}_0\bigg(\frac{\bfz-a_0}{b_0-a_0}\bigg)\bigg|\nonumber\\
	&&+\bigg|{\bfh}_{L_*}\circ \widetilde{\bfh}_{L_*-1}\circ\ldots\circ \widetilde{\bfh}_0\bigg(\frac{\bfz-a_0}{b_0-a_0}\bigg)-\bfh_{L_*}\circ \bfh_{L_*-1}\circ\ldots\circ \bfh_0\bigg(\frac{\bfz-a_0}{b_0-a_0}\bigg)\bigg|\nonumber\\
	&\leq& \sup_{\bfz \in [0, 1]^{d_{L_*}}}|\widetilde{\bfh}_{L_*}(\bfz)-\bfh_{L_*}(\bfz)|\nonumber\\
	&&+K\bigg\|\widetilde{\bfh}_{L_*-1}\circ\ldots\circ \widetilde{\bfh}_0\bigg(\frac{\bfz-a_0}{b_0-a_0}\bigg)-\bfh_{L_*-1}\circ\ldots\circ \bfh_0\bigg(\frac{\bfz-a_0}{b_0-a_0}\bigg)\bigg\|_2^{p_{L_*} \wedge 1},
\end{eqnarray*}
where we use the fact that $\bfh_{L_*}$ is $(p_{L_*}, K)$-H\" older smooth and $d_{L_*+1}=1$. Moreover, we can show the following holds for  $i=1,\ldots, L_*-1$ and $r\in (0, 1]$:
\begin{eqnarray*}
	&&\bigg\|\widetilde{\bfh}_{i}\circ\widetilde{\bfh}_{i-1}\circ\ldots\circ \widetilde{\bfh}_0\bigg(\frac{\bfz-a_0}{b_0-a_0}\bigg)-\bfh_{i}\circ\bfh_{i-1}\circ\ldots\circ \bfh_0\bigg(\frac{\bfz-a_0}{b_0-a_0}\bigg)\bigg\|_2^{r}\nonumber\\
	\leq&&\bigg\|\widetilde{\bfh}_{i}\circ\widetilde{\bfh}_{i-1}\circ\ldots\circ \widetilde{\bfh}_0\bigg(\frac{\bfz-a_0}{b_0-a_0}\bigg)-\bfh_{i}\circ\widetilde{\bfh}_{i-1}\circ\ldots\circ \widetilde{\bfh}_0\bigg(\frac{\bfz-a_0}{b_0-a_0}\bigg)\bigg\|_2^{r}\nonumber\\\
&&+\bigg\|{\bfh}_{i}\circ\widetilde{\bfh}_{i-1}\circ\ldots\circ \widetilde{\bfh}_0\bigg(\frac{\bfz-a_0}{b_0-a_0}\bigg)-\bfh_{i}\circ\bfh_{i-1}\circ\ldots\circ \bfh_0\bigg(\frac{\bfz-a_0}{b_0-a_0}\bigg)\bigg\|_2^{r}\nonumber\\	
\leq&& \bigg(\sum_{j=1}^{d_{i+1}}\|\widetilde{h}_{i,j}-h_{i,j}\|^2_\infty\bigg)^{\frac{r}{2}}\nonumber\\
&&+\bigg(\sum_{j=1}^{d_{i+1}}\bigg|h_{i,j}\circ\widetilde{\bfh}_{i-1}\circ\ldots\circ \widetilde{\bfh}_0\bigg(\frac{\bfz-a_0}{b_0-a_0}\bigg)-h_{i,j}\circ\bfh_{i-1}\circ\ldots\circ \bfh_0\bigg(\frac{\bfz-a_0}{b_0-a_0}\bigg)\bigg|^2\bigg)^{\frac{r}{2}}\nonumber\\
\leq&& d_{i+1}\sup_{1\leq j \leq d_{i+1}}\|\widetilde{h}_{i,j}-h_{i,j}\|_\infty^{r}\nonumber\\
&&+\sum_{j=1}^{d_{i+1}}\bigg|h_{i,j}\circ\widetilde{\bfh}_{i-1}\circ\ldots\circ \widetilde{\bfh}_0\bigg(\frac{\bfz-a_0}{b_0-a_0}\bigg)-h_{i,j}\circ\bfh_{i-1}\circ\ldots\circ \bfh_0\bigg(\frac{\bfz-a_0}{b_0-a_0}\bigg)\bigg|^{r}\nonumber\\
\leq&& d_{i+1}\sup_{\bfz \in [0, 1]^{d_i}}\|\widetilde{\bfh}_i(\bfz)-\bfh_i(\bfz)\|_\infty^{r}\nonumber\\
&&+d_{i+1}K^{p_{i+1}\wedge 1}\bigg\|\widetilde{\bfh}_{i-1}\circ\ldots\circ \widetilde{\bfh}_0\bigg(\frac{\bfz-a_0}{b_0-a_0}\bigg)-\bfh_{i-1}\circ\ldots\circ \bfh_0\bigg(\frac{\bfz-a_0}{b_0-a_0}\bigg)\bigg\|_2^{r(p_{i}\wedge 1)}.
\end{eqnarray*}
Finally, by induction, we conclude that
\begin{eqnarray*}
|\widetilde{f}(\bfz)-{f}(\bfz)|&\leq& K^{1+\sum_{i=1}^{L_*}p_i\wedge 1}\prod_{i=1}^{L_*}d_i \sum_{i=0}^{L_*}\sup_{1\leq j \leq d_{i+1}}\|\widetilde{h}_{i,j}-h_{i,j}\|_\infty^{\prod_{s=i+1}^{L_*}(p_s\wedge 1)},
\end{eqnarray*}
where we use the convention $\prod_{s=L_*+1}^{L_*}(p_i \wedge 1)=1$. By (\ref{eq:theorem:approximation:deep:step2:result}), we show that 
\begin{eqnarray*}
	|\widetilde{f}(\bfz)-{f}(\bfz)|&\leq& K^{1+\sum_{i=1}^{L_*}p_i\wedge 1}\prod_{i=1}^{L_*}d_i \sum_{i=0}^{L_*}(c_{i}M_i^{-2p_i/t_i}N_i^{-2p_i/t_i})^{\prod_{s=i+1}^{L_*}(p_s\wedge 1)}\nonumber\\
	&\leq&  K^{1+\sum_{i=1}^{L_*}p_i\wedge 1}\prod_{i=1}^{L_*}d_i\sum_{i=0}^{L_*}c_i^{\prod_{s=i+1}^{L_*}(p_s\wedge 1)}  \sum_{i=0}^{L_*}(M_iN_i)^{-2p_i\prod_{s=i+1}^{L_*}(p_s\wedge 1)/t_i}\nonumber.
\end{eqnarray*}
Therefore, we prove the bound with $c=K^{1+\sum_{i=1}^{L_*}p_i\wedge 1}\prod_{i=1}^{L_*}d_i\sum_{i=0}^{L_*}c_i^{\prod_{s=i+1}^{L_*}(p_s\wedge 1)}$, which is free of $d_0, L, W$.
\newline
\noindent\textbf{Step 4:}  We will show $\widetilde{f}$ indeed can be implemented by a deep neural network. We add the one more layer to transform the corresponding variables of $\widetilde{h}_{0,j}$ ($t_i$ variables) and parallelize all the $\widetilde{h}_{0,j}$'s (the total number is $d_1$) to implement $\widetilde{\bfh}_0((\bfz-a_0)/(b_0-a_0))$, which can be verified that 
\begin{eqnarray*}
\widetilde{\bfh}_0((\bfz-a_0)/(b_0-a_0))\in \mcF_{d_0, d_1}(L_0+1, d_1W_0),
\end{eqnarray*}
where $L_i, W_i$ are defined in Step 2.
Now we use the outputs of $\widetilde{\bfh}_0((\bfz-a_0)/(b_0-a_0))$ as inputs and  parallelize all the $\widetilde{h}_{1,j}$'s (the total number is $d_2$) to implement $\widetilde{\bfh}_1\circ\widetilde{\bfh}_0((\bfz-a_0)/(b_0-a_0))$. It can be verified that
\begin{eqnarray*}
\widetilde{\bfh}_1\circ\widetilde{\bfh}_0((\bfz-a_0)/(b_0-a_0))\in \mcF_{d_0, d_2}(L_0+L_1+2, \max\{d_1W_0, d_2W_1\}).
\end{eqnarray*}
By induction, we concluded that $\widetilde{f} \in \mcF_{d, 1}(L_*+\sum_{i=0}^{L_*}L_i, \max_{0\leq i \leq L_*}W_id_{i+1})$, here we use the fact $d_0=d$ and $d_{L_*+1}=1$. 
\newline
\noindent\textbf{Step 5:} Now we will prove the second result. It is not difficult to verify $M_i\asymp L_i\log^{-1}(L_i)$ and $N_i\asymp W_i^{-1}\log(W_i)$.  Specifically, we choose $L_0=L_1=\cdots=L_{L_*}\asymp L$, $W_0=W_1=\cdots=W_{L_*}\asymp W$, then  the desired result follows.
\newline
\noindent\textbf{Step 6:}  The third result can be easily obtained by parallelizing $q$ deep neural networks in Step 5.
\end{proof}

\subsection{Proof of Results in Section \ref{sec:asymptotics}}\label{sec:proof:sec:4}
\begin{proof}[\textbf{Proof of Lemma \ref{lemma:pc:smooth:degree}}]
Denote the degree of smoothness of $\bfg_i\circ\ldots\circ\bfg_0$ as $\widetilde{p}_i$. By \cite{jlt09}, the smoothness of $\bfg_1\circ\bfg_0$ is 
\begin{equation*}
	\widetilde{p}_1:=p_0p_1\wedge p_0\wedge p_1=p_0(p_1\wedge 1)\wedge p_1\leq p_0(p_1\wedge 1).
\end{equation*}
Above equation and \cite{jlt09} further imply
\begin{equation*}
	\widetilde{p}_2:=\widetilde{p}_1p_2\wedge \widetilde{p}_1\wedge p_2=\widetilde{p}_1(p_2\wedge 1)\wedge p_2\leq p_0(p_1\wedge 1)(p_2\wedge 1).
\end{equation*}
By induction, we conclude that
\begin{equation}
	\widetilde{p}_i\leq p_0\prod_{k=1}^i(p_k\wedge 1),\label{eq:lemma:pc:smooth:degree:eq1}
\end{equation}
here for convenience, we define $\widetilde{p}_0=p_0$.
Similarly, if we define the smoothness of $\bfg_{L_*}\circ\ldots\circ\bfg_{i}$ as $\bar{p}_i$, we can show that
\begin{equation}
	\bar{p}_i\leq  p_{i}\prod_{k=i+1}^{L_*}(p_k\wedge 1)=p_{i}^*,\label{eq:lemma:pc:smooth:degree:eq2}
\end{equation} 
here for convenience, we define $\bar{p}_{L_*}=p_{L_*}$.
By \cite{jlt09}, (\ref{eq:lemma:pc:smooth:degree:eq1}) and (\ref{eq:lemma:pc:smooth:degree:eq2}), it follows that
\begin{equation*}
	p=\widetilde{p}_{i-1}\bar{p}_i\wedge \widetilde{p}_{i-1} \wedge \bar{p}_i\leq 	\bar{p}_i\leq p_{i}^*, \quad\quad\textrm{ for all } i=1,\ldots, L_*.
\end{equation*}
Finally, notice $p=\widetilde{p}_{L_*}\leq p_0\prod_{k=1}^{L_*}(p_k\wedge 1)=p_{L_*}^*$, we conclude that
\begin{equation*}
	p\leq \min_{0\leq i\leq L_*}p_i^*.
\end{equation*}
\end{proof}


\begin{lemma}\label{lemma:moment:bound}
Let $\zeta_s=X_s-f_{0,s}(\bfZ)$ for $s=1,\ldots, q$. Under Assumption \ref{Assumption:AC}, the following holds for any positive integer $k$, real number $B>0$ and $s=1,\ldots, q$:
\begin{eqnarray}
	\ev\bigg(|\zeta_s|^kI(|\zeta_s|>B)\bigg)
	&\leq& \kappa_2e^{\kappa_1C}\bigg(\frac{k+1}{\kappa_1}\bigg)^ke^{-\frac{\kappa_1B}{(k+1)}}.\nonumber
\end{eqnarray}
Similarly, under Assumption \ref{Assumption:A3}, it holds that
\begin{eqnarray*}
	\ev\bigg(|\epsilon|^kI(|\epsilon|>B)\bigg)\leq \kappa_4\bigg(\frac{k+1}{\kappa_3}\bigg)^ke^{-\frac{\kappa_3B}{(k+1)}}.\nonumber
\end{eqnarray*}
\end{lemma}
\begin{proof}
For simplicity, we omit the subscript and write $\zeta_{s}$ as $\zeta$. By Assumption \ref{Assumption:AC}\ref{A1:b}, we can see $|\zeta|\leq |X_s|+C$. Therefore, it follows from  Assumption \ref{Assumption:AC}\ref{Assumption:A1} that $\ev(e^{\kappa_1|\zeta|})\leq \kappa_2e^{\kappa_1C}$. Moreover, because $I(|\zeta|>B)\leq e^{\kappa_1|\zeta|/(k+1)}e^{-\kappa_1B/(k+1)}$, we have
\begin{eqnarray}
	\ev\bigg(|\zeta|^kI(|\zeta|>B)\bigg)&\leq&\ev\bigg((k+1)^k\kappa_1^{-k}e^{\frac{k\kappa_1|\zeta|}{k+1}}I(|\zeta|>B)\bigg)\nonumber\\
	&\leq&\ev\bigg((k+1)^k\kappa_1^{-k}e^{\frac{k\kappa_1|\zeta|}{k+1}}e^{\frac{\kappa_1|\zeta|}{(k+1)}}e^{-\frac{\kappa_1B}{(k+1)}}\bigg)\nonumber\\
	&=&\bigg(\frac{k+1}{\kappa_1}\bigg)^k\ev(e^{\kappa_1|\zeta|})e^{-\frac{\kappa_1B}{(k+1)}}\nonumber\\
	&\leq& \kappa_2e^{\kappa_1C}\bigg(\frac{k+1}{\kappa_1}\bigg)^ke^{-\frac{\kappa_1B}{(k+1)}}.\nonumber
\end{eqnarray}
\end{proof}

\begin{lemma}\label{lemma:rate:of:convergence}
Under Assumption \ref{Assumption:AC}, if $LW=o(\sqrt{n})$ and $LWd=o(n)$, then 
\begin{eqnarray}
	\|\widehat{f}_s-f_{0,s}\|_n=O_P(r_n),\quad \textrm{ for all } s=1,\ldots, q,\nonumber
\end{eqnarray}
where 
\begin{eqnarray}
	r_n=\log^4(n)\sqrt{\frac{L(LW^2+Wd)}{n}\log(LW^2+Wd)}+\Delta_n,\nonumber
\end{eqnarray}
where $\Delta_n$ is the the approximation error to approximation $\bff_0$ using class $\mcF_{d,q}(L, W)$.
\end{lemma}
\begin{proof}
Let $\bff^*=(f_1^*, f_2^*, \ldots, f_q^*)^\top \in \mcF_{d,q}(L, W)$ in Lemma \ref{theorem:approximation:deep} such that $\sum_{s=1}^q\|f_s^*-f_{0,s}\|_\infty\leq \Delta_n$. Let $\zeta_s=X_s-f_{0,s}(\bfZ)$ and $\zeta_{is}=X_{is}-f_{0,s}(\bfZ_i)$ be the error terms. In the following, we divide the proof into 4 steps. 
\newline
\noindent\textbf{Step 1:} For any vector function $\bff=(f_1,\ldots, f_q)^\top: \mathbb{R}^d \to \mathbb{R}^q$, it follows that
\begin{eqnarray*}
\frac{1}{n}\sum_{i=1}^n\|\bfX_i-\bff(\bfZ_i)\|_2^2&=&\sum_{s=1}^q\|X_s-f_s\|_n^2\nonumber\\
&=&
\sum_{s=1}^q\|\zeta_s+f_{0,s}-f_s\|_n^2\nonumber\\
&=&\sum_{s=1}^q\|f_s-f_{0,s}\|_n^2+2\sum_{s=1}^q\langle f_{0,s}-f_s, \zeta_s  \rangle_n+\sum_{s=1}^q\|\zeta_s\|_n^2.
\end{eqnarray*}
By definition of $\widehat{\bff}$ and $\bff^*$, it follows that $\sum_{i=1}^n\|\bfX_i-\widehat{\bff}(\bfZ_i)\|_2^2\leq \sum_{i=1}^n\|\bfX_i-{\bff^*}(\bfZ_i)\|_2^2$. As a consequence of above two equations, we have
\begin{eqnarray}
\sum_{s=1}^q\|\widehat{f}_s-f_{0,s}\|_n^2&\leq&2\sum_{s=1}^q\langle \widehat{f}_s- f_{0,s}, \zeta_s  \rangle_n+2\sum_{s=1}^q\langle f_{0,s}-f^*_s, \zeta_s  \rangle_n+\sum_{s=1}^q\|f_s^*-f_{0,s}\|_n^2\nonumber\\
&\leq&2\sum_{s=1}^q\langle \widehat{f}_s- f_{0,s}, \zeta_s  \rangle_n2+2\sum_{s=1}^q\langle f_{0,s}-f^*_s, \zeta_s  \rangle_n+\Delta_n^2\label{eq:lemma:rate:of:convergence:eq1}\\
&\leq&2\sum_{s=1}^q\langle \widehat{f}_s- f_{0,s}, \zeta_s  \rangle_n+2\sum_{s=1}^q\|f_{0,s}-f^*_s\|_n \|\zeta_s  \|_n+\Delta_n^2\nonumber\\
&\leq&2\sum_{s=1}^q\langle \widehat{f}_s- f_{0,s}, \zeta_s  \rangle_n+2\Delta_n\sum_{s=1}^q \|\zeta_s  \|_n+\Delta_n^2\nonumber\\
&\leq& 2\sum_{s=1}^q\|\widehat{f}_s- f_{0,s}\|_n \|\zeta_s\|_n+2\Delta_n\sum_{s=1}^q \|\zeta_s  \|_n+\Delta_n^2\nonumber\\
&\leq&  2\sqrt{\sum_{s=1}^q\|\widehat{f}_s- f_{0,s}\|_n^2} \sqrt{\sum_{s=1}^q\|\zeta_s\|_n^2}+2\Delta_n\sum_{s=1}^q \|\zeta_s  \|_n+\Delta_n^2.\label{eq:lemma:rate:of:convergence:eq2}
\end{eqnarray}
By Chebyshev’s inequality, it follows that
\begin{eqnarray*}
\pr\bigg(\|\zeta_s\|_n^2>\log^2(n) \textrm{ for some } s=1,\ldots, q\bigg)\leq \sum_{s=1}^q\pr\bigg(\|\zeta_s\|_n^2>\log^2(n)\bigg)\leq \frac{\sum_{s=1}^q\ev(\zeta_s^2)}{\log^2(n)},\nonumber
\end{eqnarray*}
which further implies that with probability at least $1-\log^{-2}(n)\sum_{s=1}^q\ev(\zeta_s^2)$, it holds that
\begin{equation}\label{eq:lemma:rate:of:convergence:eq3}
	\max_{1\leq s\leq q}\|\zeta_s\|_n\leq \log(n).
\end{equation}
Now define event $F_1:=\{\max_{1\leq s\leq q}\|\zeta_s\|_n\leq \log(n)\}$, then by (\ref{eq:lemma:rate:of:convergence:eq3}), it yields that:
\begin{equation}\label{eq:lemma:rate:of:convergence:F1}
	 \pr(F_1)>1-\log^{-2}(n)\sum_{s=1}^q\ev(\zeta_s^2).
\end{equation}
According to (\ref{eq:lemma:rate:of:convergence:eq2}), we conclude that on event $F_1$, the following holds:
\begin{eqnarray*}
	\sum_{s=1}^q\|\widehat{f}_s-f_{0,s}\|_n^2\leq  2\sqrt{\sum_{s=1}^q\|\widehat{f}_s- f_{0,s}\|_n^2}\sqrt{q}\log(n)+2q\log(n)\Delta_n+\Delta_n^2,
\end{eqnarray*}
which further leads to
\begin{eqnarray}\label{eq:lemma:rate:of:convergence:step1:conclusion}
	\sqrt{\sum_{s=1}^q\|\widehat{f}_s-f_{0,s}\|_n^2}\leq 4\sqrt{q}\log(n)+\sqrt{4q\log(n)}\sqrt{\Delta_n}+\sqrt{2}\Delta_n\leq 8q\log(n)
\end{eqnarray} 
where the fact that $\Delta_n\leq \log(n)$ is used.
\newline
\noindent \textbf{Step 2:}
Consider the class $\mcG_s^r=\{f-f_{0,s}: f\in \mcF_{d,1}(L, W), \|f-f_0\|_n\leq r\}$ for some $1/n\leq r\leq n$, in the following we will establish a bound for 
\begin{equation*}
	\sup_{f-f_0\in \mcG_s^r}\langle f-f_{0,s}, \zeta_s\rangle_n.
\end{equation*}
For a diverging deterministic sequence $m_n$, we define $\zeta_{is1}=\zeta_{is}I(|\zeta_{is}|\leq m_n)$ and $\zeta_{is2}=\zeta_{is}I(|\zeta_{is}|> m_n)$. Therefore, we have
\begin{eqnarray}
	\sup_{f-f_{0,s}\in \mcG_s^r}\langle f-f_{0,s}, \zeta_s\rangle_n
	&\leq& \sup_{f-f_{0,s}\in \mcG_s^r}\frac{1}{n}\sum_{i=1}^n(f(\bfZ_i)-f_{0,s}(\bfZ_i))(\zeta_{is1}-\ev(\zeta_{is1}|\bfZ_i))\nonumber\\
	&&+\sup_{f-f_{0,s}\in \mcG_s^r}\frac{1}{n}\sum_{i=1}^n(f(\bfZ_i)-f_{0,s}(\bfZ_i))(\zeta_{is2}-\ev(\zeta_{is2}|\bfZ_i))\nonumber\\
	&=&S_1+S_2.\label{eq:lemma:rate:of:convergence:eq5}
\end{eqnarray}
For the first term, we notice that $|\zeta_{is1}|\leq m_n$, by Lemma \ref{lemma:dudley:inequality}, we have
\begin{eqnarray}
	\ev_{\mathbb{Z}}(S_1)\leq 4\epsilon\sqrt{\ev(\zeta_s^2|\bfZ)}+24m_n\int_{\epsilon}^{r}\sqrt{\frac{\log \mcN(x, \mcG_s^r, \|\cdot\|_n)}{n}}dx\nonumber.
\end{eqnarray}
By Lemma \ref{lemma:entroy:bound} and the fact that $\max_{1\leq i\leq n}|f(\bfZ_i)-f_{0,s}(\bfZ_i)|\leq nr\leq n^2$, it follows that 
\begin{eqnarray}
	\log \mcN(x, \mcG_s^r, \|\cdot\|_n)\leq cL(LW^2+Wd)\log(LW^2+Wd)\log\bigg(\frac{2en^3}{x}\bigg).\nonumber
\end{eqnarray}
Combining above two inequality and choosing
\begin{eqnarray}
	\epsilon=r\sqrt{\frac{cL(LW^2+Wd)}{n}}\leq r\quad \textrm{ with }\quad cL(LW^2+Wd)\geq 1,\nonumber
\end{eqnarray}
we have
\begin{eqnarray}
	\ev_{\mathbb{Z}}(S_1)&\leq& 4\epsilon\sqrt{\frac{1}{n}\sum_{i=1}^n\ev(\zeta_{is}^2|\bfZ_i)}+24m_nr\sqrt{\frac{\log \mcN(\epsilon, \mcG_s^r, \|\cdot\|_n)}{n}}\nonumber\\
	&\leq& 4r\sqrt{\frac{cL(LW^2+Wd)}{n}}\sqrt{\frac{1}{n}\sum_{i=1}^n\ev(\zeta_{is}^2|\bfZ_i)}\nonumber\\
	&&+24m_nr\sqrt{\frac{cL(LW^2+Wd)}{n}\log(LW^2+Wd)\log\bigg(\frac{2en^3}{\epsilon}\bigg)}\nonumber\\
		&\leq& 4r\sqrt{\frac{cL(LW^2+Wd)}{n}}\sqrt{\frac{1}{n}\sum_{i=1}^n\ev(\zeta_{is}^2|\bfZ_i)}\nonumber\\
	&&+24m_nr\sqrt{\frac{cL(LW^2+Wd)}{n}\log(LW^2+Wd)\log(2en^5)},
	\label{eq:lemma:rate:of:convergence:eq6}
\end{eqnarray}
where we use the facts that $r\geq 1/n$ and $\epsilon\geq rn^{-1/2}$. For the second term, by Cauchy–Schwarz inequality, we have
\begin{eqnarray}
	S_2&\leq&  \sqrt{\frac{1}{n}\sum_{i=1}^n\bigg(\zeta_{is2}-\ev(\zeta_{is2}|\bfZ_i)\bigg)^2}\sup_{f-f_{0,s}\in \mcG_s^r} \|f-f_{0,s}\|_n\leq r\sqrt{\frac{1}{n}\sum_{i=1}^n\bigg(\zeta_{is2}-\ev(\zeta_{is2}|\bfZ_i)\bigg)^2}.\nonumber
\end{eqnarray}
Therefore, we conclude that
\begin{eqnarray}
	\ev_{\mathbb{Z}}(S_2)\leq r \sqrt{\ev_{\mathbb{Z}}\bigg\{\frac{1}{n}\sum_{i=1}^n\bigg(\zeta_{is2}-\ev(\zeta_{is2}|\bfZ_i)\bigg)^2\bigg\}}\leq r\sqrt{\frac{1}{n}\sum_{i=1}^s\ev(\zeta_{is2}^2|\bfZ_i)}\label{eq:lemma:rate:of:convergence:eq7}.
\end{eqnarray}
Now combining (\ref{eq:lemma:rate:of:convergence:eq5})-(\ref{eq:lemma:rate:of:convergence:eq7}) with Chebyshev's inequality, we have 
\begin{eqnarray}
	&&\pr\bigg(\sup_{f-f_{0,s}\in \mcG_s^r}\langle f-f_{0,s}, \zeta_s\rangle_n>x\bigg)\nonumber\\
	&\leq& \frac{\ev(S_1+S_2)}{x}\nonumber\\
	&\leq& \frac{24m_nr}{x}\sqrt{\frac{cL(LW^2+Wd)}{n}\log(LW^2+Wd)\log(2en^5)\bigg)}\nonumber\\
	&&+\frac{4r}{x}\sqrt{\frac{cL(LW^2+Wd)}{n}}\sqrt{\ev(\zeta_s^2)}+\frac{r}{x}\sqrt{\ev(|\zeta_s|^2I(|\zeta_s|>m_n))}\nonumber\\
&\leq& 	\frac{cr}{x} \bigg(28m_n \sqrt{\frac{L(LW^2+Wd)}{n}\log(LW^2+Wd)\log(2en^5)}+e^{-\kappa_1m_n/6}\bigg),\nonumber
\end{eqnarray}
where we used the fact that $\ev(\zeta_s^2)\leq m_n^2$ and Lemma \ref{lemma:moment:bound}. As a consequence,  if $r>1/n$, we conclude that with probability at least $1-q\log^{-2}(n)$, the following holds:
\begin{eqnarray}
\sup_{f-f_{0, s}\in \mcG_s^r}\langle f-f_{0, s}, \zeta_s\rangle_n\leq rV_n\quad \textrm{ for all } s=1,\ldots, q,
\label{eq:lemma:rate:of:convergence:step2:conclusion}
\end{eqnarray}
where
\begin{eqnarray}
V_n:=c\log^2(n)\bigg(28m_n \sqrt{\frac{L(LW^2+Wd)}{n}\log(LW^2+Wd)\log(2en^5)}+e^{-\kappa_1m_n/6}\bigg)\nonumber.
\end{eqnarray}
\noindent \textbf{Step 3:}
By direct calculation, we have
\begin{eqnarray}
	\ev_{\mathbb{Z}}(|\langle f_s^*-f_{0,s}, \zeta_s  \rangle_n|^2)= \frac{1}{n^2}\sum_{i=1}^n|f_s^*(\bfZ)-f_{0,s}(\bfZ)|^2\ev(\zeta_{is}^2|\bfZ_i)\leq \frac{\Delta_n^2}{n^2}\sum_{i=1}^n\ev(\zeta_{is}^2|\bfZ_i).\nonumber
\end{eqnarray}
Therefore, by Chebyshev's inequality, it follows that 
\begin{eqnarray}
	\pr\bigg(|\langle f_s^*-f_{0,s}, \zeta_s  \rangle_n|>\log(n)\frac{\Delta_n}{\sqrt{n}}\bigg)\leq  \frac{1}{\log^2(n)}\ev(|\zeta_s|^2).\nonumber
\end{eqnarray}
Define event 
\begin{eqnarray*}
F_2:=\bigg\{|\langle f_s^*-f_{0,s}, \zeta_s  \rangle_n|\leq\log(n)\frac{\Delta_n}{\sqrt{n}}\textrm{ for all } s=1,\ldots, q\bigg\},
\end{eqnarray*}
then above inequality implies 
\begin{equation}\label{eq:lemma:rate:of:convergence:step3:conclusion}
	\pr(F_2)\geq 1-\log^{-2}(n)\sum_{s=1}^q\ev(\zeta_s^2).
\end{equation}
\newline
\noindent \textbf{Step 4:}
For any positive constant $r>1/n$, we define event $E_r=\{\sup_{f-f_{0, s}\in \mcG_s^r}\langle f-f_{0, s}, \zeta_s\rangle_n\leq rV_n \textrm{ for all } s=1,\ldots, q\}$, then (\ref{eq:lemma:rate:of:convergence:step2:conclusion}) leads to
\begin{equation}\label{eq:lemma:rate:of:convergence:Er}
	 \pr(E_r)>1-q\log^{-2}(n)
\end{equation}

We choose $T:=\ceil{\log_2(8qn\log(n))}$ and some positive $\bar{r}>1/n$ such that  
\begin{eqnarray}
	 8q\log(n)\leq 2^T/n\leq 2^T\bar{r}< n.\nonumber
\end{eqnarray}
Define event $F_3=\cap_{j=1}^TE_{2^j\bar{r}}\cap F_1\cap F_2$. By (\ref{eq:lemma:rate:of:convergence:step1:conclusion}) and inequality above, it follows that $\|\widehat{f}_s-f_{0,s}\|_n\leq 2^T\bar{r}$ for all $s=1,\ldots, q$ on event $F_3$. Now suppose $2^{j-1}\bar{r} \leq \|\widehat{f}_s-f_{0,s}\|_n\leq 2^j\bar{r}$ for some $j=1,\ldots, T$. By (\ref{eq:lemma:rate:of:convergence:eq1}), if  $\bar{r}>1/n$, on event $F_3$, it holds that
\begin{eqnarray}
{\sum_{s=1}^q\|\widehat{f}_s-f_{0,s}\|_n^2}&\leq& \sum_{s=1}^q\langle \widehat{f}_s- f_{0,s}, \zeta_s  \rangle_n+2\sum_{s=1}^q\langle f_{0,s}-f^*_s, \zeta_s  \rangle_n+\Delta_n^2\nonumber\\
&\leq&\sum_{s=1}^q\sup_{f-f_{0, s}\in \mcG_s^{2^j\bar{r}}}\langle f-f_{0, s}, \zeta_s\rangle_n+2q\log(n)\frac{\Delta_n}{\sqrt{n}}+\Delta_n^2\nonumber\\
&\leq& q2^j\bar{r}V_n+2q\log(n)\frac{\Delta_n}{\sqrt{n}}+\Delta_n^2.\nonumber
\end{eqnarray}
Therefore, if we choose $\bar{r}$ such that
\begin{eqnarray}
	\bar{r}\geq 8qV_n, \quad \bar{r}\geq 2q\frac{\log(n)}{\sqrt{n}}+2\Delta_n,\bar{r}>\frac{1}{n}\;\; \textrm{ and }\;\; 18q\log(n)\bar{r}<1,\label{eq:lemma:rate:of:convergence:step4:condition}
\end{eqnarray}
then it follows that
\begin{eqnarray}
	\sqrt{\sum_{s=1}^q\|\widehat{f}_s-f_{0,s}\|_n^2}\leq 2^{j-1}\bar{r}.\nonumber
\end{eqnarray}
As a consequence, on event $F_3$, we conclude that the following holds:
\begin{eqnarray}
	\|\widehat{f}_s-f_{0,s}\|_n\leq 2^T\bar{r} \Rightarrow \|\widehat{f}_s-f_{0,s}\|_n\leq 2^{T-1}\bar{r} \Rightarrow \cdots\Rightarrow \|\widehat{f}_s-f_{0,s}\|_n\leq \bar{r} \quad \textrm{ for all } s=1,\ldots, q.\nonumber
\end{eqnarray}
Specifically, we can choose 
\begin{eqnarray}
	\bar{r}=8q\bigg(V_n+\frac{\log(n)}{\sqrt{n}}+\Delta_n\bigg),
\end{eqnarray}
which will satisfy (\ref{eq:lemma:rate:of:convergence:step4:condition}). Moreover, by (\ref{eq:lemma:rate:of:convergence:F1}), (\ref{eq:lemma:rate:of:convergence:Er}) and  (\ref{eq:lemma:rate:of:convergence:step3:conclusion}), it follows that
\begin{eqnarray*}
	\pr(F_3)&\geq& 1-\sum_{j=1}^T\pr(E_{2^j\bar{r}}^c)-\pr(F_1^c)-\pr(F_2^c)\nonumber\\
	&\geq&1-qT\log^{-2}(n)-2\log^{-2}(n)\sum_{s=1}^q\ev(\zeta_s^2)\nonumber\\
	&\geq&1-\log^{-2}(n)\bigg(q+q\log_2(8qn)+q\log_2(\log(n))+2\sum_{s=1}^q\ev(\zeta_s^2)\bigg)\to 1.
\end{eqnarray*}
Finally, combining above, and choose $m_n=6\log(n)/\kappa_1$ in the expression of $V_n$, we finish the proof.
\end{proof}

\begin{proof}[\textbf{Proof of Theorem \ref{theorem:rate:deep}}]
By Lemmas \ref{theorem:approximation:deep} and \ref{lemma:rate:of:convergence}, it holds that
\begin{eqnarray}
	\|\widehat{f}_s-f_{0,s}\|_n=O_P\bigg(\log^4(n)\sqrt{\frac{L(LW^2+Wd)}{n}\log(LW^2+Wd)}+\bigg[\frac{LW}{\log(L)\log(W)}\bigg]^{-\frac{2p^*}{t^*}}\bigg).\nonumber
\end{eqnarray}
Since $LW=O(\sqrt{n})$ and $LWd=o(n)$, we prove the desire result.
\end{proof}

To proceed, we recall the definition of $r_n$ in Lemma \ref{lemma:rate:of:convergence} that
\begin{eqnarray}
	r_n=\log^4(n)\sqrt{\frac{L(LW^2+Wd)}{n}\log(LW^2+Wd)}+\Delta_n.\nonumber
\end{eqnarray}
\begin{lemma}\label{lemma:difference:rate:2}
Under Assumptions \ref{Assumption:AC} and \ref{Assumption:A3}, if $LW=o(\sqrt{n})$ and $LWd=o(n)$, then the following holds:
$$\bigg\|\frac{1}{n}\sum_{i=1}^n\widehat{\bfX}_i{\bfX}_i^\top-\frac{1}{n}\sum_{i=1}^n\bfD_i\bfD_i^\top\bigg\|_F=O_P(r_n).$$
\end{lemma}
\begin{proof}
By triangle inequality, it follows that
\begin{eqnarray*}
	\bigg\|\frac{1}{n}\sum_{i=1}^n\widehat{\bfX}_i{\bfX}_i^\top-\frac{1}{n}\sum_{i=1}^n\bfD_i\bfD_i^\top\bigg\|_F\nonumber&\leq&\bigg\|\frac{1}{n}\sum_{i=1}^n{\bfX}_i(\widehat{\bfX}_i-\bfD_i)^\top\bigg\|_F+\bigg\|\frac{1}{n}\sum_{i=1}^n{\bfD}_i({\bfX}_i-\bfD_i)^\top\bigg\|_F\nonumber\\
	&:=&R_1+R_2,
\end{eqnarray*}
where the definition of $R_1, R_2$ is straight forward in the context. By Lemma \ref{proposition:cauchy:F:norm}, it follows that
\begin{eqnarray}
	R_1\leq \sqrt{\frac{1}{n}\sum_{i=1}^n\|\bfX_i\|_2^2\times \frac{1}{n}\sum_{i=1}^n\|\widehat{\bfX}_i-\bfD_i\|_2^2}.\nonumber
\end{eqnarray}
Since $\sum_{i=1}^n\|\bfX_i\|_2^2=\sum_{s=1}^q \sum_{i=1}^nX_{s,i}^2$, Assumption \ref{Assumption:AC}\ref{Assumption:A1} and C.L.T together imply $\sum_{i=1}^n\|\bfX_i\|_2^2/n=O_P(1)$. By Lemma \ref{lemma:rate:of:convergence} and the definition of $\widehat{\bfX}_i, \bfD_i$,  we have
\begin{eqnarray}
\frac{1}{n}\sum_{i=1}^n\|\widehat{\bfX}_i-\bfD_i\|_2^2=\frac{1}{n}\sum_{i=1}^n\sum_{s=1}^{q}|\widehat{f}_{s}(\bfZ_i)-f_{0,s}(\bfZ_i)|^2=\sum_{s=1}^{q}\|\widehat{f}_{s}-f_{0,s}\|_n^2=O_P(r_n^2).\nonumber
\end{eqnarray} 
As a consequence of above, we conclude $R_1=O_P(r_n)$. In the following, we will analyse $R_2$. By straightforward calculation, it is not difficult to show that
\begin{eqnarray*}
	R_2^2\leq \sum_{1\leq s, k\leq q}\bigg|\frac{1}{n}\sum_{i=1}^n f_{0,s}(\bfZ_i)[X_{ik}-f_{0,k}(\bfZ_i)]\bigg|^2.
\end{eqnarray*}
Since $\ev(X_{ik}|\bfZ_i)=f_{0,k}(\bfZ_i)$ for $k=1,\ldots, q$, it follows that $\ev\{f_{0,k}(\bfZ_i)[X_{ik}-f_{0,k}(\bfZ_i)]\}=0$ for $1\leq s, k\leq q$. Therefore, by Assumption \ref{Assumption:AC}\ref{Assumption:A1}, we conclude that
\begin{eqnarray*}
\ev(R_2^2)&\leq& \frac{1}{n}\sum_{1\leq s,k\leq q}\ev\{f_{0,s}^2(\bfZ)[X_{k}-f_{0,k}(\bfZ)]^2\}=O(n^{-1})=O(r_n^2).
\end{eqnarray*}
Combining above, we finish the proof.
\end{proof}

\begin{lemma}\label{lemma:difference:rate:3}
Under Assumptions \ref{Assumption:AC} and \ref{Assumption:A3}, if $LWr_n\log^3(n)=o(1)$ and  $\sqrt{LWd}r_n\log^3(n)=o(1)$, then 
$$\bigg\|\frac{1}{\sqrt{n}}\sum_{i=1}^n\widehat{\bfX}_i\epsilon_i-\frac{1}{\sqrt{n}}\sum_{i=1}^n{\bfD}_i\epsilon_i\bigg\|_2^2=o_P(1)$$
\end{lemma}
\begin{proof}
By definition, it follows that
\begin{eqnarray*}
	\bigg\|\frac{1}{\sqrt{n}}\sum_{i=1}^n\widehat{\bfX}_i\epsilon_i-\frac{1}{\sqrt{n}}\sum_{i=1}^n{\bfD}_i\epsilon_i\bigg\|^2_2=\sum_{s=1}^q\bigg|\frac{1}{\sqrt{n}}\sum_{i=1}^n\bigg(\widehat{f}_s(\bfZ_i)-f_{0,s}(\bfZ_i)\bigg)\epsilon_i\bigg|^2.
\end{eqnarray*}
By Lemma \ref{lemma:rate:of:convergence}, it suffices to show the following holds for all $c>0$:
\begin{equation}\label{eq:lemma:difference:rate:3:goal}
	\sup_{f-f_{0,s}\in \mcH_s^b}\bigg|\frac{1}{n}\sum_{i=1}^n\bigg(f(\bfZ_i)-f_{0,s}(\bfZ_i)\bigg)\epsilon_i\bigg|=o_P(n^{-1/2}),\quad \textrm{ for all } s=1,\ldots, q,
\end{equation}
where $\mcH_s^b=\{f-f_{0,s} :  f\in \mcF_{d,1}(L, W), \|f-f_{0,s}\|_n\leq  br_n\}$. Now define 
\begin{eqnarray}
	\rho_{i1}=\epsilon_i I(|\epsilon_i|\leq m_n)-\ev\bigg(\epsilon_i I(|\epsilon_i|\leq m_n)\bigg| \bfZ_i\bigg)\quad \textrm{ and}\quad \rho_{i2}=\epsilon_i I(|\epsilon_i|> m_n)-\ev\bigg(\epsilon_i I(|\epsilon_i|> m_n)\bigg| \bfZ_i\bigg),\nonumber
\end{eqnarray}
where $m_n$ is a deterministic sequence to be specified later. Therefore, (\ref{eq:lemma:difference:rate:3:goal}) can be bounded by $S_1+S_2$, where
\begin{eqnarray}
	S_1=\sup_{f-f_{0,s}\in \mcH_s^b}\bigg|\frac{1}{n}\sum_{i=1}^n\bigg(f(\bfZ_i)-f_{0,s}(\bfZ_i)\bigg)\rho_{i1}\bigg|\quad \textrm{ and }\quad S_2=\sup_{f-f_{0,s}\in \mcH_s^b}\bigg|\frac{1}{n}\sum_{i=1}^n\bigg(f(\bfZ_i)-f_{0,s}(\bfZ_i)\bigg)\rho_{i2}\bigg|.\nonumber
\end{eqnarray}
Since $|\rho_{i1}|\leq 2m_n$ and $\ev(\rho_{i1}|\bfZ	_i)=0$, therefore, by Lemma \ref{prop:conditional:indpendence:2} and Lemma \ref{lemma:dudley:inequality}, it follows that
\begin{eqnarray}
	\ev_{\mathbb{Z}}(S_1)&\leq&4x\sqrt{\frac{1}{n}\sum_{i=1}^n\ev(\rho_{i1}^2 |\bfZ_i)}+24m_n\int_{x}^{br_n}\sqrt{\frac{\log \mcN(u, \mcH_s^b, \|\cdot\|_n)}{n}}du\quad \textrm{ for all } 0<x<br_n.\nonumber
\end{eqnarray}
By Lemma \ref{lemma:entroy:bound} and the fact that $\max_{1\leq i\leq n}|f(\bfZ_i)-f_{0,s}(\bfZ_i)|\leq bnr_n\leq bn$, it follows that 
\begin{eqnarray}
	\log \mcN(x, \mcH_s^b, \|\cdot\|_n)\leq cL(LW^2+Wd)\log(LW^2+Wd)\log\bigg(\frac{2ben^2}{x}\bigg).\nonumber
\end{eqnarray}
Combining above two inequality and choosing
\begin{eqnarray}
	x=\epsilon:=br_n\sqrt{\frac{cL(LW^2+Wd)}{n}}\leq br_n\quad \textrm{ with }\quad cL(LW^2+Wd)\geq 1,\nonumber
\end{eqnarray}
we have
\begin{eqnarray}
	\ev_{\mathbb{Z}}(S_1)&\leq& 4\epsilon\sqrt{\frac{1}{n}\sum_{i=1}^n\ev(\rho_{i1}^2|\bfZ_i)}+24bm_nr_n\sqrt{\frac{\log \mcN(\epsilon, \mcH_s^b, \|\cdot\|_n)}{n}}\nonumber\\
	&\leq& 4br_n\sqrt{\frac{cL(LW^2+Wd)}{n}}\sqrt{\frac{1}{n}\sum_{i=1}^n\ev(\rho_{i1}^2|\bfZ_i)}\nonumber\\
	&&+24bm_nr_n\sqrt{\frac{cL(LW^2+Wd)}{n}\log(LW^2+Wd)\log\bigg(\frac{2ben^2}{\epsilon}\bigg)}\nonumber\\
		&\leq& 4br_n\sqrt{\frac{cL(LW^2+Wd)}{n}}\sqrt{\frac{1}{n}\sum_{i=1}^n\ev(\rho_{i1}^2|\bfZ_i)}\nonumber\\
	&&+24bm_nr_n\sqrt{\frac{cL(LW^2+Wd)}{n}\log(LW^2+Wd)\log(2en^3)},\nonumber
\end{eqnarray}
where we use the facts that $r_n\geq n^{-1/2}$ and $\epsilon\geq b/n$. Therefore, we conclude that
\begin{eqnarray}
	\ev(S_1)&\leq& 4br_n\sqrt{\frac{cL(LW^2+Wd)}{n}}\sqrt{\ev(\epsilon^2)}\nonumber\\
	&&+24bm_nr_n\sqrt{\frac{cL(LW^2+Wd)}{n}\log(LW^2+Wd)\log(2en^3)}\nonumber\\
	&=&O\bigg(m_nr_n\sqrt{\frac{L(LW^2+Wd)}{n}\log(LW^2+Wd)\log(n)}\bigg).\label{eq:lemma:difference:rate:3:eq1}
\end{eqnarray}
For the second term, by Cauchy–Schwarz inequality, it follows that
\begin{eqnarray}
	S_2\leq \sqrt{\frac{1}{n}\sum_{i=1}^n\bigg|\rho_{i2}-\ev(\rho_{i2}|\bfZ_i)\bigg|^2}\sup_{f-f_{0,s}\in\mcH_s^b}\|f-f_{0,s}\|_n.\nonumber
\end{eqnarray}
By taking conditional expectation, we conclude that
\begin{align}
	\ev(S_2)=\ev(\ev_{\mathbb{Z}}(S_2))\leq br_n\ev\bigg(\sqrt{\frac{1}{n}\sum_{i=1}^n\ev(\rho_{i2}^2|\bfZ_i)}\bigg)&\leq br_n\sqrt{\ev(\epsilon^2I(|\epsilon|>m_n))}\nonumber\\
	&\leq br_n\sqrt{\kappa_4\bigg(\frac{3}{\kappa_3}\bigg)^3e^{-\kappa_3m_n/3}},\label{eq:lemma:difference:rate:3:eq2}
\end{align}
where the last inequality comes from Lemma \ref{lemma:moment:bound}.
As a consequence of (\ref{eq:lemma:difference:rate:3:eq1}) and (\ref{eq:lemma:difference:rate:3:eq2}), we conclude that
\begin{eqnarray}
	&&\sup_{f-f_{0,s}\in \mcH_s^b}\bigg|\frac{1}{n}\sum_{i=1}^n\bigg(f(\bfZ_i)-f_{0,s}(\bfZ_i)\bigg)\epsilon_i\bigg|\nonumber\\
	&=&O_P\bigg(m_nr_n\sqrt{\frac{L(LW^2+Wd)}{n}\log(LW^2+Wd)\log(n)}+r_ne^{-\kappa_3m_n/6}\bigg)\nonumber\\
	&=&O_P\bigg(m_nr_n\log(n)\sqrt{\frac{L(LW^2+Wd)}{n}}+r_ne^{-\kappa_3m_n/6}\bigg),\label{eq:lemma:difference:rate:3:eq3}
\end{eqnarray}
where we use the fact that $\log(LW)=O(\log(n))$.
Notice by the rate conditions given, we always can choose the sequence $m_n$ such that, the rate of (\ref{eq:lemma:difference:rate:3:eq3}) is of order $O_P(n^{-1/2})$
\end{proof}

\begin{lemma}\label{theorem:asymptotic:distribution}
Under Assumptions \ref{Assumption:AC} and \ref{Assumption:A3}, if  $LWr_n\log^3(n)=o(1)$ and  $\sqrt{LWd}r_n\log^3(n)=o(1)$, then it follows that
\begin{equation*}
	\sqrt{n}(\widehat{\beta}-\beta_0)\cid \textrm{N}(0, \sigma_\epsilon^2 \ev^{-1}(\bfD\bfD^\top)).
\end{equation*}
\end{lemma}   
\begin{proof}
By simple calculation, it follows that
\begin{eqnarray*}
	\sqrt{n}(\widehat{\beta}-\beta_0)&=&\sqrt{n}\bigg[\bigg(\frac{1}{n}\sum_{i=1}^n\widehat{\bfX}_i{\bfX}_i^\top\bigg)^{-1}\frac{1}{n}\sum_{i=1}^n\widehat{\bfX}_iY_i-\beta_0\bigg]\nonumber\\
	&=&\sqrt{n}\bigg[\bigg(\frac{1}{n}\sum_{i=1}^n\widehat{\bfX}_i{\bfX}_i^\top\bigg)^{-1}\frac{1}{n}\sum_{i=1}^n\widehat{\bfX}_i\bfX_i^\top \beta_0-\beta_0\bigg]+\bigg(\frac{1}{n}\sum_{i=1}^n\widehat{\bfX}_i{\bfX}_i^\top\bigg)^{-1}\frac{1}{\sqrt{n}}\sum_{i=1}^n\widehat{\bfX}_i\epsilon_i\nonumber\\
	&:=&R_1+R_2,
\end{eqnarray*}
where the definition of $R_1, R_2$ is straightforward in the context. By Lemma \ref{lemma:difference:rate:2} and Assumption \ref{Assumption:A3}\ref{A3:b}, $\sum_{i=1}^n\widehat{\bfX}_i{\bfX}_i^\top/n$ is asymptotically invertible, and thus $R_1=0$. Furthermore, one can verify that the  C.L.T holds for $\sum_{i=1}^n\bfD_i\epsilon_i/\sqrt{n}$ using Assumption \ref{Assumption:A3}. As a consequence of Slutsky's Theorem, Lemmas \ref{lemma:difference:rate:2} and \ref{lemma:difference:rate:3}, we can show $R_2\to  \textrm{N}(0, \sigma_\epsilon^2 \ev^{-1}(\bfD\bfD^\top))$
\end{proof}

\begin{proof}[\textbf{Proof of Theorem \ref{theorem:asymptotic:distribution:deep}}]
By Theorem \ref{theorem:rate:deep}, it follows that
\begin{equation*}
r_n=O\bigg(\log^5(n)\sqrt{\frac{L^2W^2+LWd}{n}}+\log^{\frac{4p^*}{t^*}}(n)(LW)^{-\frac{2p^*}{t^*}}\bigg).
\end{equation*}
Therefore, if the following holds
\begin{eqnarray*}
	L^2W^2\log^8(n)=o(n^{1/2}), \quad LWd\log^8(n)=o(n^{1/2}),\quad (LW)^{1-2p^*/t^*}[\log(n)]^{3+4p^*/t^*}=o(1),
\end{eqnarray*}
the rate conditions in Lemma \ref{theorem:asymptotic:distribution} will be satisfied.
\end{proof}

\begin{proof}[\textbf{Proof of Lemma \ref{lemma:estimation:variance}}]
By Lemma \ref{lemma:difference:rate:2}, Assumption \ref{Assumption:A3}\ref{A3:b} and C.L.T, it is not difficult to see that
\begin{eqnarray*}
	\frac{1}{n}\sum_{i=1}^n\widehat{\bfX}_i\bfX_i^\top=\ev(\bfD\bfD^\top)+o_P(1).
\end{eqnarray*}
Moreover, since $\widehat{\beta}=\beta_0+o_P(1)$, we have
\begin{eqnarray*}
	\frac{1}{n}\sum_{i=1}^n\widehat{\epsilon}_i^2&=&\frac{1}{n}\sum_{i=1}^n|Y_i-\widehat{\beta}^\top\bfX_i|^2\nonumber\\
	&=&\frac{1}{n}\sum_{i=1}^n|Y_i-\beta_0^\top\bfX_i|^2+\frac{1}{n}\sum_{i=1}^n|\beta_0^\top\bfX_i-\widehat{\beta}^\top\bfX_i|^2+\frac{2}{n}\sum_{i=1}^n(Y_i-\beta_0^\top\bfX_i)(\beta_0^\top\bfX_i-\widehat{\beta}^\top\bfX_i)\nonumber\\
	&=&\frac{1}{n}\sum_{i=1}^n\epsilon_i^2+(\beta_0-\widehat{\beta})^\top\frac{1}{n}\sum_{i=1}^n\bfX_i\bfX_i^\top(\beta_0-\widehat{\beta})+(\beta_0-\widehat{\beta})^\top\frac{2}{n}\sum_{i=1}^n \epsilon_i\bfX_i
\end{eqnarray*}
\end{proof}

\subsection{Proof of Results in Section \ref{sec:extra:results}}\label{sec:proof:sec:5}
 For constant $C>0$, we define the truncation operator at level $C$ by $T_C(f)=fI(|f|\leq C)$, for real-valued function $f$. Therefore, $\widecheck{f}_s=T_{C_n}(\widehat{f}_s)$.
\begin{lemma}\label{lemma:convergence:truncated:DNN}
Under Assumption \ref{Assumption:AC}, if $C_n \to \infty$, $C_n=O(\log(n))$, $LW=o(\sqrt{n})$ and $LWd=o(n)$, then
\begin{eqnarray}
	\|\widecheck{f}_s-f_{0,s}\|_n=O_P(r_n)\;\;\textrm{ and }\;\; \|\widecheck{f}_s-f_{0,s}\|=O_P(r_n)\quad \textrm{ for all } s=1,\ldots, q.\nonumber
\end{eqnarray}
\end{lemma}
\begin{proof}
Since $\|f_{0,s}\|_\infty<C$ and $C_n\geq C$, it is not difficult to verify that $|\widecheck{f}_s(\bfz)-f_{0,s}(\bfz)|\leq |\widehat{f}_s(\bfz)-f_{0,s}(\bfz)|$. Therefore, by Lemma \ref{lemma:rate:of:convergence}, we conclude that
\begin{eqnarray}
	\|\widecheck{f}_s-f_{0,s}\|_n=O_P(r_n)\quad \textrm{ for all } s=1,\ldots, q.\nonumber
\end{eqnarray}
Moreover, by definition, it follows that $\|\widecheck{f}_s-f_{0,s}\|_\infty\leq C_n+C\leq 2C_n$. 
\newline
\noindent \textbf{Step 1:}
For fixed $a>1$, we define 
\begin{equation*}
	\mcH_{a,r}=\{T_{C_n}(f)-f_{0,s}: f\in \mcF_{d,1}(L, W),\|T_{C_n}(f)-f_{0,s}\|_n\leq ar_n  ,\|T_{C_n}(f)-f_{0,s}\|\leq r\}
\end{equation*}
and $\mcH_{a,r}^2=\{h^2: h\in \mcH_a\}$. Therefore, $\|T_{C_n}(f)-f_{0,s}\|_\infty\leq C_n+C\leq 2C_n$. Notice the map $x:\to x^2$ is Lipschitz with Lipschitz constant  $4C_n$ for $x\in [-2C_n, 2C_n]$,  as a consequence of  Lemma \ref{lemma:contraction:inequality}, we have $\ev(\mcRn\mcH_{a,r}^2)\leq 4C_n\ev(\mcRn\mcH_{a,r})$.  Since $\|T_{C_n}(f)-f_{0,s}\|_\infty\leq 2C_n$ and $\textrm{Var}[|T_{C_n}(f)(\bfZ)-f_{0,s}(\bfZ)|^2]\leq 4C_n^2\ev(|T_{C_n}(f)(\bfZ)-f_{0,s}(\bfZ)|^2)\leq 4C_n^2r^2$, 	by Lemma \ref{lemma:concentration:inequality}, with probability at least $1-2e^{-\eta}$, the following holds:
\begin{eqnarray}
	\sup_{T_{C_n}(f)-f_{0,s}\in \mcH_{a,r}}\bigg|\|T_{C_n}(f)-f_{0,s}\|_n^2-\|T_{C_n}(f)-f_{0,s}\|^2\bigg|&=&\sup_{h\in \mcH_{a,r}^2}\bigg|(\pr-\pr_n)[h]\bigg|\nonumber\\
	&\leq&3	\ev(\mcRn\mcH_{a,r}^2)+3C_nr\sqrt{\frac{\eta}{n}}+\frac{16C_n^2\eta}{3n}\nonumber\\
	&\leq& 12C_n\ev(\mcRn\mcH_{a,r})+3C_nr\sqrt{\frac{\eta}{n}}+\frac{6C_n^2\eta}{n}.\nonumber
\end{eqnarray}
Therefore, with probability at least $1-2e^{-\eta}$, it holds for all $T_{C_n}(f)-f_{0,s}\in \mcH_{a,r}$ that
\begin{eqnarray}
	\|T_{C_n}(f)-f_{0,s}\|^2&\leq&\|T_{C_n}(f)-f_{0,s}\|_n^2+12C_n\ev(\mcRn\mcH_{a,r})+3C_nr\sqrt{\frac{\eta}{n}}+\frac{6C_n^2\eta}{n}\nonumber\\
	&\leq& a^2r_n^2+12C_n\ev(\mcRn\mcH_{a,r})+3C_nr\sqrt{\frac{\eta}{n}}+\frac{6C_n^2\eta}{n}.\label{eq:lemma:convergence:truncated:DNN:step1:eq1}
\end{eqnarray}
By (\ref{eq:lemma:convergence:truncated:DNN:step1:eq1}), if $r$ is chosen such that
\begin{eqnarray}
	a^2r_n^2\leq \frac{r^2}{16},\quad 12C_n\ev(\mcRn\mcH_{a,r})\leq\frac{r^2}{16},\quad 3C_nr\sqrt{\frac{\eta}{n}}\leq \frac{r^2}{16} \;\textrm{ and }\;  \frac{6C_n^2\eta}{n}\leq \frac{r^2}{16},\label{eq:lemma:convergence:truncated:DNN:step1:eq2}
\end{eqnarray}
then we conclude that with probability at least $1-2e^{-\eta}$,
\begin{eqnarray}
	\|T_{C_n}(f)-f_{0,s}\|^2\leq \frac{r^2}{4}\;\;\;\; \textrm{ for all } T_{C_n}(f)-f_{0,s}\in \mcH_{a,r}.\label{eq:lemma:convergence:truncated:DNN:step1:eq3}
\end{eqnarray}
\noindent \textbf{Step 2:} Define event $F_1:=\{\|T_{C_n}(\widehat{f}_s)-f_{0,s}\|_n\leq ar_n\}$. Since $\|T_{C_n}(f)-f_{0,s}\|_\infty\leq 2C_n$, we conclude that, on event $F_1$, $T_{C_n}(\widehat{f}_s)-f_{0,s}\in \mcH_{a,  2C_n}$. If we choose some positive integer $S=\ceil{\log_2(2nC_n)}$ and real number $\bar{r}>n^{-1}$, then
\begin{eqnarray}
	2C_n \leq 2^S\bar{r}\nonumber
\end{eqnarray}
As a consequence, on event $F_1$, $\|T_{C_n}(\widehat{f}_s)-f_{0,s}\|\leq 2^S\bar{r}$. For $j=1,\ldots, S$, we further define the event 
\begin{eqnarray}
	E_j:=\bigg\{\|T_{C_n}(f)-f_{0,s}\|\leq 2^{j-1}\bar{r} \;\;\;\; \textrm{ for all } T_{C_n}(f)-f_{0,s}\in \mcH_{a,2^j\bar{r}}\bigg\}\quad \textrm{ for } j=1,\ldots, S.\nonumber
\end{eqnarray}
By (\ref{eq:lemma:convergence:truncated:DNN:step1:eq2}) and (\ref{eq:lemma:convergence:truncated:DNN:step1:eq3}), we have
\begin{eqnarray}
	\pr(E_j)\geq 1-2e^{-\eta}, \label{eq:lemma:convergence:truncated:DNN:step2:eq1}
\end{eqnarray}
if the following holds:
\begin{eqnarray}
	a^2r_n^2\leq \frac{2^{2j}\bar{r}^2}{16},\quad 12C_n\ev(\mcRn\mcH_{a,2^j\bar{r}})\leq\frac{2^{2j}\bar{r}^2}{16},\quad 3C_n\sqrt{\frac{\eta}{n}}\leq \frac{2^j\bar{r}}{16} \;\textrm{ and }\;  \frac{6C_n^2\eta}{n}\leq \frac{2^{2j}\bar{r}^2}{16}.\label{eq:lemma:convergence:truncated:DNN:step2:eq2}
\end{eqnarray}
As a consequence, on the event $F_2:=\cap_{j=1}^SE_j\cap F_1$, we conclude that the following holds:
\begin{eqnarray}
	\|T_{C_n}(\widehat{f}_s)-f_{0,s}\|\leq 2^S\bar{r} \Rightarrow \|T_{C_n}(\widehat{f}_s)-f_{0,s}\|\leq 2^{S-1}\bar{r} \Rightarrow \cdots \|T_{C_n}(\widehat{f}_s)-f_{0,s}\|\leq \bar{r}.\label{eq:lemma:convergence:truncated:DNN:step2:eq3}
\end{eqnarray}
\noindent \textbf{Step 3:} Combining (\ref{eq:lemma:convergence:truncated:DNN:step2:eq1})-(\ref{eq:lemma:convergence:truncated:DNN:step2:eq3}), we are ready to choose appropriate $\bar{r}$.

By Lemma \ref{lemma:dudley:inequality}, we have
\begin{eqnarray}
	\ev_{\mathbb{Z}}(\mcRn\mcH_{a,r})&\leq& \inf_{0<x<ar_n}\bigg\{4x+12\int_{x}^{ar_n}\sqrt{\frac{\log \mcN(u, \mcH_{a,r}, \|\cdot\|_n)}{n}}du\bigg\}\nonumber\\
	&\leq& 4\epsilon+12ar_n\sqrt{\frac{\log \mcN(\epsilon, \mcH_{a,r}, \|\cdot\|_n)}{n}} \;\;\textrm{ for any } 0<\epsilon<ar_n.\label{eq:lemma:convergence:truncated:DNN:step3:eq1}
\end{eqnarray}
Notice the facts that $\|T_{C_n}(f)-T_{C_n}(g)\|_n\leq \|f-g\|_n$ for all functions $f, g$, and that $\max_{1\leq i\leq n}|T_{C_n}(f)(\bfZ_i)-f_{0,s}(\bfZ_i)|\leq 2C_n\leq n^2$. Therefore, by Lemma \ref{lemma:entroy:bound}, we have
\begin{eqnarray}
	\log \mcN(\epsilon, \mcH_{a,r}, \|\cdot\|_n)\leq cL(LW^2+Wd)\log(LW^2+Wd)\log\bigg(\frac{2en^3}{\epsilon}\bigg).\nonumber
\end{eqnarray}
Now we choose 
\begin{eqnarray}
	\epsilon=ar_n\sqrt{\frac{cL(LW^2+Wd)}{n}}<ar_n\;\textrm{ with } cL(LW^2+Wd)\geq 1,\nonumber
\end{eqnarray}
then (\ref{eq:lemma:convergence:truncated:DNN:step3:eq1}) becomes
\begin{eqnarray}
	\ev_{\mathbb{Z}}(\mcRn\mcH_{a,r})&\leq& 4ar_n\sqrt{\frac{cL(LW^2+Wd)}{n}}+12ar_n\sqrt{\frac{cL(LW^2+Wd)}{n}\log(n)\log\bigg(\frac{2en^3}{\epsilon}\bigg)}\nonumber\\
	&\leq&4ar_n\sqrt{\frac{cL(LW^2+Wd)}{n}}+12ar_n\sqrt{\frac{cL(LW^2+Wd)}{n}\log(n)\log\bigg(\frac{2en^4}{a}\bigg)}\nonumber\\
	&\leq&16ar_n\sqrt{\frac{cL(LW^2+Wd)}{n}\log(n)\log\bigg(\frac{2en^4}{a}\bigg)}\nonumber,
\end{eqnarray}
where we use the fact that $r_n\geq n^{-1/2}$ and $\epsilon\geq an^{-1}$. Using above inequality, the conditions in (\ref{eq:lemma:convergence:truncated:DNN:step2:eq2}) will be satisfied for all $j=1,\ldots, S$, if we choose
\begin{equation}
\bar{r}=\bar{r}_*:=2ar_n+32C_n\sqrt{\frac{cL(LW^2+Wd)}{n}\log(n)\log\bigg(\frac{2en^4}{a}\bigg)}+4C_n\sqrt{\frac{\eta}{n}}.\nonumber
\end{equation}

As a consequence of (\ref{eq:lemma:convergence:truncated:DNN:step2:eq1}) and (\ref{eq:lemma:convergence:truncated:DNN:step2:eq3}), we conclude that $\|T_{C_n}(\widehat{f}_s)-f_{0,s}\|\leq \bar{r}_*$ with probability at least
\begin{align}
	\pr(F_2)=\pr(\cap_{j=1}^SE_j\cap F_1)&\geq 1-\pr(F_1^C)-\sum_{j=1}^S\pr(E_j^c)\nonumber\\
	&\geq 1-\pr(\|\widehat{f}_s)-f_{0,s}\|_n> ar_n)-2Se^{-\eta}\nonumber\\
	&\geq 1-\pr(\|\widehat{f}_s)-f_{0,s}\|_n> ar_n)-3\log_2(2nC_n)e^{-\eta}. \label{eq:lemma:convergence:truncated:DNN:step3:eq2}
\end{align}
If choosing $\eta=2\log(n)$ and using the fact $a>1$, the following holds on event $F_2$:
\begin{eqnarray}
	\|T_{C_n}(\widehat{f}_s)-f_{0,s}\|&\leq& 2ar_n+32C_n\sqrt{\frac{cL(LW^2+Wd)}{n}\log(n)\log(2en^4)}+8C_n\sqrt{\frac{\log(n)}{n}}.\nonumber
\end{eqnarray}
Therefore, it follows that 
\begin{eqnarray}
	&&\lim_{a\to\infty}\lim_{n\to\infty}\pr(\|T_{C_n}(\widehat{f}_s)-f_{0,s}\|>4ar_n)\nonumber\\&\leq&\lim_{a\to\infty}\lim_{n\to\infty}\pr(F_2^C)+\lim_{a\to\infty}\lim_{n\to\infty}\pr(2ar_n>2ar_n)\nonumber\\
	&&+\lim_{a\to\infty}\lim_{n\to\infty}\pr\bigg(32C_n\sqrt{\frac{cL(LW^2+Wd)}{n}\log(n)\log(2en^4)}>ar_n\bigg)\nonumber\\
	&&+\lim_{a\to\infty}\lim_{n\to\infty}\pr\bigg(8C_n\sqrt{\frac{\log(n)}{n}}>ar_n\bigg)\nonumber\\
	&\leq&  \lim_{a\to\infty}\lim_{n\to\infty}\pr(\|T_{C_n}(\widehat{f}_s)-f_{0,s}\|_n> ar_n)+0+0+0=0,\nonumber
\end{eqnarray}
where the last inequality follows from (\ref{eq:lemma:convergence:truncated:DNN:step3:eq2}) and the rate conditions
 \begin{eqnarray}
	C_n=O(\log(n)),\; \log(n)\sqrt{\frac{cL(LW^2+Wd)}{n}\log(n)\log(2en^4)}=O(r_n),\;\log(n)\sqrt{\frac{\log(n)}{n}}=O(r_n).\nonumber
\end{eqnarray}
%
\end{proof}

\begin{lemma}\label{lemma:split:sample:empirical:rate}
Under Assumption \ref{Assumption:AC}, if $C_n \to \infty$, $C_n=O(\log(n))$, $LW=o(\sqrt{n})$, $LWd=o(n)$ and $\log^2(n)=o(nr_n^2)$, then
\begin{eqnarray*}
	\frac{1}{n_a}\sum_{i=1}^{n_a}|\widecheck{f}_s^b(\bfZ_i^a)-f_{0,s}(\bfZ_i^a)|^2=O_P(r_n^2) \quad\textrm{  and } \quad \frac{1}{n_b}\sum_{i=1}^{n_b}|\widecheck{f}_s^a(\bfZ_i^b)-f_{0,s}(\bfZ_i^b)|^2=O_P(r_n^2).
\end{eqnarray*}
\end{lemma}
\begin{proof}
By Lemma \ref{lemma:convergence:truncated:DNN}, we have
\begin{eqnarray*}
	\|\widecheck{f}_s^b-f_{0,s}\|^2=O_P(r_n^2).
\end{eqnarray*}
Now conditioning on observations $\mcD_b=\{(Y_i^b, \bfX_i^b, \bfZ_i^b), i=1,\ldots, n_b\}$ and by Chebyshev's inequality, we have
\begin{eqnarray*}
	&&\pr\bigg(\bigg|\frac{1}{n_a}\sum_{i=1}^{n_a}|\widecheck{f}_s^b(\bfZ_i^a)-f_{0,s}(\bfZ_i^a)|^2-\|\widecheck{f}_s^b-f_{0,s}\|^2\bigg|>\delta r_n^2\bigg| \mcD_b\bigg)\nonumber\\
	&\leq& \frac{1}{\delta^2 n_ar_n^4}\ev\bigg(|\widecheck{f}_s^b(\bfZ_1^a)-f_{0,s}(\bfZ_a^a)|^4\bigg| \mcD_b\bigg)\nonumber\\
	&\leq&\frac{4C_n^2}{\delta^2n_ar_n^4}\ev\bigg(|\widecheck{f}_s^b(\bfZ_1^a)-f_{0,s}(\bfZ_1^a)|^2\bigg| \mcD_b\bigg)\nonumber\\
	&=&\frac{4C_n^2}{\delta^2n_ar_n^4}\|\widecheck{f}_s^b-f_{0,s}\|^2=O_P\bigg(\frac{\log^2(n)}{nr_n^2}\bigg)=o_P(1),
\end{eqnarray*}
where we use the facts that $\|\widecheck{f}_s^b-f_{0,s}\|_\infty\leq 2C_n=O(\log(n))$ and $\log^2(n)=o(nr_n^2)$. 
\end{proof}
\begin{lemma}\label{lemma:split:sample:difference:rate:2}
Under Assumptions \ref{Assumption:AC} and \ref{Assumption:A3}, if $C_n \to \infty$, $C_n=O(\log(n))$, $LW=o(\sqrt{n})$, $LWd=o(n)$ and $\log^2(n)=o(nr_n^2)$, then the following holds:
$$\bigg\|\frac{1}{n_k}\sum_{i=1}^{n_k}\widecheck{\bfX}_i^k{\bfX}_i^{k\top}-\frac{1}{n_k}\sum_{i=1}^{n_k}\bfD_i^k\bfD_i^{k\top}\bigg\|_F=O_P(r_n)\;\;\textrm{ for } k=a,b.$$
\end{lemma}
\begin{proof}
W.L.O.G, we prove $k=a$. By triangle inequality, it follows that
\begin{eqnarray*}
	\bigg\|\frac{1}{n_a}\sum_{i=1}^{n_a}\widecheck{\bfX}_i^a{\bfX}_i^{a\top}-\frac{1}{n_a}\sum_{i=1}^{n_a}\bfD_i^a\bfD_i^{a\top}\bigg\|_F\nonumber&\leq&\bigg\|\frac{1}{n_a}\sum_{i=1}^{n_a}{\bfX}_i^a(\widecheck{\bfX}_i^a-\bfD_i^a)^{\top}\bigg\|_F+\bigg\|\frac{1}{n_a}\sum_{i=1}^{n_a}{\bfD}_i^a({\bfX}_i^a-\bfD_i^a)^\top\bigg\|_F\nonumber\\
	&:=&R_1+R_2,
\end{eqnarray*}
where the definition of $R_1, R_2$ is straight forward in the context. By Lemma \ref{proposition:cauchy:F:norm}, it follows that
\begin{eqnarray}
	R_1\leq \sqrt{\frac{1}{n_a}\sum_{i=1}^{n_a}\|\bfX_i^a\|_2^2\times \frac{1}{n_a}\sum_{i=1}^{n_a}\|\widecheck{\bfX}_i^a-\bfD_i^a\|_2^2}.\nonumber
\end{eqnarray}
Since $\sum_{i=1}^{n_a}\|\bfX_i^a\|_2^2=\sum_{s=1}^q \sum_{i=1}^{n_a}|X_{s,i}^{a}|^2$, Assumption \ref{Assumption:AC}\ref{Assumption:A1} and C.L.T together imply $\sum_{i=1}^{n_a}\|\bfX_i^a\|_2^2/n_a=O_P(1)$. By Lemma \ref{lemma:split:sample:empirical:rate} and the definition of $\widecheck{\bfX}_i^a, \bfD_i^a$,  we have
\begin{eqnarray}
\frac{1}{n_a}\sum_{i=1}^{n_a}\|\widecheck{\bfX}_i^a-\bfD_i^a\|_2^2=\frac{1}{n}\sum_{i=1}^{n_a}\sum_{s=1}^{q}|\widecheck{f}_{s}^b(\bfZ_i^a)-f_{0,s}(\bfZ_i^a)|^2=O_P(r_n^2).\nonumber
\end{eqnarray} 
As a consequence of above, we conclude $R_1=O_P(r_n)$. In the following, we will analyse $R_2$. By straightforward calculation, it is not difficult to show that
\begin{eqnarray*}
	R_2^2\leq \sum_{1\leq s, k\leq q}\bigg|\frac{1}{n_a}\sum_{i=1}^{n_a} f_{0,s}(\bfZ_i^a)[X_{ik}^a-f_{0,k}(\bfZ_i^a)]\bigg|^2.
\end{eqnarray*}
Since $\ev(X_{ik}^a|\bfZ_i^a)=f_{0,k}(\bfZ_i^a)$ for $k=1,\ldots, q$, it follows that $\ev\{f_{0,s}(\bfZ_i^a)[X_{ik}^a-f_{0,k}(\bfZ_i^a)]\}=0$ for $1\leq s, k\leq q$. Therefore, by Assumption \ref{Assumption:AC}\ref{Assumption:A1}, we conclude that
\begin{eqnarray*}
\ev(R_2^2)&\leq& \frac{1}{n_a}\sum_{1\leq s,k\leq q}\ev\{f_{0,s}^2(\bfZ)[X_k-f_{0,s}(\bfZ)]^2\}=O(n^{-1})=O(r_n^2).
\end{eqnarray*}
Combining above, we finish the proof.
\end{proof}

\begin{lemma}\label{lemma:split:sample:difference:rate:3}
Under Assumptions \ref{Assumption:AC} and \ref{Assumption:A3}, if $C_n \to \infty$, $C_n=O(\log(n))$, $LW=o(\sqrt{n})$, $LWd=o(n)$, $\log^2(n)=o(nr_n^2)$ and $\log(n)r_n=o(1)$, then it holds that 
$$\bigg\|\frac{1}{\sqrt{n}}\sum_{i=1}^{n_k}\widecheck{\bfX}_i^k\epsilon_i^k-\frac{1}{\sqrt{n}}\sum_{i=1}^{n_k}\bfD_i^k\epsilon_i^k\bigg\|_2^2=o_P(1) \quad \textrm{ for } k=a, b.$$
\end{lemma}
\begin{proof}
We only prove the case when $k=a$. By conditioning on observations $\mcD_b=\{(Y_i^b, \bfX_i^b, \bfZ_i^b), i=1,\ldots, n_b\}$ and Chebyshev's inequality, we have
\begin{eqnarray*}
	\pr\bigg(\bigg|\frac{1}{n}\sum_{i=1}^{n_a}\bigg(\widecheck{f}_s^b(\bfZ_i^a)-f_{0,s}(\bfZ_i^a)\bigg)\epsilon_i^a\bigg|>\frac{\delta}{\sqrt{n}} \bigg| \mcD_b\bigg)&\leq& \frac{1}{\delta^2n}\sum_{i=1}^{n_a}\ev\bigg(|\widecheck{f}_s^b(\bfZ^a_i)-f_{0,s}(\bfZ^a_i)|^2|\epsilon_i^a|^2\bigg| \mcD_b\bigg)\nonumber\\
	&\leq& S_1+S_2,\nonumber
\end{eqnarray*}
where 
\begin{eqnarray}
	S_1&=&\frac{1}{\delta^2n} \sum_{i=1}^{n_a}\ev\bigg(|\widecheck{f}_s^b(\bfZ_i^a)-f_{0,s}(\bfZ_i^a)|^2|\epsilon_i^a|^2I(|\epsilon_i^a|\leq m_n)\bigg| \mcD_b\bigg)\nonumber\\
	S_2&=&\frac{1}{\delta^2n}\sum_{i=1}^{n_a} \ev\bigg(|\widecheck{f}_s^b(\bfZ_i^a)-f_{0,s}(\bfZ_i^a)|^2|\epsilon_i^a|^2I(|\epsilon_i^a|> m_n)\bigg| \mcD_b\bigg)\nonumber,
\end{eqnarray}
for any diverging sequence $m_n$. By Lemma \ref{lemma:convergence:truncated:DNN}, it follows that 
\begin{eqnarray}
	S_1\leq \frac{m_n^2}{\delta^2n}\sum_{i=1}^{n_a}\ev\bigg(|\widecheck{f}_s^b(\bfZ_i^a)-f_{0,s}(\bfZ_i^a)|^2|\bigg| \mcD_b\bigg)= \frac{n_am_n^2}{\delta^2n}\|\widecheck{f}_s^b-f_{0,s}\|^2=O_P(m_n^2r_n^2).\nonumber
\end{eqnarray}
Moreover, due to truncation and Lemma \ref{lemma:moment:bound}, we have
\begin{eqnarray}
	S_2&\leq&\frac{4C_n^2}{\delta^2n}\sum_{i=1}^{n_a} \ev\bigg(|\epsilon_i^a|^2I(|\epsilon_i^a|> m_n)\bigg| \mcD_b\bigg)\leq \frac{36\kappa_4n_aC_n^2}{\kappa_3^2\delta^2n}e^{-\kappa_3m_n/3}=O_P(C_n^2e^{-\kappa_3m_n/3}).\nonumber
\end{eqnarray}
Since $C_n=O(\log(n))$, if we choose $m_n=3\kappa^{-1}\log(n)$, then $S_1+S_2=o_P(1)$ provided $\log(n)r_n=o(1)$. 
Notice the fact that
\begin{eqnarray*}
	\bigg\|\frac{1}{n}\sum_{i=1}^{n_a}\bigg(\widecheck{\bfX}_i^a-\bfD_i^a\bigg)\epsilon_i^a\bigg\|_2^2&=&\sum_{s=1}^q \bigg|\frac{1}{n}\sum_{i=1}^{n_a}\bigg(\widecheck{f}_s^b(\bfZ_i^a)-f_{0,s}(\bfZ_i^a)\bigg)\epsilon_i^a\bigg|^2,
\end{eqnarray*}
the desired result follows.
\end{proof}

\begin{lemma}\label{lemma:split:sample:DIVE}
Under Assumptions \ref{Assumption:AC} and \ref{Assumption:A3}, if $C_n \to \infty$, $C_n=O(\log(n))$, $LW=o(\sqrt{n})$, $LWd=o(n)$, $\log^2(n)=o(nr_n^2)$ and $\log(n)r_n=o(1)$, then it holds that
\begin{equation*}
	\sqrt{n}(\widecheck{\beta}^{ab}-\beta_0)\to  \textrm{N}(0, \sigma_\epsilon^2 \ev^{-1}(\bfD\bfD^\top)).
\end{equation*}
\end{lemma}
\begin{proof}
By simple calculation, it follows that
\begin{eqnarray*}
	\widecheck{\beta}^{ab}&=&	\bigg(\sum_{i=1}^{n_a}\widecheck{\bfX}_i^a\bfX_i^{a\top}+\sum_{i=1}^{n_b}\widecheck{\bfX}_i^b\bfX_i^{b\top}\bigg)^{-1}\bigg(\sum_{i=1}^{n_a}\widecheck{\bfX}_i^a\bfX_i^{a\top}\widecheck{\beta}^a+\sum_{i=1}^{n_b}\widecheck{\bfX}_i^b\bfX_i^{b\top}\widecheck{\beta}^b\bigg)\nonumber\\
	&=&	\bigg(\sum_{i=1}^{n_a}\widecheck{\bfX}_i^a\bfX_i^{a\top}+\sum_{i=1}^{n_b}\widecheck{\bfX}_i^b\bfX_i^{b\top}\bigg)^{-1}\bigg(\sum_{i=1}^{n_a}\widecheck{\bfX}_i^aY_i^a+\sum_{i=1}^{n_b}\widecheck{\bfX}_i^bY_i^b\bigg)\nonumber\\
	&=&	\bigg(\sum_{i=1}^{n_a}\widecheck{\bfX}_i^a\bfX_i^{a\top}+\sum_{i=1}^{n_b}\widecheck{\bfX}_i^b\bfX_i^{b\top}\bigg)^{-1}\bigg(\sum_{i=1}^{n_a}\widecheck{\bfX}_i^a\bfX_i^{a\top}+\sum_{i=1}^{n_b}\widecheck{\bfX}_i^b\bfX_i^{b\top}\bigg)\beta_0\nonumber\\
	&&+	\bigg(\sum_{i=1}^{n_a}\widecheck{\bfX}_i^a\bfX_i^{a\top}+\sum_{i=1}^{n_b}\widecheck{\bfX}_i^b\bfX_i^{b\top}\bigg)^{-1}\bigg(\sum_{i=1}^{n_a}\widecheck{\bfX}_i^a\epsilon_i^a+\sum_{i=1}^{n_b}\widecheck{\bfX}_i^b\epsilon_i^b\bigg)\nonumber\\
	&:=&R_1+R_2.
\end{eqnarray*}
where the definition of $R_1, R_2$ is straightforward in the context. By Lemma \ref{lemma:split:sample:difference:rate:2}  and Assumption \ref{Assumption:A3}\ref{A3:b}, we can see that $\sum_{i=1}^{n_a}\widecheck{\bfX}_i^a\bfX_i^{a\top}/n+\sum_{i=1}^{n_b}\widecheck{\bfX}_i^b\bfX_i^{b\top}/n$ is asymptotically invertible, and thus $R_1=\beta_0$. Furthermore, one can verify that the  C.L.T holds for $\sum_{i=1}^{n_a}\bfD_i^a\epsilon_i^a/\sqrt{n}+\sum_{i=1}^{n_b}\bfD_i^b\epsilon_i^b/\sqrt{n}$ using Assumption \ref{Assumption:A3}. As a consequence of Slutsky's Theorem, Lemmas \ref{lemma:split:sample:difference:rate:2} and \ref{lemma:split:sample:difference:rate:3}, we can show $\sqrt{n}R_2\to  \textrm{N}(0, \sigma_\epsilon^2 \ev^{-1}(\bfD\bfD^\top))$
\end{proof}
\begin{proof}[\bf{Proof of Theorem \ref{theorem:split:sample:DIVE:deep}}]
By Lemmas \ref{theorem:approximation:deep} and \ref{lemma:rate:of:convergence}, we have
\begin{eqnarray}
	r_n=\log^4(n)\sqrt{\frac{L(LW^2+Wd)}{n}\log(LW^2+Wd)}+\bigg(\frac{LW}{\log(L)\log(W)}\bigg)^{-\frac{2p^*}{t^*}}.\nonumber
\end{eqnarray}
The first result follows by combining above and Lemma \ref{lemma:convergence:truncated:DNN}.

Notice that the conditions in Lemma \ref{lemma:split:sample:DIVE} can be satisfied  by the rate conditions given. So the second result follows.
\end{proof}

\begin{proof}[\bf{Proof of Lemma  \ref{lemma:split:sample:estimation:variance}}]
By Lemma \ref{lemma:split:sample:difference:rate:2} and Theorems \ref{theorem:split:sample:DIVE:deep} or   \ref{theorem:split:sample:DIVE:shallow}, we can prove the desired result using similar argument in the proof of Lemma \ref{lemma:estimation:variance}.
\end{proof}

\begin{proof}[\textbf{Proof of Theorem \ref{theorem:specification:test}}]
Under $H_0$, similar proof of  Lemma \ref{theorem:asymptotic:distribution} leads to 
\begin{eqnarray}
	\sqrt{n}(\widetilde{\beta}-\beta_0)&=&\ev^{-1}(\widetilde{\bfD}\widetilde{\bfD}^\top)\frac{1}{\sqrt{n}}\sum_{i=1}^n\widetilde{\bfD}_i\epsilon_i+o_P(1),\label{eq:lemma:test:H0:eq1}\\
	\sqrt{n}(\widehat{\beta}-\beta_0)&=&\ev^{-1}({\bfD}{\bfD}^\top)\frac{1}{\sqrt{n}}\sum_{i=1}^n{\bfD}_i\epsilon_i+o_P(1).\label{eq:lemma:test:H0:eq2}
\end{eqnarray}
Notice $\bfZ$ contains $\widetilde{\bfZ}$, we show that
\begin{eqnarray*}
	\ev(\epsilon_i^2 \widetilde{\bfD}_i\bfD^\top_i| \widetilde{\bfZ_i})=\sigma_\epsilon^2 \widetilde{\bfD}_i\ev(\bfD^\top_i | \widetilde{\bfZ}_i)=\sigma_\epsilon^2 \widetilde{\bfD}_i \widetilde{\bfD}_i^\top,
\end{eqnarray*}
which further implies
\begin{equation*}
	\ev(\epsilon_i^2 \widetilde{\bfD}_i\bfD^\top_i)=\sigma_\epsilon^2 \ev(\widetilde{\bfD}_i \widetilde{\bfD}_i^\top).
\end{equation*}
By above equation, C.L.T, (\ref{eq:lemma:test:H0:eq1}) and (\ref{eq:lemma:test:H0:eq2}) , we conclude that
\begin{equation*}
	\sqrt{n}\begin{pmatrix}
	\widetilde{\beta}-\beta_0\\
	\widehat{\beta}-\beta_0
	\end{pmatrix}\cid N\bigg(0, \sigma_\epsilon^2\begin{pmatrix}
	\ev^{-1}(\widetilde{\bfD}\widetilde{\bfD}^\top) &  \ev^{-1}({\bfD}{\bfD}^\top)\\
	 \ev^{-1}({\bfD}{\bfD}^\top) & \ev^{-1}({\bfD}{\bfD}^\top)
	\end{pmatrix}\bigg).
\end{equation*}
By delta method, it follows that
\begin{equation}\label{eq:lemma:test:H0:eq3}
	\sqrt{n}(\widehat{\beta}-\widetilde{\beta})\cid N(0, \sigma_\epsilon^2[\ev^{-1}(\widetilde{\bfD}\widetilde{\bfD}^\top)-\ev^{-1}({\bfD}{\bfD}^\top)]).
\end{equation}
As a consequence of Assumption \ref{Assumption:A7}\ref{A7.b} and (\ref{eq:lemma:test:H0:eq3}), we have
\begin{equation*}
	\sigma_\epsilon^{-2} n(\widehat{\beta}-\widetilde{\beta})^\top[\ev^{-1}(\widetilde{\bfD}\widetilde{\bfD}^\top)-\ev^{-1}({\bfD}{\bfD}^\top)]^{-1}(\widehat{\beta}-\widetilde{\beta})\cid \chi^2(q).
\end{equation*}
Under $H_0$ and assumptions given, above equation still holds when the unknown parameters $\sigma_\epsilon$, $\ev(\widetilde{\bfD}\widetilde{\bfD}^\top)$ and $\ev({\bfD}{\bfD}^\top)$ are replaced with their empirical counterparts.

Now we will prove the result under $H_1$. Direct calculation reveals that
\begin{eqnarray*}
	\bigg\|\frac{1}{\sqrt{n}}\sum_{i=1}^n\widehat{\bfX}_i\epsilon_i-\frac{1}{\sqrt{n}}\sum_{i=1}^n{\bfD}_i\epsilon_i\bigg\|_2^2=\sum_{s=1}^q\bigg|\frac{1}{\sqrt{n}}\sum_{i=1}^n\bigg(\widehat{f}_s(\bfZ_i)-f_{0,s}(\bfZ_i)\bigg)\epsilon_i\bigg|^2.
\end{eqnarray*}
By Cauchy's inequality and Lemma \ref{lemma:rate:of:convergence}, it shows that
\begin{eqnarray*}
\bigg|\frac{1}{n}\sum_{i=1}^n\bigg(\widehat{f}_s(\bfZ_i)-f_{0,s}(\bfZ_i)\bigg)\epsilon_i\bigg|^2&\leq& \|\widehat{f}_s-f_{0,s}\|_n^2\times \frac{1}{n}\sum_{i=1}^n\epsilon_i^2=o_P(r_n^2).
\end{eqnarray*}
Therefore by conditions given, it follows that
\begin{eqnarray}
	\bigg\|\frac{1}{\sqrt{n}}\sum_{i=1}^n\widehat{\bfX}_i\epsilon_i\bigg\|_2&\geq& \bigg\|\frac{1}{\sqrt{n}}\sum_{i=1}^n{\bfD}_i\epsilon_i\bigg\|_2-\bigg\|\frac{1}{\sqrt{n}}\sum_{i=1}^n\widehat{\bfX}_i\epsilon_i-\frac{1}{\sqrt{n}}\sum_{i=1}^n{\bfD}_i\epsilon_i\bigg\|_2\nonumber\\
	&\geq& \bigg\|\sqrt{n}\ev(\bfD\epsilon)\bigg\|_2-\bigg\|\frac{1}{\sqrt{n}}\sum_{i=1}^n{\bfD}_i\epsilon_i-\sqrt{n}\ev(\bfD\epsilon)\bigg\|_2-\bigg\|\frac{1}{\sqrt{n}}\sum_{i=1}^n\widehat{\bfX}_i\epsilon_i-\frac{1}{\sqrt{n}}\sum_{i=1}^n{\bfD}_i\epsilon_i\bigg\|_2\nonumber\\
	&\geq& \sqrt{n}\delta_{\textrm{bias}}-o_P(\sqrt{n})-o_P(\sqrt{n}r_n).\label{eq:lemma:difference:rate:3:H1:eq1}
\end{eqnarray}
On the other hand, it yields that
\begin{eqnarray}
	\bigg\|\frac{1}{\sqrt{n}}\sum_{i=1}^n\widehat{\bfX}_i\epsilon_i\bigg\|_2&\leq& \bigg\|\frac{1}{\sqrt{n}}\sum_{i=1}^n{\bfD}_i\epsilon_i\bigg\|_2+\bigg\|\frac{1}{\sqrt{n}}\sum_{i=1}^n\widehat{\bfX}_i\epsilon_i-\frac{1}{\sqrt{n}}\sum_{i=1}^n{\bfD}_i\epsilon_i\bigg\|_2\nonumber\\
	&\leq& \bigg\|\sqrt{n}\ev(\bfD\epsilon)\bigg\|_2+\bigg\|\frac{1}{\sqrt{n}}\sum_{i=1}^n{\bfD}_i\epsilon_i-\sqrt{n}\ev(\bfD\epsilon)\bigg\|_2+\bigg\|\frac{1}{\sqrt{n}}\sum_{i=1}^n\widehat{\bfX}_i\epsilon_i-\frac{1}{\sqrt{n}}\sum_{i=1}^n{\bfD}_i\epsilon_i\bigg\|_2\nonumber\\
	&=& \sqrt{n}\delta_{\textrm{bias}}+o_P(\sqrt{n})+o_P(\sqrt{n}r_n)=O_P(\sqrt{n}).\label{eq:lemma:difference:rate:3:H1:eq2}
\end{eqnarray}
By simple calculation, we have
\begin{eqnarray*}
	\sqrt{n}(\widehat{\beta}-\beta_0)&=&\sqrt{n}\bigg[\bigg(\frac{1}{n}\sum_{i=1}^n\widehat{\bfX}_i{\bfX}_i^\top\bigg)^{-1}\frac{1}{n}\sum_{i=1}^n\widehat{\bfX}_iY_i-\beta_0\bigg]\nonumber\\
	&=&\sqrt{n}\bigg[\bigg(\frac{1}{n}\sum_{i=1}^n\widehat{\bfX}_i{\bfX}_i^\top\bigg)^{-1}\frac{1}{n}\sum_{i=1}^n\widehat{\bfX}_i\bfX_i^\top \beta_0-\beta_0\bigg]+\bigg(\frac{1}{n}\sum_{i=1}^n\widehat{\bfX}_i{\bfX}_i^\top\bigg)^{-1}\frac{1}{\sqrt{n}}\sum_{i=1}^n\widehat{\bfX}_i\epsilon_i\nonumber\\
	&:=&R_1+R_2,
\end{eqnarray*}
Lemma \ref{lemma:difference:rate:2} and Assumption \ref{Assumption:A3}\ref{A3:b} imply that $\sum_{i=1}^n\widehat{\bfX}_i{\bfX}_i^\top/n$ is asymptotically invertible, and thus $R_1=0$. 
Direct examination and (\ref{eq:lemma:difference:rate:3:H1:eq2}) together lead to
\begin{eqnarray*}
	R_2&=&\ev^{-1}(\bfD\bfD^\top)\frac{1}{\sqrt{n}}\sum_{i=1}^n\widehat{\bfX}_i\epsilon_i+\bigg[\bigg(\frac{1}{n}\sum_{i=1}^n\widehat{\bfX}_i{\bfX}_i^\top\bigg)^{-1}-\ev^{-1}(\bfD\bfD^\top)\bigg]\frac{1}{\sqrt{n}}\sum_{i=1}^n\widehat{\bfX}_i\epsilon_i\nonumber\\
	&=& \ev^{-1}(\bfD\bfD^\top)\frac{1}{\sqrt{n}}\sum_{i=1}^n\widehat{\bfX}_i\epsilon_i+o_P(\sqrt{n}),\nonumber
\end{eqnarray*}
where we use Lemma \ref{lemma:difference:rate:2}  and Assumption  \ref{Assumption:A3}\ref{A3:b} that $\ev(\bfD\bfD^\top)$ is positive definite. As a consequence of (\ref{eq:lemma:difference:rate:3:H1:eq1}),  (\ref{eq:lemma:difference:rate:3:H1:eq2}) and above equation, we have
it follows that
\begin{eqnarray*}
	n\|\widehat{\beta}-\beta_0\|_2^2&=& R_2^{\top}R_2\nonumber\\
	&\geq& \lambda_{\min}^{-2}(\ev(\bfD\bfD^\top))\bigg\|\frac{1}{\sqrt{n}}\sum_{i=1}^n\widehat{\bfX}_i\epsilon_i\bigg\|_2^2-o_P(n)\nonumber\\
	&\geq&  \lambda_{\min}^{-2}(\ev(\bfD\bfD^\top))n\delta_{\textrm{bias}}^2-o_P(n).
\end{eqnarray*}
Therefore, with probability approaching one, it follows that
\begin{eqnarray*}
	\sqrt{n}\|\widehat{\beta}-\widetilde{\beta}\|_2&\geq&	\sqrt{n}\|\widehat{\beta}-\beta_0\|_2-	\sqrt{n}\|\widetilde{\beta}-\beta_0\|_2\nonumber\\
	&\geq&  \lambda_{\min}^{-1}(\ev(\bfD\bfD^\top))\sqrt{n}\delta_{\textrm{bias}}-o_P(1),
\end{eqnarray*}
where the last inequality follows from the fact that $\widetilde{\beta}$ is $\sqrt{n}$-consistent under $H_1$. Now direct calculation leads to
\begin{eqnarray*}
	&&\sigma_\epsilon^{-2} n(\widehat{\beta}-\widetilde{\beta})^\top[\ev^{-1}(\widetilde{\bfD}\widetilde{\bfD}^\top)-\ev^{-1}({\bfD}{\bfD}^\top)]^{-1}(\widehat{\beta}-\widetilde{\beta})\nonumber\\
	&\geq& \sigma_\epsilon^{-2} \lambda_{\min}^{-1}\bigg(\ev^{-1}(\widetilde{\bfD}\widetilde{\bfD}^\top)-\ev^{-1}({\bfD}{\bfD}^\top)\bigg)n\|\widehat{\beta}-\widetilde{\beta}\|_2^2 \to\infty \textrm{ in probability}.
\end{eqnarray*}
Finally, Under $H_1$ and assumptions given, we replace unknown parameters $\sigma_\epsilon$, $\ev(\widetilde{\bfD}\widetilde{\bfD}^\top)$ and $\ev({\bfD}{\bfD}^\top)$ with their empirical counterparts, and conclude that $J \to \infty$ in probability.
\end{proof}

\end{document}